\documentclass{amsart}
\usepackage[dvips]{epsfig}
\usepackage{graphics}
\usepackage{latexsym}
\usepackage{amsmath}
\usepackage{amsthm}
\usepackage{amssymb}
\usepackage{color}

\newtheorem{thm}{Theorem}
\newtheorem{lem}[thm]{Lemma}

\newtheorem{cor}[thm]{Corollary}
\newtheorem{defi}[thm]{Definition}
\newtheorem{prop}[thm]{Proposition}
\newtheorem{rk}[thm]{Remark}

\newcommand{\vip}{\vskip.18cm}
\newcommand{\e}{{\varepsilon}}
\newcommand{\rr}{{\mathbb{R}}}
\newcommand{\cc}{{\mathbb{C}}}
\newcommand{\rd}{{\rr^2}}
\newcommand{\nn}{{\mathbb{N}}}
\newcommand{\E}{\mathbb{E}}
\newcommand{\cF}{{\mathcal F}}
\newcommand{\cH}{{\mathcal H}}
\newcommand{\cL}{{\mathcal L}}
\newcommand{\cK}{{\mathcal K}}
\newcommand{\cP}{{\mathcal P}}
\newcommand{\cS}{{\mathcal S}}

\newcommand{\tX}{{\tilde X}}
\newcommand{\tz}{{\tilde z}}

\newcommand{\tB}{{\tilde B}}

\newcommand{\indiq}{{{\bf 1}}}
\newcommand{\intot}{\int_0^t}
\newcommand{\intrd}{{\int_{\rd}}}

\newcommand{\diver}{{\mathrm {div}}}

\newcommand{\Pro}{\mathbb{P}}

\begin{document}

\title[Particle approximation of the Keller-Segel equation]{Stochastic particle approximation of the 
Keller-Segel equation and two-dimensional generalization of Bessel processes}

\author{Nicolas Fournier}

\address{Nicolas Fournier, LPMA, UMR 7599,
UPMC, Case 188, 4 place Jussieu, F-75252 Paris Cedex 5, France.
E-mail: {\tt nicolas.fournier@upmc.fr}}

\author{Benjamin Jourdain}
\address{Benjamin Jourdain, Universit\'e Paris-Est, Cermics (ENPC), INRIA, F-77455 Marne-la-Vall\'ee, France.
E-mail: {\tt jourdain@cermics.enpc.fr }}

\begin{abstract}
The Keller-Segel partial differential equation is a two-dimensional model for chemotaxis. When the total 
mass of the 
initial density is one, it is known to exhibit blow-up in finite time as soon as the sensitivity $\chi$ 
of bacteria to 
the chemo-attractant is larger than $8\pi$. We investigate its approximation by a system of $N$ 
two-dimensional Brownian 
particles interacting through a singular attractive kernel in the drift term.

In the very subcritical case $\chi<2\pi$, the diffusion strongly dominates this singular drift: 
we obtain existence 
for the particle system and prove that its flow of empirical measures converges, as $N\to\infty$
and up to extraction of a subsequence, to a weak solution of the Keller-Segel equation.

We also show that for any $N\ge 2$ and any value of $\chi>0$, pairs of particles do collide with positive
probability: the singularity of the drift is indeed visited. 
Nevertheless, when $\chi<2\pi N$, it is possible to control the drift and obtain existence of the particle system 
until the first time when at least three particles collide. We check that this time is a.s. infinite,
so that global existence holds for the particle system,
if and only if $\chi\leq 8\pi(N-2)/(N-1)$.

Finally, we remark that in the system with $N=2$ particles, the difference between the two positions 
provides a natural 
two-dimensional generalization of Bessel processes, which we study in details.
\end{abstract}

\keywords{Keller-Segel equation, Stochastic particle systems, Propagation of chaos, Bessel processes.}

\subjclass[2010]{65C35, 35K55, 60H10.}

\thanks{We thank Aur\'elien Alfonsi (CERMICS) for numerous discussions about the properties of 
Bessel and squared Bessel processes.}

\maketitle

\section{Introduction and results}
\subsection{The model}

The Keller-Segel equation, introduced by Patlak \cite{p} and Keller and Segel \cite{ks}, 
is a model for chemotaxis. 
It describes the collective motion of cells 
which are attracted 
by a chemical substance and are able to emit it. In its simplest form it is a conservative drift/diffusion 
equation 
for the density $f_t(x)\geq 0$ of cells (particles) with position $x \in \rd$ at time $t\geq 0$ 
coupled with an elliptic 
equation for the chemo-attractant concentration. By making the chemo-attractant concentration explicit in terms of the cell 
density, one obtains the following closed equation:
\begin{equation}\label{ks}
\partial_t f_t(x) + \chi \diver_x ((K\star f_t)(x) f_t(x))  = \Delta_x f_t(x),
\end{equation}
where $\chi>0$ is the sensitivity of cells to the chemo-attractant and where
\begin{equation}\label{K}
K(x)= \frac{-x}{2\pi|x|^2}.
\end{equation}
In the whole paper, we adopt the convention that $K(0)=0$.

\vip

This equation preserves mass and $f_t(x)/\intrd f_0(y)dy$ solves the same equation with $\chi$ replaced by 
$\chi\intrd f_0(y)dy$. We thus may assume without loss of generality that $\intrd f_0(x)dx=1$.

\vip

As is well-known, we have formally $\frac{d}{dt}\int_{\rd}xf_t(x)dx=0$ and 
$\frac{d}{dt}\int_{\rd}|x|^2f_t(x)dx=4-\chi/(2\pi)$. Consequently,
introducing $V_t:=\int_{\rd}|x-\int_{\rd}yf_t(y)dy|^2f_t(x)dx$, it holds that
$\frac{d}{dt}V_t = 4-\chi/(2\pi)$.
Since $V_t$ is nonnegative, some kind of blow-up necessarily occurs before time 
$2\pi V_0/(\chi-8\pi)$ when $\chi$ is larger than the critical value $8\pi$.

\vip

Concerning the well-posedness theory, let us mention J\"ager and Luckhaus \cite{jl},
Blanchet, Dolbeault and Perthame \cite{bdp}, Dolbeault and Schmeiser \cite{ds} and Ega\~na and Mischler
\cite{em}. In particular, the existence of solutions is verified in \cite{jl} (for sufficiently smooth
initial conditions), these solutions being local (in time) if $\chi>0$ is large and global if $\chi>0$ is small.
The existence of a unique {\it strong} (in some precise sense) solution when $\chi < 8\pi$
is shown in \cite{bdp} (existence) and \cite{em} (uniqueness), still for reasonable initial conditions.
The main tool is the {\it free energy} and its relation with its time derivative.
By passing to the limit in a sequence of regularized Keller-Segel equations where the kernel $K$ is 
replaced by a bounded kernel and by introducing defect measures to take into account blow-up, 
the existence of generalized weak solutions to \eqref{ks} is checked in \cite{ds}, even when 
$\chi\ge 8\pi$. The blow-up phenomenon has been investigated by Herrero and Velazquez \cite{hv,v1,v2}.
We refer to Horstmann \cite{h1,h2} and Perthame \cite{pe} for review papers on this model.

\subsection{Weak solutions}
We denote by $\cP(\rd)$ the set of probability measures on $\rd$ and we set 
$\cP_1(\rd)=\{f \in \cP(\rd)\;:\; m_1(f)<\infty\}$, where $m_1(f)=\intrd |x|f(dx)$.
We will use the following notion of weak solutions.

\begin{defi}
Let $\chi>0$ and $T\in (0,\infty]$ be fixed. 
We say that a measurable family $(f_t)_{t\in [0,T)}$ of probability measures on $\rd$ 
is a weak solution to \eqref{ks} on $[0,T)$ if the following conditions hold true:

\vip

(a) for all $t\in [0,T)$, $\int_0^t\intrd\intrd |x-y|^{-1} f_s(dy)f_s(dx) ds<\infty$;

\vip

(b) for all $\phi \in C^2_b(\rd)$, all $t\in [0,T)$, 
\begin{align*}
\intrd \phi(x)f_t(dx) =& \intrd \phi(x)f_0(dx) + \intot \intrd \Delta \phi(x) f_s(dx) ds  \\
&+ \chi \intot\intrd\intrd K(x-y)\cdot \nabla \phi(x) f_s(dy)f_s(dx) ds.
\end{align*}
\end{defi}

Of course, (a) implies that everything makes sense in (b). 
Performing a symmetrization in the last term leads to 
another weak formulation of \eqref{ks} which requires less stringent 
integrability conditions, but which is not 
suitable in view of the following probabilistic interpretation. 

\subsection{The associated trajectories}

We now introduce a natural probabilistic interpretation of the Keller-Segel equation.

\begin{defi}\label{dfnsde}
Let $\chi>0$ and $T\in (0,\infty]$ be fixed. We say that a $\rd$-valued continuous process
$(X_t)_{t \in [0,T)}$ adapted to some filtration $(\cF_t)_{t\in[0,T)}$ 
solves the nonlinear SDE \eqref{nsde} on $[0,T)$ if, for $f_t:=\cL(X_t)$, it holds that 

\vip

(a) $\intot \intrd\intrd |x-y|^{-1} f_s(dy)f_s(dx)ds<\infty$ for all $t\in [0,T)$;

\vip

(b) there is a $2$-dimensional $(\cF_t)_{t\in[0,T)}$-Brownian motion $(B_t)_{t\in [0,T)}$ such that for all $t\in[0,T)$ 
\begin{equation}\label{nsde}
X_t = X_0 + \sqrt 2 B_t + \chi \intot (K\star f_s)(X_s)ds.
\end{equation}
\end{defi}

The main idea is that $(X_t)_{t \in [0,T)}$ represents the time-evolution of the position of a
{\it typical cell}, in an infinite system of cells undergoing the dynamics prescribed by the
Keller-Segel equation. 
The following remark immediately follows from the It\^o formula.

\begin{rk}\label{tlm}
Let $\chi>0$ be fixed.
For $(X_t)_{t \in [0,T)}$ solving the nonlinear SDE \eqref{nsde},
the family $(f_t=\cL(X_t))_{t\in [0,T)}$ is a weak solution to the Keller-Segel equation \eqref{ks}.
\end{rk}

\subsection{The particle system}
We next consider a natural discretization of the nonlinear SDE: we consider
$N\geq 2$ particles (cells) with positions $X^{1,N}_t,\dots,X^{N,N}_t$ solving
(recall that $K(0)=0$)
\begin{equation}\label{ps}
X^{i,N}_t=X^{i}_0 + \sqrt 2 B^i_t + \frac \chi N \sum_{j=1}^N \intot K(X^{i,N}_s-X^{j,N}_s) ds .
\end{equation}
More precisely, a solution on $[0,T)$ is a continuous $(\rd)^N$-valued process
$(X^{i,N}_t)_{i=1,\dots,N,t\in[0,T)}$ adapted to some filtration $(\cF_t)_{t\in[0,T)}$ if 
the initial conditions $X^{i}_0$, $i=1,\dots, N$ are i.i.d. with common law $f_0\in \cP(\rd)$ and if there 
is a $2N$-dimensional $(\cF_t)_{t\in[0,T)}$-Brownian motion $(B^1_t,\dots,B^N_t)_{t\geq 0}$ 
such that \eqref{ps} holds true for all
$t\in [0,T)$ and all $i=1,\dots,N$.

\vip

Of course, such a particle system is not clearly well-defined, due to the singularity of $K$.
Moreover, the singularity {\it is visited}, as shown by the following statement.

\begin{prop}\label{pairco}
For any $N\geq 2$, any $\chi>0$, any $f_0 \in \cP(\rr^2)$, any $t_0>0$ and any 
solution (if it exists) $(X^{i,N}_t)_{i=1,\dots,N,t \in [0,t_0]}$ to \eqref{ps}, 
$$
\Pro \Big(\exists \; s\in[0,t_0], \; \exists \; 1\leq i < j\leq N \; : \; X^{i,N}_s=X^{j,N}_s \Big) >0.
$$
\end{prop}

However, we expect that particles are almost independent (for $N$ large) and look like $N$ copies of the 
solution to the nonlinear SDE, at least in the subcritical case $\chi\in(0,8\pi)$ or locally in time
in the supercritical case $\chi \geq 8\pi$.
This problem seems important, both from a physical point of view, as a step to the rigorous 
derivation of the Keller-Segel equation, and from a numerical point of view.

\subsection{Main results}

We first check that the particle system \eqref{ps} exists when $\chi$
is (very) subcritical.

\begin{thm}\label{pse}
Let  $N \geq 2$ and $\chi\in(0,2\pi N/(N-1))$ be fixed, as well as $f_0 \in \cP_1(\rd)$.
There exists a solution $(X^{i,N}_t)_{t\in [0,\infty),i=1,\dots,N}$ 
to \eqref{ps}.
Furthermore, the family $\{(X^{i,N}_t)_{t\in [0,\infty)},\; i=1,\dots,N\}$ is exchangeable and
for any $\alpha \in ((N-1)\chi/(2\pi N),1)$, any $T>0$,
\begin{align}\label{fund}
\E\Big[\int_0^T |X^{1,N}_s- X^{2,N}_s|^{\alpha-2} ds \Big] 
\leq \frac{(2\sqrt{2}\langle f_0,\sqrt{1+|x|^2}\rangle + 4\sqrt{2} T)^{\alpha}}
{\alpha(2\alpha- (N-1)\chi/(\pi N))}.
\end{align}
\end{thm}

As already mentioned, such a result is not obvious, since $K$ is singular and since
its singularity is visited. The main point is to observe that \eqref{fund} {\it a priori} holds true 
for some $\alpha<1$. This will imply that
that $\E[|K(X^{1,N}_s- X^{2,N}_s)|]$ should be controlled 
(with some margin since $\alpha-2<-1$). This will be sufficient
to prove existence by compactness. The formal computation is as follows: by
the It\^o formula, for $\alpha\in (0,1)$, 
\begin{align}
\label{keycalc}d|&X^{1,N}_t-X^{2,N}_t|^\alpha=\sqrt{2}\alpha|X^{1,N}_t-X^{2,N}_t|^{\alpha-2}(X^{1,N}_t-X^{2,N}_t)
\cdot (dB^1_t-dB^2_t)\\
&+2\alpha^2|X^{1,N}_t-X^{2,N}_t|^{\alpha-2}dt-\frac{\alpha\chi}{\pi N}|X^{1,N}_t-X^{2,N}_t|^{\alpha-2}dt\notag\\
&+\frac{\alpha\chi}{2\pi N}|X^{1,N}_t-X^{2,N}_t|^{\alpha-2}(X^{1,N}_t-X^{2,N}_t)\cdot
\sum_{i=3}^N\Big(\frac{X^{i,N}_t-X^{1,N}_t}{|X^{i,N}_t-X^{1,N}_t|^2}+\frac{X^{2,N}_t-X^{i,N}_t}{|X^{2,N}_t-X^{i,N}_t|^2}
\Big)dt.\notag
\end{align}
The second term in the right-hand side is the It\^o correction due to diffusion, the third term is 
the contribution 
of the interaction between the particles $1$ and $2$ and the last term is the contribution of the 
interactions with between particles $1,2$ and the rest of the system.
By exchangeability and H\"older's inequality, the expectation of the last term in the right-hand 
side is greater
than $-[\alpha(N-2)\chi/(\pi N)]\E[|X^{1,N}_t-X^{2,N}_t|^{\alpha-2}]$. The assumption $\chi<2\pi N/(N-1)$ 
ensures us that the 
It\^o correction dominates the drift contribution. More precisely choosing $\alpha\in(\chi(N-1)/(2\pi N),1)$,
integrating in time and taking expectations, one obtains
$$
\alpha\Big(2\alpha -\frac{\chi(N-1)}{\pi N}\Big)\int_0^t\E[|X^{1,N}_s-X^{2,N}_s|^{\alpha-2}]ds
\leq \E[|X^{1,N}_t-X^{2,N}_t|^{\alpha}].
$$ 
The right-hand side is easily bounded, uniformly in $N$, using the oddness of $K$, whence
\eqref{fund}.
A similar computation was performed by Osada in \cite[Lemma 3.2]{o} for systems of stochastic
vortices.

\vip

Next, and this is the main result of the paper, we show some tightness/consistency as $N\to \infty$
in the (very) subcritical case $\chi<2\pi$. Such a result follows quite easily from the 
the bound \eqref{fund}, which is uniform in $N$ (when $\chi<2\pi$).
We endow $C([0,\infty),\rd)$ with the topology
of uniform convergence on compact time intervals, and $\cP(C([0,\infty),\rd))$ with the associated weak 
convergence topology. Finally, we endow $C([0,\infty),\cP(\rd))$ with the topology
of uniform convergence on compact time intervals associated with the weak convergence topology 
in $\cP(\rd)$.

\begin{thm}\label{tico}
Let $\chi\in(0,2\pi)$ be fixed, as well as $f_0 \in \cP_1(\rd)$.
For each $N\geq 2$, consider the particle system $(X^{i,N}_t)_{t\in [0,\infty),i=1,\dots,N}$ built in 
Theorem \ref{pse},
as well as the empirical measure $\mu^N=N^{-1}\sum_1^N \delta_{(X^{i,N}_t)_{t\in [0,\infty)}}$, 
which a.s. belongs to $\cP(C([0,\infty),\rd))$. For each $t\geq 0$, we also set
$\mu^N_t=N^{-1}\sum_1^N \delta_{X^{i,N}_t}$, which a.s. belongs to $\cP(\rd)$.

\vip

(i) The sequence $\{\mu^N, \; N\geq 2\}$ is tight in $\cP(C([0,\infty),\rd))$.

\vip

(ii) Any (possibly random) weak limit point $\mu$ of  $(\mu^N)_{N\geq 2}$ is a.s. the law of a solution to the 
nonlinear SDE
\eqref{nsde} with initial law $f_0$.

\vip

(iii) In particular, we can find a subsequence $N_k$ such that $(\mu^{N_k}_t)_{t\geq 0}$ goes in 
law, as $k\to \infty$,
in $C([0,\infty),\cP(\rd))$, to some $(\mu_t)_{t\geq 0}$, which is a.s. a weak solution 
to \eqref{ks} starting from $\mu_0=f_0$.
\end{thm}

We are quite satisfied, since this result seems to be the first result
concerning the convergence of the {\it true} particle system (without cutoff) 
to the Keller-Segel equation. 
However, there are two main limitations. First, this result should more or less always hold true
in the subcritical case $\chi \in (0,8\pi)$.
Second, we are not able to prove the convergence,
we have only compactness/consistency. This is due to the fact that we are not able to
prove that our limit point $(\mu_t)_{t\geq 0}$ a.s. belongs to the class of weak solutions
in which uniqueness is known to hold true. Thanks to Ega\~na and Mischler \cite{em}, 
it would suffice to show that $(\mu_t)_{t\geq 0}$ satisfies the {\it free energy dissipation inequality},
which is slightly stronger than the requirement 
$(\mu_t)_{t\geq 0} \in \cap_{p\geq 1} L^1_{loc}([0,\infty),L^p(\rd))$ a.s.
We believe this is a very difficult problem.

\vip

We next prove that, when $\chi<2\pi N$, the particle system always exists until $3$ particles encounter. 
In view of \eqref{keycalc}, this is not surprising. Indeed, the assumption $\chi<2\pi N$ ensures us that the 
It\^o correction still dominates the contribution of the interaction between the particles $1$ and $2$. 
Moreover, it is not very hard to control the last term of \eqref{keycalc} until a $3$-particle collision occurs.

\begin{thm}\label{pse2}
Let $\chi>0$, $N > \max\{2,\chi/(2\pi)\}$ be fixed, as well as $f_0 \in\cP_1(\rd)$
such that $f_0(\{x\})=0$ for all $x\in \rd$. 
There exists a solution 
$(X^{i,N}_t)_{t\in [0,\tau_N),i=1,\dots,N}$ to \eqref{ps}, with
\begin{align*}
\tau_N= \sup_{\ell\geq 1} \;\inf \;\Big\{t\geq 0 \; : \; \exists \; i,j,k \;& \hbox{ pairwise different such that} \\
&|X^{i,N}_{t}-X^{j,N}_{t}|+|X^{j,N}_{t}-X^{k,N}_{t}|+|X^{k,N}_{t}-X^{i,N}_{t}|\leq 1/\ell\Big\}.
\end{align*}
The family $\{(X^{i,N}_t)_{t\in [0,\tau_N)},\; i=1,\dots,N\}$ is exchangeable and
for any $\alpha \in (\chi/(2\pi N),1)$,
\begin{align}\label{fund2}
\hbox{a.s., for all $t\in[0,\tau_N)$,}\quad
\int_0^{t} |X^{1,N}_s- X^{2,N}_s|^{\alpha-2} ds < \infty.
\end{align}
Finally, it holds that 
(i) $\tau_N=\infty$ a.s. if $\chi\leq 8\pi(N-2)/(N-1)$ and
(ii) $\tau_N<\infty$ a.s. if $\chi> 8\pi(N-2)/(N-1)$.
\end{thm}

This result thus in particular shows the global existence for the particle system in the subcritical case
$\chi<8\pi$ for all $N$ large enough. This result seems to be new, as well as our method to check it,
which is quite specific to the model.

\vip

As we will see in the proof of Lemma \ref{lemnoncol3part}-Step 2, for any subsystem  $I\subset\{1,\hdots,N\}$,
the process $R^I_t=2^{-1}\sum_{i\in I}|X^{i,N}_t-\bar X^I_t|^2$, where  $\bar X^{I}_t=|I|^{-1}\sum_{i\in I}X^{i,N}_t$, 
behaves like the square of a Bessel process of dimension
$(|I|-1)(2-(\chi|I|)/(4\pi N))$, when neglecting the contribution of the interaction 
with the other particles. Similar computations for $I=\{1,\dots,N\}$ were performed by
Ha\v{s}kovec and Schmeiser \cite[Page 139]{hs} and Fatkullin \cite[Page 89]{f}.
The condition $\chi\leq 8\pi(N-2)/(N-1)$ implies that for all
$|I| =3,\dots,N$, the dimension $(|I|-1)(2-(\chi|I|)/(4\pi N))$ is greater than $2$,
so that $R^I_t$ does never reach $0$: there are no collisions involving more than two particles.
Of course, the situation is actually much more complicated, since we have to justify that of all $I$, we can indeed
neglect the contribution of the interaction with the other particles.

\begin{rk}
When $\chi \in (0,8\pi)$, we thus show that, for $N$ large enough,
the particles labelled $1,2,3$ do a.s. never encounter. 
To extend the tightness/consistency result of 
Theorem \ref{tico} to some $\chi \in [2\pi,8\pi)$, we believe that a 
quantitative and uniform (in $N$) version
of this fact might be sufficient.
\end{rk}

Finally, we study the case of two particles $N=2$.
The average of the 
two positions is a two-dimensional Brownian motion and their difference $D_t$ follows an autonomous 
SDE with singular drift driven by a Brownian motion,
which can be seen as a natural two-dimensional 
generalization of a Bessel process of dimension $(2-\chi/(4\pi))$.
We show that the equation for $D_t$ (as well as \eqref{ps}) is nonsense when $\chi\geq 4\pi$,
in that there cannot exist global solutions. But this is only a small problem related to
the fact that a Bessel process with dimension $\delta \in (0,1]$ does not solve a {\it classical}
SDE, while its square does (see Revuz and Yor 
\cite[Exercise 1.26 p 451]{reyor}). We thus reformulate the equation in an adequate sense and in such
a way it has a unique solution (in law). We also prove that this solution is stuck at $0$
when $\chi \geq 8\pi$, while it reaches $0$ but escapes instantaneously when $\chi \in (0,8\pi)$.

\subsection{References}
Approximating a large particle system by a partial differential equation (for deriving the PDE)
or a partial differential equation by a large particle system (to compute numerically the solution of the PDE)
is now a classical topic, called {\it propagation of chaos}. This notion was introduced by 
Kac \cite{k} as a step to the rigorous justification of the Boltzmann equation.
When the interaction is regular, the situation is now well-understood, some important contributions
are due to McKean \cite{mc}, Sznitman \cite{s}, M\'el\'eard \cite{m}, Mischler and Mouhot \cite{mm}, etc.
The main idea is that one can generally prove true quantified convergence when the interaction
is Lipschitz continuous and tightness/consistency (and true unquantified convergence if the PDE is known to
have a unique solution) when the coefficients are only continuous. Of course, each PDE is specific and these
are only formal rules.

\vip

The case of singular interactions is much more complicated. In dimension one,
let us mention the works of Bossy-Talay \cite{bt} and Jourdain \cite{j} which concern the viscous Burgers equation
and more general scalar conservation laws (where particles interact through the Heaviside function) and of 
Cepa-L\'epingle \cite{cl} on the very singular
Dyson model.

\vip

A model closely related to the one studied in the present paper is the $2d$-vortex model, that approximates the 
vorticity formulaltion of 
the $2d$-incompressible Navier-Stokes equation. The PDE is the same as \eqref{ks}
and the particle system is the same as \eqref{ps}, replacing everywhere the kernel $K$, see \eqref{K}, by
the Biot and Savart kernel $x^\perp/(2\pi|x|^2)$. This kernel is as singular as $K$, but the interaction is
of course not {\it attractive}, so that the situation is simpler. In particular, there is no blow-up for the PDE
and Osada \cite{o0}
has shown that the particle system is well-posed and that particles do never collide.
Osada \cite{o1,o} has also proved the (true but unquantified) convergence
of the particle system to the solution of the PDE when $\chi$ is sufficiently small (in our notation),
and this limitation has been recently removed in \cite{fhm}. The method developed in 
\cite{fhm} relies on a control of the Fisher information of the law of the particle system provided
by the dissipation of its entropy. It has been applied to a {\it subcritical} Keller-Segel equation by 
Godinho-Quininao \cite{gq}, where $K$ is replaced by $-x/(2\pi|x|^{1+\alpha})$ with some $\alpha \in (0,1)$
and to the Landau equation for moderately soft potentials in \cite{fh}.
Let us finally mention the propagation of chaos results for some particle systems
with deterministic dynamics 
by Marchioro-Pulvirenti \cite{mp} (for the $2d$-Euler equation) by Hauray-Jabin \cite{hj} (for
some singular Vlasov equations) and by Jourdain-Reygner \cite{jr} (for diagonal hyperbolic systems).

\vip

In the above mentioned works, some true convergence is derived. Here, we obtain only a 
tightness/consistency result, but the singularity is really strong and attractive.
Concerning the Keller-Segel equation, we are not aware of
papers dealing with the convergence of the true particle system {\it without any cutoff}.
Stevens \cite{s} studies a physically more convincing particle system with two kinds of particles
(for bacteria and chemo-attractant particles). She proves the convergence 
of this particle system when the kernel $K$ is regularized. In \cite{hs}, Ha\v{s}kovec and Schmeiser
also prove some results for a regularized kernel of the form $K_\e(x)=-x/[|x|(|x|+\e)]$. Finally, 
Godinho-Quininao \cite{gq} study the case where $K$ is replaced by 
$-x/(2\pi|x|^{1+\alpha})$ for some $\alpha \in (0,1)$.

\subsection{Plan of the paper}
In the next section, we prove \eqref{fund} for a regularized particle system. This is the
main tool for the proofs of Theorem \ref{pse} (existence for the particle system when $\chi\in (0,2\pi)$) and
Theorem \ref{tico} (tightness/consistency as $N\to \infty$ when $\chi\in (0,2\pi)$) given in Section \ref{tc}, 
as well as for checking Theorem \ref{pse2} (local or global existence for the particle system in the general case)
in Section \ref{local}.
We establish Proposition \ref{pairco} (positive probability of collisions) in Section \ref{posprobcol}.
Section \ref{222} is devoted to a detailed study of the case $N=2$.
Finally, we quickly and formally discuss in Section \ref{Nps} how to build an relevant
$N$-particle system when $\chi\geq 8\pi(N-2)/(N-1)$ and we explain why it seems to be a difficult problem.

\section{A regularized particle system}\label{rps}

Let $f_0 \in \cP_1(\rd)$, $\chi>0$ and $N\geq 2$ be fixed.
We consider a family $X^{i}_0$, $i=1,\dots, N$ of $f_0$-distributed random variables
and a family $(B^i_t)_{t\geq 0}$,  $i=1,\dots, N$ of $2$-dimensional Brownian motions,
all these random objects being independent. For $\e\in(0,1)$, we define the regularized version 
$K_\e$ of $K$ as
\begin{align}\label{ke}
K_\e(x)= \frac{-x}{2\pi (|x|^2+\e^2)}.
\end{align}
This kernel is globally Lipschitz continuous, so that the particle system 
\begin{equation}\label{psee}
X^{i,N,\e}_t=X^{i}_0 + \sqrt 2 B^i_t + \frac \chi N \sum_{j=1}^N \intot K_\e(X^{i,N,\e}_s-X^{j,N,\e}_s) ds
\end{equation}
is strongly and uniquely well-defined. These particles are furthermore clearly exchangeable.
The following estimates are crucial for our study.

\begin{prop}\label{cru}
For $f_0\in \cP_1(\rd)$, $N\geq 2$ and $\e\in(0,1)$, we consider the
unique solution $(X^{i,N,\e}_t)_{t\geq 0,i=1,\dots,N}$ to \eqref{psee}.

\vip

(i) For all $t\geq 0$,
$\E[(1+|X^{1,N,\e}_{t}|^2)^{1/2}] \leq \langle f_0,\sqrt{1+|x|^2}\rangle + 2 t.$

\vip

(ii) For all $\alpha \in(0,1)$, all $T>0$,
all $\eta\in (0,\e]$,
\begin{align*}
&\Big(2 \alpha - \frac\chi{\pi N}\Big) 
\E\Big[ \int_0^{T} (|X^{1,N,\e}_s - X^{2,N,\e}_s|^2+\eta^2)^{\alpha/2-1} ds\Big]
\leq \frac {(2\sqrt{2}\langle f_0,\sqrt{1+|x|^2}\rangle + 4\sqrt{2}T)^{\alpha}}\alpha \\
&\hskip1.5cm + \frac{(N-2)\chi}{\pi N}
\E\Big[ \int_0^{T} (|X^{1,N,\e}_s - X^{2,N,\e}_s|^2+\eta^2)^{(\alpha-1)/2}
(|X^{1,N,\e}_s - X^{3,N,\e}_s|^2+\e^2)^{-1/2} ds\Big].
\end{align*}
\end{prop}

\begin{proof}
We start with point (i).
Using the It\^o formula  (with $\phi(x)=(1+|x|^2)^{1/2}$ whence $\nabla \phi(x)=(1+|x|^2)^{-1/2} x$ and 
$\Delta \phi(x)=(1+|x|^2)^{-3/2} (2+|x|^2)$ and
taking expectations, we find
\begin{align*}
\E[(1+|X^{1,N,\e}_{t}|^2)^{1/2}] = &\E[(1+|X^1_0|^2)^{1/2}] + 
\E\Big[ \int_0^{t} \frac{2+|X^{1,N,\e}_{t}|^2}{(1+|X^{1,N,\e}_{t}|^2)^{3/2}} ds \Big]\\
& + \frac{\chi}{N} \sum_{j\ne 1} 
\E\Big[ \int_0^{t} \frac{X^{1,N,\e}_s}{(1+|X^{1,N,\e}_s|^2)^{1/2}}\cdot 
K_\e(X^{1,N,\e}_s - X^{j,N,\e}_s ) ds \Big].
\end{align*}
By exchangeability and oddness of $K_\e$, for $j\in\{2,\hdots,N\}$,
\begin{align*}
   \E&\Big[\int_0^{t}\frac{X^{1,N,\e}_s}{(1+|X^{1,N,\e}_s|^2)^{1/2}}\cdot 
K_\e(X^{1,N,\e}_s - X^{j,N,\e}_s )ds\Big]\\&=\frac 1 2 
\E\Big[\int_0^{t}\Big(\frac{X^{1,N,\e}_s}{(1+|X^{1,N,\e}_s|^2)^{1/2}}-\frac{X^{j,N,\e}_s}{(1+|X^{j,N,\e}_s|^2)^{1/2}}\Big) 
\cdot K_\e(X^{1,N,\e}_s - X^{j,N,\e}_s ) ds \Big]
\end{align*}
This last expectation is non-positive since for $x,y\in\rd$, the inequality $|x|^4+|y|^4\geq 2|x|^2|y|^2$ 
implies $(|x|^2(1+|y|^2)^{1/2}+|y|^2(1+|x|^2)^{1/2})^2\geq (|x||y|((1+|y|^2)^{1/2}+(1+|x|^2)^{1/2}))^2$,
whence $(x(1+|y|^2)^{1/2}-y(1+|x|^2)^{1/2})\cdot(x-y)\geq 0$ and thus 
$(x(1+|x|^2)^{-1/2}-y(1+|y|^2)^{-1/2})\cdot(x-y)\geq 0$. Hence
$$
\E[(1+|X^{1,N,\e}_{t}|^2)^{1/2}] = \E[(1+|X^1_0|^2)^{1/2}] + 
\E\Big[ \int_0^{t} \frac{2+|X^{1,N,\e}_{t}|^2}{(1+|X^{1,N,\e}_{t}|^2)^{3/2}} ds \Big]
\leq \E[(1+|X^1_0|^2)^{1/2}] + 2t.
$$
as desired. To prove point (ii), we fix $\alpha\in(0,1)$ and start from
$$
X^{1,N,\e}_t-X^{2,N,\e}_t=X^{1}_0-X^{2}_0 + \sqrt 2 (B^1_t-B^2_t) + \chi R^{12}_t + \chi S^{12}_t,
$$
where $R^{12}_t = N^{-1} \sum_{j=3}^N\intot [ K_\e(X^{1,N,\e}_s - X^{j,N,\e}_s )- K_\e(X^{2,N,\e}_s - X^{j,N,\e}_s )] ds$
and where  $S^{12}_t= 2N^{-1} \intot K_\e(X^{1,N,\e}_s - X^{2,N,\e}_s ) ds$.
We next fix $\eta \in(0,\e]$, introduce $\phi_\eta(x)=(|x|^2+\eta^2)^{\alpha/2}$ and
use the It\^o formula to write
\begin{align*}
\E[\phi_\eta(X^{1,N,\e}_{T}-X^{2,N,\e}_{T})]=&\E[\phi_\eta(X^{1}_0-X^{2}_0)]
+ \E\Big[ \int_0^{T} 2\Delta \phi_\eta( X^{1,N,\e}_s - X^{2,N,\e}_s) ds \Big]\\
& + \chi\E\Big[ \int_0^{T}\nabla \phi_\eta(X^{1,N,\e}_s - X^{2,N,\e}_s) \cdot (dR^{12}_s+dS^{12}_s)  \Big] .
\end{align*}
Since $\eta\in (0,1)$, we have $\phi_\eta(x-y)\leq [\sqrt 2 ((1+|x|^2)^{1/2}+(1+|y|^2)^{1/2}))]^{\alpha}$,
whence $\E[\phi_\eta(X^{1,N,\e}_{T}-X^{2,N,\e}_{T})] \leq (2\sqrt{2}\langle f_0,\sqrt{1+|x|^2}\rangle 
+ 4\sqrt{2}T)^{\alpha}$ by (i). 
Since furthermore
$\E[\phi_\eta(X^{1}_0-X^{2}_0)]\geq 0$,
\begin{align}\label{a1}
&\E\Big[ \int_0^{T} 2\Delta \phi_\eta( X^{1,N,\e}_s - X^{2,N,\e}_s) ds \Big] \\
\leq & (2\sqrt{2}\langle f_0,\sqrt{1+|x|^2}\rangle + 4\sqrt{2} T)^{\alpha} - \chi
\E\Big[ \int_0^{T}\nabla \phi_\eta(X^{1,N,\e}_s - X^{2,N,\e}_s) \cdot (dR^{12}_s+dS^{12}_s) \Big].\nonumber
\end{align}
Using exchangeability and recalling the definition of $R^{12}_t$,
\begin{align*}
&- \E\Big[ \int_0^{T}\nabla \phi_\eta(X^{1,N,\e}_s - X^{2,N,\e}_s) \cdot dR^{12}_s  \Big]\\
\leq & \frac{2(N-2)}{N} \E\Big[ \int_0^{T} |\nabla \phi_\eta(X^{1,N,\e}_s - X^{2,N,\e}_s)|
|K_\e(X^{1,N,\e}_s - X^{3,N,\e}_s)|ds\Big].
\end{align*}
But $\nabla \phi_\eta(x)=\alpha (|x|^2+\eta^2)^{\alpha/2-1}x$, whence
$|\nabla \phi_\eta(x)| \leq \alpha (|x|^2+\eta^2)^{\alpha/2-1/2}$. Furthermore,
$|K_\e(x)|\leq (|x|^2+\e^2)^{-1/2}/(2\pi)$. Hence 
\begin{align}\label{a2}
&- \E\Big[ \int_0^{T}\nabla \phi_\eta(X^{1,N,\e}_s - X^{2,N,\e}_s) \cdot dR^{12}_s  \Big]\\
\leq& \frac{(N-2)\alpha }{\pi N} \E\Big[ \int_0^{T} (|X^{1,N,\e}_s - X^{2,N,\e}_s|^2+\eta^2)^{\alpha/2-1/2}
(|X^{1,N,\e}_s - X^{3,N,\e}_s|^2+\e^2)^{-1/2} ds\Big].\nonumber
\end{align}
Recalling the definition of $S^{12}_t$ and using that $|K_\e(x)|\leq 
(|x|^2+\eta^2)^{-1/2}/(2\pi)$,
\begin{align}\label{a3}
&- \E\Big[ \int_0^{T}\nabla \phi_\eta(X^{1,N,\e}_s - X^{2,N,\e}_s) \cdot dS^{12}_s  \Big]
\leq \frac{\alpha }{\pi N} \E\Big[ \int_0^{T} (|X^{1,N,\e}_s - X^{2,N,\e}_s|^2+\eta^2)^{\alpha/2-1} ds\Big].
\end{align}
Finally, we observe that $\Delta \phi_\eta(x)=\alpha (|x|^2+\eta^2)^{\alpha/2-2} (\alpha |x|^2+2\eta^2)
\geq \alpha^2 (|x|^2+\eta^2)^{\alpha/2-1}$. Inserting this into \eqref{a1} and using \eqref{a2} and \eqref{a3},
we find
\begin{align}\label{ttt}
&2\alpha^2 \E\Big[ \int_0^{T} (|X^{1,N,\e}_s - X^{2,N,\e}_s|^2+\eta^2)^{\alpha/2-1} ds\Big] \\
\leq& (2\sqrt{2}\langle f_0,\sqrt{1+|x|^2}\rangle + 4\sqrt{2} T)^{\alpha} + \frac{\alpha\chi}{\pi N} 
\E\Big[ \int_0^{T} (|X^{1,N,\e}_s - X^{2,N,\e}_s|^2+\eta^2)^{\alpha/2-1} 
ds\Big] \nonumber \\
&+ \frac{(N-2)\alpha\chi}{\pi N} \E\Big[ \int_0^{T} (|X^{1,N,\e}_s - X^{2,N,\e}_s|^2+\eta^2)^{(\alpha-1)/2}
(|X^{1,N,\e}_s - X^{3,N,\e}_s|^2+\e^2)^{-1/2} ds\Big].\nonumber
\end{align}
The conclusion immediately follows.
\end{proof}

\section{Tightness and consistency in the (very) subcritical case}\label{tc}

The aim of this section is to prove Theorems \ref{pse} and \ref{tico}.
In the whole section, $f_0 \in \cP_1(\rd)$ is fixed.
First, we deduce from Proposition \ref{cru}
an estimate saying that in some sense,
particles do not meet too much, uniformly in $N\geq 2$ and $\e\in(0,1)$ when $\chi<2\pi$.

\begin{cor}\label{corcru}
For each $N\geq 2$, each $\chi\in (0,2\pi N/(N-1))$ and each $\e\in(0,1)$, consider the
unique solution $(X^{i,N,\e}_t)_{t\geq 0,i=1,\dots,N}$ to \eqref{psee}.
For all $T>0$ and all $\alpha \in (\chi (N-1)/(2\pi N),1)$,
$$
\E\Big[\int_0^T |X^{1,N,\e}_s- X^{2,N,\e}_s|^{\alpha-2} ds \Big] 
\leq\frac{(2\sqrt{2}\langle f_0,\sqrt{1+|x|^2}\rangle + 4\sqrt{2} T)^{\alpha}}
{\alpha(2\alpha - (N-1)\chi /(\pi N))}.
$$
\end{cor}

\begin{proof}
We thus fix $\alpha \in (\chi(N-1)/(2\pi N),1)$.
By H\"older's inequality and exchangeability, we have, for any $\eta \in (0,\e]$,
$$
\E[(|X^{1,N,\e}_s - X^{2,N,\e}_s|^2+\eta^2)^{(\alpha-1)/2} (|X^{1,N,\e}_s - X^{3,N,\e}_s|^2+\e^2)^{-1/2}]
\leq \E[(|X^{1,N,\e}_s - X^{2,N,\e}_s|^2+\eta^2)^{\alpha/2-1}].
$$
Applying Proposition \ref{cru}-(ii), we thus find
$$
\Big( 2\alpha - \frac {(N-1)\chi}{\pi N}\Big)
\int_0^T \E[(|X^{1,N,\e}_s - X^{2,N,\e}_s|^2+\eta^2)^{\alpha/2-1}]ds
\leq \frac{(2\sqrt{2}\langle f_0,\sqrt{1+|x|^2}\rangle + 4\sqrt{2} T)^{\alpha}}\alpha.
$$
It suffices to let $\eta \searrow 0$ to complete the proof.
\end{proof}

Such an estimate easily implies tightness.

\begin{lem}\label{tight}
For each $N\geq 2$, each $\e\in(0,1)$, consider the unique solution $(X^{i,N,\e}_t)_{t\in [0,\infty),i=1,\dots,N}$
to \eqref{psee}.

\vip

(i) For $N \geq 2$ fixed, if $\chi<2\pi N/(N-1)$, 
the family $\{(X^{1,N,\e}_t)_{t\geq 0}, \; \e\in(0,1)\}$ is tight 
in $C([0,\infty),\rd)$. 

\vip

(ii) If $\chi<2\pi$, the family $\{(X^{1,N,\e}_t)_{t\geq 0}, \; N\geq 2, \; \e\in(0,1)\}$ 
is tight in $C([0,\infty),\rd)$.
\end{lem}

\begin{proof} 
We first prove (ii) and thus suppose that $\chi<2\pi$.
Since $C([0,\infty),\rd)$ is endowed with the topology of the uniform
convergence on compact time intervals, it suffices to prove that for all $T>0$,
$\{(X^{1,N,\e}_t)_{t\in [0,T]},\; N\geq 2,\;\e\in(0,1) \}$
is tight in $C([0,T],\rd)$. Let thus $T>0$ be fixed and 
recall that $X^{1,N,\e}_t=X^{1}_0 + \sqrt 2 B^1_t + J^{1,N,\e}_t$, where
$$
J^{1,N,\e}_t := \frac{\chi}{N} \sum_{j=2}^N \intot K_\e(X^{1,N,\e}_s-X^{j,N,\e}_s)ds.
$$
Observing that the laws of $X^{1}_0$ and $(B^1_t)_{t\in[0,T]}$ do not depend on $N\geq 2$ nor on $\e>0$, 
it suffices to prove that the family $\{(J^{1,N,\e}_t)_{t\in [0,T]},\; N\geq 2,\; \e\in(0,1)\}$, is tight in 
$C([0,T],\rd)$. 
To do so, we fix $\alpha \in (\chi/(2\pi),1)$, and we use H\"older's inequality to write, 
for $0\leq s < t \leq T$,
\begin{align*}
|J^{1,N,\e}_t-J^{1,N,\e}_s| \leq & \frac{\chi}{2\pi N} \sum_{j=2}^N \int_s^t |X^{1,N,\e}_u-X^{j,N,\e}_u|^{-1} du \\
\leq & |t-s|^{(1-\alpha)/(2-\alpha)} \frac{\chi}{2\pi N} \sum_{j=2}^N 
\Big(\int_s^t |X^{1,N,\e}_u-X^{j,N,\e}_u|^{\alpha-2} du\Big)^{1/(2-\alpha)} \\
\leq & Z_T^{N,\e} |t-s|^{\beta},
\end{align*}
where $\beta = (1-\alpha)/(2-\alpha)>0$ and where $Z_T^{N,\e}:=(\chi/(2\pi N)) \sum_{j=2}^N [1+
\int_0^T |X^{1,N,\e}_u-X^{j,N,\e}_u|^{\alpha-2} du]$. Indeed, $x^{1/(2-\alpha)}\leq 1 + x$ because $\alpha \in (0,1)$. 
But we immediately deduce from Corollary \ref{corcru} and exchangeability that
$\sup_{\e\in(0,1),N\geq 2} \E[Z^{N,\e}_T] < \infty$, so that there is a constant $C_T$, not 
depending on $\e\in(0,1)$ nor
on $N\geq2$ such that for all $A>0$, $\Pro(Z^{N,\e}_T > A) \leq C_T/A$. Since $J^{1,N,\e}_0=0$ a.s.,
we conclude that for all $A>0$, for all $N\geq 2$, all $\e\in(0,1)$, $\Pro[(J^{1,N,\e}_t)_{t\in [0,T]} \notin 
\cK_{A} ]
\leq C_T/A$, where $\cK_{A}$ is the set of all functions $\gamma:[0,T]\mapsto \rd$ such that
$\gamma(0)=0$ and for all $0\leq s < t \leq T$, $|\gamma(t)-\gamma(s)|\leq A |t-s|^\beta$.
The Ascoli theorem ensures us that $\cK_{A}$ is a compact subset of $C([0,T],\rd)$ for all $A>0$.
Since $\lim_{A\to\infty} \sup_{N\geq2,\e\in(0,1)} 
\Pro[(J^{1,N,\e}_t)_{t\in [0,T]} \notin \cK_{A} ] = 0$, the proof of (ii) is complete.

\vip

The proof of (i) is exactly the same: the only difference is that $N$ is fixed so that
we can choose $\alpha\in (\chi(N-1)/(2\pi N),1)$.
\end{proof}

We now prove the existence of the particle system without cutoff in the very subcritical case.

\begin{proof}[Proof of Theorem \ref{pse}]
We divide the proof in two steps. Recall that $\chi<2\pi N/(N-1)$.

\vip

{\it Step 1.} For each $\e\in(0,1)$, we consider the unique solution $(X^{i,N,\e}_t)_{t\in [0,\infty),i=1,\dots,N}$
to \eqref{psee}. By Lemma \ref{tight}-(i), we know that the family
$\{(X^{1,N,\e}_t)_{t\geq 0}, \; \e\in(0,1)\}$ is tight in $C([0,\infty),\rd)$. By exchangeability, 
we of course deduce that
$\{(X^{1,N,\e}_t,\dots,X^{N,N,\e}_t)_{t\geq 0}, \; \e\in(0,1)\}$ is tight in $C([0,\infty),(\rd)^N)$ and
consequently that $\{((X^{1,N,\e}_t,B^{1}_t),\dots,(X^{N,N,\e}_t,B^N_t))_{t\geq 0}, \; \e\in(0,1)\}$ is tight in 
$C([0,\infty),(\rd\times\rd)^{N})$
(this last assertion only uses that the law of $(B^1_t,\dots,B^N_t)_{t\geq 0}$ does not depend on $\e$).
It is thus possible to find a decreasing sequence $\e_k \searrow 0$ such that the family 
$((X^{1,N,\e_k}_t,B^{1}_t),\dots,(X^{N,N,\e_k}_t,B^N_t))_{t\geq 0}$ converges in law in
$C([0,\infty),(\rd\times\rd)^{N})$ as $k \to \infty$. By the Skorokhod representation theorem,
we can realize this convergence almost surely. All this shows that we can find, for each $k\geq 1$,
a solution $(\tX^{1,N,\e_k}_t,\dots,\tX^{N,N,\e_k}_t)_{t\geq 0}$ to \eqref{pse}, associated to some Brownian
motions $(\tB^{1,N,\e_k}_t,\dots,\tB^{N,N,\e_k}_t)_{t\geq 0}$, in such a way that the sequence
$((\tX^{1,N,\e_k}_t,\tB^{1,N,\e_k}_t),\dots,(\tX^{N,N,\e_k}_t,\tB^{N,N,\e_k}_t))_{t\geq 0}$ a.s. goes to some
$((X^{1,N}_t,B^{1}_t),\dots,(X^{N,N}_t,B^{N}_t))_{t\geq 0}$ in $C([0,\infty),(\rd\times\rd)^{N})$ as $k \to \infty$.
Let us observe at once that the family $\{(X^{i,N}_t)_{t\geq 0},\; i=1,\dots,N\}$ is exchangeable
and that, by Corollary \ref{corcru} and the Fatou Lemma, for all $T>0$ and 
$\alpha \in (\chi(N-1)/(2\pi N),1)$,
\begin{align}\label{etac}
&\max\Big\{\E\Big[\int_0^T |X^{1,N}_s- X^{2,N}_s|^{\alpha-2} ds \Big],
\sup_k \E\Big[\int_0^T |\tX^{1,N,\e_k}_s- \tX^{2,N,\e_k}_s|^{\alpha-2} ds \Big]\Big\}  \\
\leq &\frac{(2\sqrt{2}\langle f_0,\sqrt{1+|x|^2}\rangle + 
4\sqrt{2} T)^{\alpha}}{\alpha(2\alpha-(N-1)\chi/(\pi N))}.\nonumber
\end{align}

{\it Step 2.} We introduce $\cF_t=\sigma((X_s^{i,N},B^{i}_s)_{i=1,\dots,N,s\in[0,t]})$.
Of course, $(X^{i,N}_t)_{i=1,\dots,N,t\geq 0}$ is $(\cF_t)_{t\geq 0}$-adapted. 
The family $(X^{i,N}_0)_{i=1,\dots,N}$ is of course i.i.d. and $f_0$-distributed 
(because this is the case of $(X^{i,N,\e_k}_0)_{i=1,\dots,N}$ for all $k\geq 1$) and 
$(B^{i}_s)_{i=1,\dots,N,s\in[0,t]}$ is obviously a $2N$-dimensional Brownian motion
(because this is the case of $(B^{i,N,\e_k}_s)_{i=1,\dots,N,s\geq 0}$ for all $k\geq 1$).
We now show that
$(B^{i}_s)_{i=1,\dots,N,s\in[0,t]}$ is a $2N$-dimensional $(\cF_t)_{t\geq 0}$-Brownian motion.
Let thus $t>0$, $\phi:C([0,\infty),(\rd)^N)$  and $\psi:C([0,t],(\rd\times\rd)^N)$ be continuous and bounded.
We have to check that 
\begin{align*}
&\E[\psi((X_s^{i,N},B^{i}_s)_{i=1,\dots,N,s\in[0,t]})\phi((B^{i}_{t+s}-B^{i}_t)_{i=1,\dots,N,s\geq 0}))]\\
&\hskip3cm=\E[\psi((X_s^{i,N},B^{i,N}_s)_{i=1,\dots,N,s\in[0,t]})]E[\phi((B^{i}_{t+s}-B^{i}_t)_{i=1,\dots,N,s\geq 0}))].
\end{align*}
This immediately follows from the fact that for all $k\geq 1$, 
\begin{align*}
&\E[\psi((\tX_s^{i,N,\e_k},\tB^{i,N,\e_k}_s)_{i=1,\dots,N,s\in[0,t]})\phi((\tB^{i,N,\e_k}_{t+s}
-\tB^{i,N,\e_k}_t)_{i=1,\dots,N,s\geq 0}))]\\
&\hskip3cm=\E[\psi((\tX_s^{i,N,\e_k},\tB^{i,N,\e_k}_s)_{i=1,\dots,N,s\in[0,t]})]
E[\phi((\tB^{i,N,\e_k}_{t+s}-\tB^{i,N,\e_k}_t)_{i=1,\dots,N,s\geq 0}))],
\end{align*}
which holds true because $(\tX_t^{i,N,\e_k})_{i=1,\dots,N,t\geq 0}$ is a strong solution to \eqref{psee}
and is thus adapted to the filtration $\cF_t^k=\sigma((X_0^{i},B^{i,N,\e_k}_s)_{i=1,\dots,N,s\geq 0})$.

\vip

{\it Step 3.}
It only remains to check that for each $i\in \{1,\dots,N\}$, each $t\geq 0$,
$X^{i,N}_t=X^{i,N}_0 + \sqrt 2 B^{i}_t + (\chi/N) \sum_{j= 1}^N \intot K(X^{i,N}_s-X^{j,N}_s) ds.$
We of course start from the identity
$\tX^{i,N,\e_k}_t=\tX^{i,N,\e_k}_0 + \sqrt 2 \tB^{i,N,\e_k}_t 
+ (\chi /N) \sum_{j=1}^N \intot K_{\e_k}(\tX^{i,N,\e_k}_s-\tX^{j,N,\e_k}_s) ds$
and pass to the limit as $k\to \infty$, e.g. in probability. 
The only difficulty is to prove that $J^{ij}_k(t)$ tends to $J^{ij}(t)$, where
$$
J^{ij}_k(t)= \intot K_{\e_k}(\tX^{i,N,\e_k}_s-\tX^{j,N,\e_k}_s) ds \quad \hbox{and}\quad
J^{ij}(t)= \intot K(X^{i,N}_s-X^{j,N}_s) ds.
$$
We introduce, for $\eta\in(0,1)$,
$$
J^{ij}_{k,\eta}(t)= \intot K_{\eta}(\tX^{i,N,\e_k}_s-\tX^{j,N,\e_k}_s) ds \quad \hbox{and}\quad
J^{ij}_{\eta}(t)= \intot K_\eta (X^{i,N}_s-X^{j,N}_s) ds.
$$
For $\alpha \in (0,1)$ and $k$ sufficiently large so that
$\e_k<\eta$, we have
\begin{equation}\label{ar}
|K_\eta(x)-K_{\e_k}(x) | + |K_\eta(x)-K(x)| \leq \frac {\eta^2}{\pi |x|(|x|^2+\eta^2)} \leq
\frac{\eta^{1-\alpha} |x|^{\alpha-2}}{\pi}.
\end{equation}
We thus deduce from \eqref{etac} that, for $\alpha \in (\chi (N-1)/(2\pi N),1)$, there exists $C_{\alpha,t}<+\infty$ such that
$$
\E[|J^{ij}(t)-J^{ij}_\eta(t)|] + \limsup_k \E[|J^{ij}_k(t)-J^{ij}_{k,\eta}(t)|] \leq C_{\alpha,t} \eta^{1-\alpha}.
$$
Next, since $K_\eta$ is continuous and bounded and since $(\tX^{i,N,\e_k}_s)_{s\ge 0}$ goes a.s. to 
$(X^{i,N}_s)_{s\ge 0}$, it holds that $J^{ij}_{k,\eta}(t) \to J^{ij}_\eta(t)$ a.s. and in $L^1$ for each $\eta>0$.
Writing
$$
\E[|J^{ij}(t)-J^{ij}_k(t)|] \leq \E[|J^{ij}(t)-J^{ij}_\eta(t)|]  +
\E[|J^{ij}_\eta(t)-J^{ij}_{k,\eta}(t)|]+ \E[|J^{ij}_{k,\eta}(t)-J^{ij}_k(t)|],
$$
we conclude that
$
\limsup_{k\to\infty} \E[|J^{ij}(t)-J^{ij}_k(t)|] \leq C_{\alpha,t} \eta^{1-\alpha}.
$
Since $\eta \in (0,1)$ can be chosen arbitrarily small, 
we deduce that indeed, $J^{ij}_k(t)$ tends to $J^{ij}(t)$ in $L^1$ as $k\to \infty$.
\end{proof}

Following some ideas of \cite[Proposition 6.1]{fhm}, we now give the

\begin{proof}[Proof of Theorem \ref{tico}]
For each $N\geq 2$, we consider the particle system $(X^{i,N}_t)_{t\in [0,\infty),i=1,\dots,N}$ 
built in Theorem \ref{pse},
and we set $\mu^N=N^{-1}\sum_1^N \delta_{(X^{i,N}_t)_{t\in [0,\infty)}}$,
which a.s. belongs to $\cP(C([0,\infty),\rd))$. For each $t\geq 0$, we also set
$\mu^N_t=N^{-1}\sum_1^N \delta_{X^{i,N}_t}$, which a.s. belongs to $\cP(\rd)$.

\vip

{\it Step 1.}
For each $N\geq 2$, $(X^{i,N}_t)_{t\in [0,\infty),i=1,\dots,N}$ has been obtained as a limit point (in law),
of $(X^{i,N,\e}_t)_{t\in [0,\infty),i=1,\dots,N}$ as $\e\to 0$. By 
Lemma \ref{tight}-(ii), the family $\{(X^{1,N}_t)_{t\geq 0}, \; N\geq 2\}$ is thus tight in $C([0,\infty),\rd)$.
As is well-known, see Sznitman \cite[Proposition 2.2]{s}, this implies that
the family $\{\mu^{N}, \; N\geq 2\}$ is tight in $\cP(C([0,\infty),\rd))$ 
(because for each $N\geq2$, the system is exchangeable). This proves point (i).

\vip

{\it Step 2.} We now consider a (not relabelled for notational simplicity) subsequence of $\mu^N$ 
going in law to some $\mu$
and show that $\mu$ a.s. belongs to $\cS:=\{\cL((X_t)_{t\geq 0})\; :\; (X_t)_{t\geq 0}$ solution
to the nonlinear SDE \eqref{nsde} with initial law $f_0\}$, recall Definition \ref{dfnsde}.
This will prove point (ii).

\vip

{\it Step 2.1.} Consider the identity map $\gamma=(\gamma_t)_{t\geq 0}: C([0,\infty),\rr^2) 
\mapsto C([0,\infty),\rr^2)$.
Using the classical theory of martingale problems, we realize that $Q \in \cP(C([0,\infty),\rd)))$ 
belongs to $\cS$ as soon as, setting $Q_t= Q\circ \gamma_t^{-1} \in \cP(\rd)$ for each $t\geq 0$,

\vip

(a) $Q_0=f_0$;

\vip

(b) $\int_0^T \intrd\intrd |x-y|^{-1}Q_s(dy)Q_s(dx)ds <\infty$ for all $T>0$;

\vip

(c) for all $0<t_1<\dots <t_k<s<t$, all
$\varphi_1,\dots,\varphi_k \in C_b(\rr^2)$, all $\varphi \in C^2_b(\rr^2)$,
\begin{align*}
\cF(Q):=&\int\!\!\int Q(d z)Q(d\tz)\varphi_1(z_{t_1})\dots\varphi_k(z_{t_k})\\
& \hskip0.8cm 
\Big[
\varphi(z_t)-\varphi(z_s)-\chi\int_s^t K(z_u-\tz_u)\cdot \nabla \varphi(z_u)du - 
\int_s^t \Delta\varphi(z_u)du\Big]=0. 
\end{align*} 
Indeed, let $(X_t)_{t\geq 0}$ be $Q$-distributed, so that $\cL(X_t)=Q_t$ for all $t\geq 0$. 
Then (a) says that $X_0$ is
$f_0$-distributed, (b) is nothing but the requirement (a) of Definition \ref{dfnsde}, and (c)
tells us that for all $\varphi \in C^2_b(\rr^2)$,
$$
\varphi(X_t)-\varphi(X_s)-\chi \intot \int K(X_s-\tz_s) \cdot \nabla\varphi(X_s)Q(d\tz)ds 
- \intot \Delta\varphi(X_s)ds
$$
is a martingale in the filtration $(\cF_t)_{t\geq 0}$ generated by $(X_t)_{t\geq 0}$.
This classically implies the existence of
a $2$-dimensional $(\cF_t)_{t\geq 0}$-Brownian motion $(B_t)_{t\geq 0}$
such that
$X_t=X_0+\sqrt 2 B_t+ \chi \intot \int K(X_s-\tz_s) Q_s(d\tz)ds$ for all $t\geq 0$.
It finally suffices to observe that for all $x\in \rd$ and all $s\geq0$, $\int K(x-\tz_s) Q(d\tz)=(K\star Q_s)(x)$.

\vip

We now prove that $\mu$ a.s. satisfies these three points. For each
$t\geq 0$, we set $\mu_t= \mu \circ \gamma_t^{-1}$.

\vip

{\it Step 2.2.} 
Since $\mu^N_0$ is the empirical measure of $N$ i.i.d. $f_0$-distributed random variables
and since $\mu_0$ is the limit (in law) of $\mu^N_0$, we obviously have that $\mu_0=f_0$ a.s.,
i.e. $\mu$ a.s. satisfies (a).

\vip

{\it Step 2.3.} Using Corollary \ref{corcru} and exchangeability, we see
that for any $\alpha \in (\chi/(2\pi),1)$, any $T>0$, there is a finite constant $C_{\alpha,T}$ 
such that for all $m>0$, all $N\ge 2$,
\begin{align*}
\E\Big[\int_0^T \intrd \intrd  (m\wedge|x-y|^{\alpha-2}) \mu^N_s(dy)\mu^N_s(dx)ds \Big] 
\leq &\frac{mT}{N}+\frac{1}{N^2}\sum_{i\ne j}\E\Big[\int_0^T |X^{i,N}_s-X^{j,N}_s|^{\alpha-2} ds \Big]\\
\leq &\frac{mT}{N}+C_{\alpha,T}.
\end{align*}
Since $\mu^N$ goes in law to $\mu$, the LHS converges to $\E\Big[\int_0^T \intrd \intrd  
(m\wedge|x-y|^{\alpha-2}) \mu_s(dy)\mu_s(dx)ds \Big]$ as $N\to\infty$. 
Letting $m$ increase to infinity and using the monotone convergence theorem, we find that
$$
\E\Big[\int_0^T \intrd \intrd |x-y|^{\alpha-2}\mu_s(dy)\mu_s(dx)ds \Big]\leq C_{\alpha,T}.
$$
Since $\alpha <1$, this of course implies that $\mu$ a.s. satisfies (b).

\vip

{\it Step 2.4.} From now on, we consider some fixed $\cF: \cP(C([0,\infty),\rr^2))\mapsto \rr$
as in point (c) and we check that $\cF(\mu)=0$ a.s.

\vip

{\it Step 2.4.1.} Here we prove that for all $N\geq 2$,
\begin{equation}\label{scc1}
\E \Big[ (\cF(\mu^N))^2  \Big] \leq \frac {C_\cF} N.
\end{equation}
To this end, we recall that $\varphi \in C^2_b(\rr^2)$ is fixed 
and we apply the It\^o formula to (\ref{ps}):
\begin{align*}
O^N_i(t):=&\varphi(X^{i,N}_t)
-\frac \chi N \sum_{j= 1}^N \intot \nabla\varphi(X^{i,N}_s) \cdot
K(X^{i,N}_s-X^{j,N}_s) ds - \intot \Delta \varphi(X^{i,N}_s)ds\\
=& \varphi(X^i_0) + \sqrt 2 \intot \nabla \varphi(X^{i,N}_s)\cdot dB^i_s.
\end{align*}
By definition of $\cF$ (recall that $K(0)=0$ by convention),
\begin{align*}
\cF(\mu^N)=& \frac 1 N \sum_{i=1}^N \varphi_1(X^{i,N}_{t_1})\dots \varphi_k(X^{i,N}_{t_k}) 
[O^N_i(t)-O^N_i(s)]\\
=&\frac {\sqrt 2}N \sum_{i=1}^N \varphi_1(X^{i,N}_{t_1})\dots \varphi_k(X^{i,N}_{t_k}) 
\int_s^t \nabla \varphi(X^{i,N}_s)\cdot dB^i_s.
\end{align*}
Then (\ref{scc1}) follows from some classical stochastic calculus argument, using that $0<t_1<\dots<t_k<s<t$, that
$\varphi_1,\dots,\varphi_k,\nabla\varphi$ are bounded and that
the Brownian motions $B^1,\dots,B^N$ are independent.

\vip

{\it Step 2.4.2.} Next we introduce, for $\eta\in (0,1)$, 
$\cF_\eta$ defined as $\cF$ with $K$ replaced by the smooth and bounded kernel $K_\eta$, recall \eqref{ke}.
Then one easily checks that $Q\mapsto \cF_\eta(Q)$
is continuous and bounded from $\cP(C([0,\infty),\rr^2))$ to $\rr$. Since
$\mu^N$ goes in law to $\mu$, we deduce that for any $\eta\in(0,1)$,
$$
\E[|\cF_\eta(\mu)|]=\lim_{N} \E[|\cF_\eta(\mu^N)|].
$$

{\it Step 2.4.3.} We now prove that for all $N\geq 2$, all $\eta\in(0,1)$, all $\alpha \in (\chi/(2\pi),1)$,
$$
\E[|\cF(\mu)-\cF_\eta(\mu)|] + \sup_{N\geq 2}\E[|\cF(\mu^N)-\cF_\eta(\mu^N)|] \leq C_{\alpha,\cF} \;\eta^{1-\alpha}.
$$
Using that all the functions (including the derivatives) involved in $\cF$ are bounded and that
we have $|K_\eta(x)-K(x)|\leq \eta^{1-\alpha} |x|^{\alpha-2}\indiq_{\{x\ne 0\}}/\pi$ by \eqref{ar}, 
we get the existence of a finite constant $C_\cF$ such that
\begin{align*}
|\cF(Q)-\cF_\eta(Q)| \leq& C_\cF \; \eta^{1-\alpha} \int \int  \int_0^t |z_u-\tz_u|^{\alpha-2}
\indiq_{\{z_u \ne \tz_u\}}
 du \; Q(d\tz) Q(d z)\\
= & C_\cF \; \eta^{1-\alpha} \intot \intrd \intrd |x-y|^{\alpha-2}\indiq_{\{x\ne y\}} Q_u(dy)Q_u(dx) du.
\end{align*}
The conclusion then follows from Step 2.3. combined with the estimate
\begin{align*}
\E\Big[\int_0^T \intrd \intrd  |x-y|^{\alpha-2}\indiq_{\{x\ne y\}} \mu^N_s(dy)\mu^N_s(dx)ds \Big] 
\leq \frac{1}{N^2}\sum_{i\ne j}\E\Big[\int_0^T |X^{i,N}_s-X^{j,N}_s|^{\alpha-2} ds \Big]\leq C_{\alpha,T}
\end{align*}
deduced from Corollary \ref{corcru} and exchangeability.
\vip

{\it Step 2.4.4.} For any $\eta \in (0,1)$, we write
\begin{align*}
\E[|\cF(\mu)|] \leq & \E[|\cF(\mu)-\cF_\eta(\mu)|]+ \limsup_N| \E[|\cF_\eta(\mu)|]- \E[|\cF_\eta(\mu^N)|]|\\
&+ \limsup_N \E[|\cF_\eta (\mu^N)-\cF(\mu^N)|] + \limsup_N \E[|\cF(\mu^N)|].
\end{align*}
By Steps 2.4.1 and 2.4.2, the fourth and second terms on the right-hand side are zero. We thus deduce from Step 2.4.3
that $\E[|\cF(\mu)|] \leq  C_{\alpha,\cF} \;\eta^{1-\alpha}$. Since $\eta\in(0,1)$ can be chosen
arbitrarily small, we conclude
that $\E[|\cF(\mu)|]=0$, whence $\cF(\mu)=0$ a.s. as desired.

\vip

{\it Step 3.} It only remains to check point (iii).
Consider the (not relabelled) subsequence $\mu^{N}$ going to $\mu$ in $\cP(C([0,\infty),\rd))$ 
as in Step 2. This implies that $(\mu^N_t)_{t\geq 0}$ goes to  $(\mu_t)_{t\geq 0}$ in $C([0,\infty),\cP(\rd))$.
By Step 2, $\mu$ is a.s. the law of a solution to the nonlinear SDE \eqref{nsde}.
As seen in Remark \ref{tlm}, this implies that a.s., $(\mu_t)_{t\geq 0}$ is a weak solution to the Keller-Segel
equation \eqref{ks}.
\end{proof}

\section{(Local) existence for the particle system in the general case}\label{local}

The aim of this section is to prove Theorem \ref{pse2}.
We thus fix $\chi>0$ and $f_0 \in \cP_1(\rd)$. 
We introduce the domain, for $\ell\geq 1$ and $N\geq 2$,
$$
D_\ell^N := \{(x_1,\dots,x_N)\in (\rd)^N \; : \; |x_i-x_j|+|x_j-x_k|+|x_k-x_i|>1/\ell
\; \hbox{for all}\; i,j,k \; \hbox{pairwise different}\},
$$
and we consider the Lipschitz continuous function $\Phi_\ell^N: (\rd)^N \mapsto [0,1]$ defined
by
$$
\Phi_\ell^N(x_1,\dots,x_N)=0\vee \Big(2\ell\min_{i,j,k\;\rm{ distinct}}\{|x_i-x_j|+|x_j-x_k|+|x_k-x_i|\}
-1\Big)\wedge 1,
$$
which satisfies $\indiq_{D_\ell^N}\leq\Phi_\ell^N \leq \indiq_{D_{2\ell}^N}$.
As usual, the random variables 
$X^{i}_0$, $i=1,\dots, N$ are i.i.d. with common law $f_0$ and independent of the i.i.d.
$2$-dimensional Brownian motions $(B^{i}_t)_{t\geq 0}$,  $i=1,\dots, N$.
For $\e\in(0,1)$ and $\ell\geq 1$, the particle system
\begin{equation}\label{pseea}
X^{i,N,\e,\ell}_t=X^{i}_0 + \sqrt 2 B^i_t + \frac \chi N \sum_{j=1}^N \intot K_\e(X^{i,N,\e,\ell}_s-X^{j,N,\e,\ell}_s) 
\Phi^N_\ell((X^{k,N,\e,\ell}_s)_{k=1,\dots,N}) ds
\end{equation}
is strongly well-posed, since $K_\e$ and $\Phi^N_\ell$ are bounded and Lipschitz continuous.
For a fixed $\ell\geq 1$, we can show as in Corollary \ref{corcru} that particles do not meet too often.

\begin{lem}\label{ptiti} Fix $\chi>0$ and consider, for each $N\geq 2$, $\e \in(0,1)$ and $\ell \geq1$,
the unique solution $(X^{i,N,\e,\ell}_t)_{t\geq 0, i=1,\dots,N}$ to \eqref{pseea}.

\vip

(i) For all $t\geq 0$, all $\ell>0$, all $\e\in(0,1)$,
 $\E[(1+|X^{1,N,\e,\ell}_{t}|^2)^{1/2}] \leq \langle f_0,\sqrt{1+|x|^2}\rangle + 2 t$.

\vip

(ii) For all $T>0$, all $\alpha \in (0,1)$, all $\ell>0$, there is a constant $C_{T,\alpha,\ell}$
(depending also on $\chi$ and $f_0$) such that for all $\e\in(0,1)$, all $N > \chi/(2\alpha \pi)$,
\begin{align*}
\E\Big[ \int_0^{T} |X^{1,N,\e,\ell}_s - X^{2,N,\e,\ell}_s|^{\alpha-2} ds\Big]
\leq  1+ C_{T,\alpha,\ell} \Big(2\alpha -\frac \chi {\pi N}\Big)^{(\alpha-2)/(1-\alpha)}.
\end{align*}
\end{lem}

\begin{proof}
First, (i) can be checked exactly as Proposition \ref{cru}-(i), using only that $\Phi^N_\ell$ is nonnegative
and does break the exchangeability.
We now prove (ii) and thus fix $\alpha \in (0,1)$.
Proceeding exactly as in the proof of
Proposition \ref{cru}-(ii), see \eqref{ttt}, we find that for all $\eta \in (0,\e]$,
\begin{align*}
2 \alpha^2 I_{\eta,\alpha,T}^{N,\e,\ell} \leq A_{\alpha,T} + \frac{\chi \alpha }{\pi N} J_{\eta,\alpha,T}^{N,\e,\ell}
+  \frac{(N-2)\chi \alpha}{\pi N} K_{\eta,\alpha,T}^{N,\e,\ell},
\end{align*}
where $A_{\alpha,T}=(2\sqrt{2}\langle f_0,\sqrt{1+|x|^2}\rangle + 4\sqrt{2}T)^{\alpha}$ and where
\begin{align*}
I_{\eta,\alpha,T}^{N,\e,\ell}=&\E[ \int_0^{T} (|X^{1,N,\e,\ell}_s - X^{2,N,\e,\ell}_s|^2+\eta^2)^{\alpha/2-1} ds],\\
J_{\eta,\alpha,T}^{N,\e,\ell}=&\E\Big[ \int_0^{T} \!(|X^{1,N,\e,\ell}_s - X^{2,N,\e,\ell}_s|^2+\eta^2)^{\alpha/2-1}
\Phi^N_\ell((X^{k,N,\e,\ell}_s)_{k=1,\dots,N})ds\Big], \\
K_{\eta,\alpha,T}^{N,\e,\ell}=& \E\Big[ \int_0^{T} (|X^{1,N,\e,\ell}_s - X^{2,N,\e,\ell}_s|^2+\eta^2)^{(\alpha-1)/2}
(|X^{1,N,\e,\ell}_s - X^{3,N,\e,\ell}_s|^2+\e^2)^{-1/2}\\
&\hskip8cm \Phi^N_\ell((X^{k,N,\e,\ell}_s)_{k=1,\dots,N}) ds\Big].
\end{align*}
Since $\Phi^N_\ell \leq 1$, we obviously have $J_{\eta,\alpha,T}^{N,\e,\ell}\leq I_{\eta,\alpha,T}^{N,\e,\ell}$.
We next note that for $u,v>0$, $u^{(\alpha-1)/2}v^{-1/2}\leq(1+u^{-1/2})(1+v^{-1/2})\leq
(1+\max\{u,v\}^{-1/2})(1+u^{-1/2}+v^{-1/2})$ and that,
for $x=(x_1,\dots,x_N) \in (\rd)^N$, $\Phi^N_\ell(x)>0$ implies that $x \in D^N_{2\ell}$, whence
$|x_1-x_2|+|x_1-x_3|+|x_2-x_3|\geq 1/(2\ell)$ and thus $\max\{|x_1-x_2|,|x_1-x_3|\}\geq 1/(8\ell)$.
Consequently, since $\eta\in(0,\e]$,
\begin{align*}
&(|x_1-x_2|^2+\eta^2)^{(\alpha-1)/2}(|x_1-x_3|^2+\e^2)^{-1/2}\Phi^N_\ell(x) \\
\leq & [1+\max\{\eta^2+|x_1-x_2|^2,\eta^2+|x_1-x_3|^2\}^{-1/2}]\\
&\hskip3cm \times [1+ (\eta^2+|x_1-x_2|^2)^{-1}+(\eta^2+|x_1-x_3|^2)^{-1}]
\indiq_{\{x\in D^N_{2\ell}\}}\\
\leq & (1+8\ell) [1+(|x_1-x_2|^2+\eta^2)^{-1/2} + (|x_1-x_3|^2+\eta^2)^{-1/2}].
\end{align*}
This implies that
\begin{align*}
K_{\eta,\alpha,T}^{N,\e,\ell}\leq& (1+8\ell)\E\Big[\int_0^T \Big(1+ (|X^{1,N,\e,\ell}_s - X^{2,N,\e,\ell}_s|^2+\eta^2)^{-1/2}+ \\
&\hskip5cm (|X^{1,N,\e,\ell}_s - X^{3,N,\e,\ell}_s|^2+\eta^2)^{-1/2}\Big)  ds  \Big]\\
\leq & (1+8\ell)T + 2(1+8\ell)\E\Big[\int_0^T (|X^{1,N,\e,\ell}_s - X^{2,N,\e,\ell}_s|^2+\eta^2)^{-1/2} ds \Big]\\
\leq & (1+8\ell)T +2(1+8\ell)T^{(1-\alpha)/(2-\alpha)} [I_{\eta,\alpha,T}^{N,\e,\ell} ]^{1/(2-\alpha)}
\end{align*}
by the H\"older inequality. All in all, we have checked that
$$
\Big( 2\alpha -\frac {\chi}{\pi N}\Big) I_{\eta,\alpha,T}^{N,\e,\ell} \leq B_{\alpha,T,\ell}+
C_{\alpha,T,\ell} [I_{\eta,\alpha,T}^{N,\e,\ell} ]^{1/(2-\alpha)},
$$
where $B_{\alpha,T,\ell}=A_{\alpha,T}/\alpha+ (1+8\ell)T\chi /\pi$ and 
$C_{\alpha,T,\ell}=2(1+8\ell)T^{(1-\alpha)/(2-\alpha)} \chi /\pi$. 

Separating the cases
$I_{\eta,\alpha,T}^{N,\e,\ell}\leq 1$ and $I_{\eta,\alpha,T}^{N,\e,\ell}> 1$, we easily conclude that
$$
I_{\eta,\alpha,T}^{N,\e,\ell}\leq 1 + \Big(B_{\alpha,T,\ell}+C_{\alpha,T,\ell}\Big)^{(2-\alpha)/(1-\alpha)}
\Big(2\alpha -\frac {\chi} {\pi N} \Big)^{(\alpha-2)/(1-\alpha)}.
$$
It finally suffices to let $\eta \searrow 0$ to conclude the proof.
\end{proof}

We now deduce some compactness, still for $\ell$ fixed.

\begin{lem}\label{titi} 
Fix $\chi>0$ and consider, for each $N\geq 2$, $\e \in(0,1)$ and $\ell \geq1$,
the unique solution $(X^{i,N,\e,\ell}_t)_{t\geq 0, i=1,\dots,N}$ to \eqref{pseea}.
For all $\ell\geq 1$, the family $\{(X^{1,N,\e,\ell}_t)_{t\ge 0}, \; N > \max\{2,\chi/(2\pi)\},
\; \e \in (0,1)\}$ is tight in $C([0,\infty),\rd)$.
\end{lem}

\begin{proof}
We fix $\ell\geq 1$ and $T>0$. 
As in the proof of Lemma \ref{tight}, the only difficulty is to prove that
the family $\{(J^{1,N,\e,\ell}_t)_{t\in [0,T],i=1,\dots,N}, \; N \geq N_0,
\; \e \in (0,1) > 0\}$ is tight in $C([0,T],\rd)$, where $N_0=\lfloor \max\{2,\chi/(2\pi) \rfloor +1$ and
$$
J^{1,N,\e,\ell}_t=
\frac \chi N \sum_{j=1}^N \intot K_\e(X^{i,N,\e,\ell}_s-X^{j,N,\e,\ell}_s) 
\Phi^N_\ell((X^{k,N,\e,\ell}_s)_{k=1,\dots,N}) ds.
$$
We consider $\alpha \in (0,1)$ such that
$2\alpha - \chi/(\pi N_0)>0$, so that, by Lemma \ref{ptiti},
\begin{equation}\label{rst}
\sup_{N\geq N_0, \e \in (0,1)} \E\Big[ \int_0^{T} |X^{1,N,\e,\ell}_s - X^{2,N,\e,\ell}_s|^{\alpha-2} ds\Big] <\infty.
\end{equation}
Using that $|\Phi^N_\ell| \leq 1$, we check as in the proof of Lemma \ref{tight} that 
for all $0\leq s<t\leq T$, we have $|J^{1,N,\e,\ell}_t-J^{1,N,\e,\ell}_s| \leq Z^{N,\e,\ell}_T |t-s|^\beta$,
where $\beta = (1-\alpha)/(2-\alpha)$ and where
$$
Z^{N,\e,\ell}_T=\frac{\chi}{2\pi N} \sum_{j=2}^N \Big[1+\int_0^T |X^{1,N,\e,\ell}_s - X^{j,N,\e,\ell}_s|^{\alpha-2} ds \Big].
$$
But \eqref{rst} and exchangeability imply that $\sup_{N\geq N_0, \e \in (0,1)} \E[Z^{N,\e,\ell}_T]<\infty$.
We conclude exactly as in the proof of  Lemma \ref{tight}.
\end{proof}

We now make $\e$ tend to $0$ in the particle system \eqref{pseea}, simultaneously for all $\ell\geq 1$.

\begin{lem}\label{existpseac}
Let $\chi>0$, $N>\max\{2,\chi/(2\pi)\}$ and $f_0 \in \cP_1(\rd)$ be fixed.
There exists, on some probability space endowed with some filtration $(\cF_t)_{t\geq 0}$, a family $(X^i_0)_{i=1,\dots,N}$
of i.i.d. $f_0$-distributed $\cF_0$-measurable random variables, a $2N$-dimensional $(\cF_t)_{t\geq 0}$-Brownian
motion $(B^i_t)_{i=1,\dots,N,t\geq 0}$ and, for each $\ell\geq 1$, an $(\cF_t)_{t\geq 0}$-adapted solution to 
\begin{equation}\label{psea}
X^{i,N,\ell}_t=X^{i}_0 + \sqrt 2 B^i_t + \frac \chi N \sum_{j=1}^N \intot K(X^{i,N,\ell}_s-X^{j,N,\ell}_s) 
\Phi^N_\ell((X^{k,N,\ell}_s)_{k=1,\dots,N}) ds .
\end{equation}
The family $\{ (X^{i,N,\ell}_t)_{t\geq 0, \ell\geq 1} \; i=1,\dots, N\}$ is furthermore exchangeable.
Moreover, for all $\ell\geq 1$, all $t>0$, we have 
$\E[(1+|X^{1,N,\ell}_{t}|^2)^{1/2}] \leq \langle f_0,\sqrt{1+|x|^2}\rangle + 2 t$
and, for all $\alpha \in (\chi/(2\pi N),1)$,
$$
\E\Big[\int_0^t |X^{1,N,\ell}_s-X^{2,N,\ell}_s|^{\alpha-2}] <\infty. 
$$
Finally, we have the following compatibility property:
for all $\ell'\geq\ell\geq 1$, a.s., $(X^{i,N,\ell}_t)_{i=1,\dots,N}=(X^{i,N,\ell'}_t)_{i=1,\dots,N}$ for all
$t \in [0,\tau_N^\ell)$, where
$$
\tau_N^\ell=\inf\{t\geq 0 \; : \; (X^{i,N,\ell}_t)_{i=1,\dots,N} \notin D^N_{\ell}\}.
$$
\end{lem}

\begin{proof} We thus fix $\chi>0$, $N>\max\{2,\chi/(2\pi)\}$ and $f_0 \in \cP_1(\rd)$ 
and divide the proof in several steps.

\vip

{\it Step 1.} We know from Lemma \ref{titi} that for each $\ell\geq 1$, the family
$\{(X^{1,N,\e,\ell}_t)_{t\geq 0},\; \e \in (0,1)\}$ is tight in $C([0,\infty),\rd)$.
By exchangeability, $\{(X^{i,N,\e,\ell}_t)_{t\geq 0,i=1,\dots,N},\; \e \in (0,1)\}$ 
is thus tight in $C([0,\infty),(\rd)^N)$, still for each $\ell\geq 1$. Since $C([0,\infty),(\rd)^N)$ endowed with the topology of uniform convergence on compact subsets of $[0,\infty)$ is a Polish space, by the Prokhorov theorem,
for all $\eta>0$, we can find  a compact subset $\cK^\ell_\eta$ of $C([0,\infty),(\rd)^N))$ such that
$\sup_{\e\in(0,1)} \Pro((X^{i,N,\e,\ell}_t)_{t\geq 0,i=1,\dots,N} \notin \cK^\ell_\eta)\leq \eta 2^{-\ell}$.
We now introduce $\cK_\eta:=\prod_{\ell \geq 1} \cK^\ell_\eta$, which is a compact subset of
$[C([0,\infty),(\rd)^N)]^\nn$ (endowed with the product topology) by Tychonoff's theorem. It holds that 
$$
\sup_{\e\in(0,1)} \Pro(((X^{i,N,\e,\ell}_t)_{t\geq 0,i=1,\dots,N})_{\ell\geq 1} \notin \cK_\eta)\leq 
\sum_{\ell\geq 1}\sup_{\e\in(0,1)} \Pro((X^{i,N,\e,\ell}_t)_{t\geq 0,i=1,\dots,N} \notin \cK^\ell_\eta)
\leq \eta.
$$
Consequently, the family  $\{((X^{i,N,\e,\ell}_t)_{t\geq 0,i=1,\dots,N})_{\ell \geq 1},\; \e \in (0,1)\}$ 
is tight in $[C([0,\infty),(\rd)^N)]^\nn$.
Finally, we conclude that the family 
$$
\{(((X^{i,N,\e,\ell}_t)_{t\geq 0,i=1,\dots,N})_{\ell \geq 1},(B^i_t)_{t\geq 0,i=1,\dots,N}),\; \e \in (0,1)\}
$$ 
is tight in $[C([0,\infty),(\rd)^N)]^\nn\times C([0,\infty),(\rd)^N)$.

\vip

{\it Step 2.} We now use the Skorokhod representation theorem: we can find a sequence $\e_k\searrow 0$
and a sequence $(((\tX^{i,N,\e_k,\ell}_t)_{t\geq 0,i=1,\dots,N})_{\ell \geq 1},(\tB^{i,k}_t)_{t\geq 0,i=1,\dots,N})$
going a.s. in $[C([0,\infty),(\rd)^N)]^\nn\times C([0,\infty,(\rd)^N)$ to some 
$(((X^{i,N,\ell}_t)_{t\geq 0,i=1,\dots,N})_{\ell \geq 1},(B^{i}_t)_{t\geq 0,i=1,\dots,N})$
and such that, for each $\ell\geq 1$,
each $k\geq 1$, $(\tX^{i,N,\e_k,\ell}_t)_{t\geq 0,i=1,\dots,N}$ solves \eqref{pseea} with the Brownian
motions $(\tB^{i,k}_t)_{t\geq 0,i=1,\dots,N}$ and some i.i.d. $f_0$-distributed initial conditions 
$(\tX^{i,N,\e_k}_0)_{i=1,\dots,N}$ (not depending on $\ell\geq1$).
The exchangeability of $\{(X^{i,N,\ell}_t)_{t\geq 0,\ell \geq 1},\;i=1,\dots,N\}$ is inherited from that
of $\{(\tX^{i,N,\e_k,\ell}_t)_{t\geq 0,\ell \geq 1},\;i=1,\dots,N\}$. Next, Lemma \ref{ptiti} and the Fatou Lemma
imply that for all $t\geq 0$, all $\ell\geq 1$,
$$
\max\Big\{\E[(1+|X^{1,N,\ell}_t|^2)^{1/2}], \sup_{k\geq 1} \E[(1+|X^{1,N,\e_k,\ell}_t|^2)^{1/2}]\Big\} 
\leq  \langle f_0,\sqrt{1+|x|^2}\rangle + 2 t
$$
and that, for all $\alpha \in (\chi/(2\pi N),1)$, all $T>0$, all $\ell\geq 1$,
$$
\E\Big[\int_0^T |X^{1,N,\ell}_s -X^{2,N,\ell}_s |^{\alpha-2} ds \Big]
+ \sup_{k\geq1} \Big[\int_0^T |X^{1,N,\e_k,\ell}_s -X^{2,N,\e_k,\ell}_s |^{\alpha-2} ds  \Big]
<\infty.
$$

{\it Step 3.} We introduce $\cF_t=\sigma((X^{i,N,\ell}_s,B^i_s)_{i=1,\dots,N,s\in[0,t]})$, to which
$(X^{i,N,\ell}_t)_{t\geq 0,i=1,\dots,N}$ is of course adapted for each $\ell\geq 1$.
We clearly have $X^{i,N,\ell}_0=X^{i,N,\ell'}_0$ for all $i=1,\dots,N$ and all $\ell,\ell'\geq 1$
(because $\tX^{i,N,\e_k,\ell}_0=\tX^{i,N,\e_k,\ell'}_0$ for all $k\geq 1$,  all $i=1,\dots,N$ and all $\ell,\ell'\geq 1$).
We thus may define $X^i_0:=X^{i,N,\ell}_0$ for all $i=1,\dots,N$, for any value of $\ell$.
The family $(X^{i}_0)_{i=1,\dots,N}$
consists of i.i.d. $f_0$-distributed random variables (because it is the limit of such objects).
Finally, one checks as in the proof of Theorem \ref{pse}-Step 2
$(B^{i}_t)_{t\geq 0,i=1,\dots,N}$ is $2N$-dimensional $(\cF_t)_{t\geq 0}$-Brownian motion.

\vip

{\it Step 4.} It is checked exactly as in the proof of Theorem \ref{pse}-Step 3 that for each $\ell\geq 1$,
$(X^{i,N,\ell}_t)_{t\geq 0, i=1,\dots,N}$ solves \eqref{psea}: it suffices to pass to the limit in probability
as $k\to \infty$ in the equation satisfied by $(\tX^{i,N,\e_k,\ell}_t)_{t\geq 0, i=1,\dots,N}$,
using the estimates proved in Step 2 and that $\Phi^N_\ell$ is continuous.

\vip

{\it Step 5.} It only remains to prove the compatibility property.
We introduce, for $\ell\geq 1$ and $k\geq 1$,
$$
\tau_N^{\ell,k}:=\inf\{t\geq 0 \; : \; (\tX^{i,N,\e_k,\ell}_t)_{i=1,\dots,N} \notin D^N_{\ell}\}
\quad \hbox{and}\quad \tau_N^{\ell}:=\inf\{t\geq 0 \; : \; (X^{i,N,\ell}_t)_{i=1,\dots,N} \notin D^N_{\ell}\}.
$$
Since $(\tX^{i,N,\e_k,\ell}_t)_{t\geq 0,i=1,\dots,N}$ goes a.s. to  $(X^{i,N,\e_k,\ell}_t)_{t\geq 0,i=1,\dots,N}$
in $C([0,\infty),(\rd)^N)$ and since $(D^N_{\ell})^c$ is an closed subset of $(\rd)^N$, we deduce
that $\tau_N^\ell\leq \liminf_{k\to\infty} \tau_N^{\ell,k}$.
But for all $\ell'\geq\ell\geq 1$, we have $(\tX^{i,N,\e_k,\ell}_t)_{i=1,\dots,N}=(\tX^{i,N,\e_k,\ell'}_t)_{i=1,\dots,N}$
on the time interval $[0,\tau_N^{\ell,k}]$ for any $k\geq 1$: this follows from the 
pathwise uniqueness for \eqref{psea} and from the fact that $\Phi^N_\ell=\Phi^N_{\ell'}=1$ on $D^N_\ell$.
Using finally that $(\tX^{i,N,\e_k,\ell}_t,\tX^{i,N,\e_k,\ell'}_t)_{t\geq 0,i=1,\dots,N}$ goes a.s. to  
$(X^{i,N,\ell}_t,X^{i,N,\ell'}_t)_{t\geq 0,i=1,\dots,N}$ in $C([0,\infty),(\rd)^N\times(\rd)^N)$,
we conclude that indeed, $(X^{i,N,\ell}_t)_{i=1,\dots,N}=(X^{i,N,\ell'}_t)_{i=1,\dots,N}$ on $[0,\tau_N^{\ell}]$.
\end{proof}

Finally, we let $\ell$ increase to infinity.

\begin{proof}[Proof of Theorem \ref{pse2}]
We fix $\chi>0$, $N>\max\{2,\chi/(2\pi)\}$ and $f_0\in \cP_1(\rd)$ such that
$f_0(\{x\})=0$ for all $x\in\rd$. We consider the objects built in Lemma \ref{existpseac}: 
the filtration $(\cF_t)_{t\geq 0}$, the $2N$-dimensional $(\cF_t)_{t\geq 0}$-Brownian motion $(B^i_t)_{i=1,\dots,N,t\geq 0}$,
the $(\cF_t)_{t\geq 0}$-adapted solution $(X^{i,N,\ell}_t)_{t\geq 0, i=1,\dots, N}$, for each $\ell\geq 1$, to \eqref{pseea}, 
and associated stopping times $\tau_N^\ell$.
Using the compatibility property, we deduce that $\tau_N^\ell$ is a.s. increasing (as a function of $\ell$)
and we define $\tau_N=\sup_{\ell\geq 1} \tau_N^\ell$.
Still using the compatibility property, we deduce that for all $t\in [0,\tau_N)$,
all $\ell$ such that $\tau_N^\ell \geq t$, all $\ell'\geq \ell$, 
$(X^{i,N,\ell}_t)_{i=1,\dots, N}=(X^{i,N,\ell'}_t)_{i=1,\dots, N}$.
Hence for $t\in [0,\tau_N)$, we can define $(X^{i,N}_t)_{i=1,\dots, N}$ as $(X^{i,N,\ell}_t)_{i=1,\dots, N}$ 
for any choice of $\ell$ such that  $\tau_N^\ell \geq t$.
Since $\Phi_\ell^N((X^{i,N,\ell}_t)_{i=1,\dots, N})=1$ for $t \in [0,\tau_N^\ell]$, by the definitions of $\Phi_N^\ell$
and of $\tau_N^\ell$, we conclude that indeed, $(X^{i,N}_t)_{t\in[0,\tau_N), i=1,\dots, N}$ solves \eqref{ps}
with the Brownian motions $(B^i_t)_{i=1,\dots,N,t\geq 0}$,
and that $\tau_N^\ell=\inf \{t\geq 0 \; : \; (X^{i,N}_t)_{i=1,\dots,N} \notin D^N_{\ell} \}$, so that 
$$
\tau_N = \sup_{\ell\geq 1}\;\inf\; \{t\geq 0 \; : \; (X^{i,N}_t)_{i=1,\dots,N} \notin D^N_{\ell} \}
$$ 
as in the statement.
The exchangeability and $(\cF_t)_{t\geq 0}$-adaptation of the family 
$\{(X^{i,N}_t)_{t\in[0,\tau_N)},\; i=1,\dots, N\}$ is of course
inherited from $\{(X^{i,N,\ell}_t)_{t \geq 0, \ell \geq 1},\; i=1,\dots,N\}$.
We also have a.s., for all $t \in [0,\tau_N)$, all $\alpha \in (\chi/(2\pi N),1)$,
$$
\intot |X^{1,N}_s-X^{2,N}_s|^{\alpha-2} ds = \intot |X^{1,N,\ell}_s-X^{2,N,\ell}_s|^{\alpha-2} ds
$$
as soon as $\ell$ is large enough so that $\tau_N^\ell\geq t$. This last quantity is a.s. finite
by  Lemma \ref{existpseac} again.

\vip

It remains to decide whether $\tau_N$ is finite or infinite.
For $I\subset\{1,\hdots,N\}$ with cardinality 
$|I|\geq 2$ and $t\in[0,\tau_N)$, let $\bar{X}^I_t=|I|^{-1}\sum_{i\in I}X^{i,N}_t$ and 
$R^I_t=2^{-1}\sum_{i\in I}|X^{i,N}_t-\bar{X}^I_t|^2$.

\vip

First assume that $\chi>8\pi(N-2)/(N-1)$. Consider $I_N=\{1,\dots,N\}$.
A direct computation using the It\^o formula 
(see \eqref{dynrI} in the proof of Lemma \ref{lemnoncol3part} below,
the last term obviously vanishes when $I=I_N$) shows that
$(R^{I_N}_t)_{t\in [0,\tau_N)}$ is a squared Bessel process
with dimension $(N-1)(2-\chi/4\pi)<2$, restricted to $[0,\tau_N)$. 
But a squared Bessel process with dimension smaller than $2$
a.s. reaches zero in finite time,
see \cite[page 442]{reyor}. We conclude that on the event $\{\tau_N=\infty\}$, $R^{I_N}$ reaches zero
in finite time, which of course implies that $\tau_N<\infty$.
Thus $\Pro(\tau_N=\infty)=0$ as desired.

\vip

Assume next that  $\chi \leq 8\pi(N-2)/(N-1)$.
Observe that for $(x_1,x_2,x_3)\in(\rd)^3$ and $\bar{x}=(x_1+x_2+x_3)/3$,
\begin{align*}|x_1-x_2|+|x_2-x_3|+|x_3-x_1|&\geq (|x_1-x_2|^2+|x_2-x_3|^2+|x_3-x_1|^2)^{1/2}\\
&=\sqrt{3}(|x_1-\bar{x}|^2+|x_2-\bar{x}|^2+|x_3-\bar{x}|^2)^{1/2}.
\end{align*}
Consequently, for $\ell\geq 1$,
\begin{align*}
\Pro(\tau_N<\infty)&=\Pro(\tau_N<\infty,\tau_N^\ell\leq \tau_N)\\
&=\Pro\Big(\tau_N<+\infty,\min_{i,j,k \; \hbox{{\tiny distinct}}}
\inf_{t\in[0,\tau_N)}(|X^i_t-X^j_t|+|X^j_t-X^k_t|+|X^k_t-X^i_t|)
\leq \frac{1}{\ell}\Big) \\
&\leq \Pro\Big(\tau_N<+\infty,\min_{I\;:\;|I|=3}\inf_{t\in[0,\tau_N)}R^I_t\leq \frac{1}{6\ell^2}\Big).
\end{align*}
This last quantity tends to $0$ as $\ell\to\infty$ thanks to the following Lemma, 
whence $\Pro(\tau_N<\infty)=0$.
\end{proof}

\begin{lem}\label{lemnoncol3part}
Let $N\geq 3$ and $\chi \in (0,8\pi(N-2)/(N-1)]$.
Consider $(X^{i,N}_t)_{t\in[0,\tau_N), i=1,\dots, N}$ built in the previous proof.
For $I\subset\{1,\hdots,N\}$ with cardinality 
$|I|\geq 2$ and $t\in[0,\tau_N)$, let $\bar{X}^I_t=|I|^{-1}\sum_{i\in I}X^{i,N}_t$ and
$R^I_t=2^{-1}\sum_{i\in I}|X^{i,N}_t-\bar{X}^I_t|^2$. 
For all $I\subset\{1,\hdots,N\}$ with $|I|\geq 3$,
\begin{equation*}
\Pro\Big(\tau_N<\infty,\inf_{t\in[0,\tau_N)}R^I_t=0\Big)=0.
\end{equation*}
\end{lem}

\begin{proof}[Proof of Lemma \ref{lemnoncol3part}] 
We divide the proof in several steps.

\vip

{\it Step 1.}
Since the initial conditions $(X^{i,N}_0)_{1\leq i\leq N}$ are independent and $f_0$-distributed
with $f_0(\{x\})=0$ for all $x\in\rd$, they are a.s. pairwise distinct and for all $I\subset\{1,\hdots,N\}$ 
with $|I|\geq 2$, $\Pro(R^I_0>0)=1$.
Also, by definition of $\tau_N$, we have a.s. $R^I_t>0$ for all $t \in [0,\tau_N)$ and all
$|I|\geq 3$.
For all $|I|\geq 3$ and $t\in [0,\tau_N)$, let 
$$
\beta^I_t=\int_0^t\frac{1}{\sqrt{2R^I_s}}\sum_{i\in I}(X^{i,N}_s-\bar{X}^I_s)\cdot dB^i_s.
$$ 
This process can easily be extended into a one-dimensional Brownian motion $(\beta^I_t)_{t\geq 0}$.
In the remaining of the step, we check that for $t\in[0,\tau_N)$,
\begin{equation}
dR^I_t=2\sqrt{R^I_t}d\beta^I_t+(|I|-1)\Big(2-\frac{\chi|I|}{4\pi N}\Big)dt
+\frac{\chi}{N}\sum_{i\in I}\sum_{j\notin I}(X^{i,N}_t-\bar{X}^I_t)\cdot K(X^{i,N}_t-X^{j,N}_t)dt.\label{dynrI}
\end{equation}
We work on $[0,\tau_N)$. Sarting from \eqref{ps} and setting $\bar B^I_t=|I|^{-1}\sum_{i\in I}B^i_t$,
$$
d(X^{i,N}_t-\bar X^I_t)= \sqrt{2}d(B^i-\bar B^I)_t
+\frac\chi N \Big[\sum_{j\ne i}K(X^{i,N}_t-X^{j,N}_t) - |I|^{-1} Z^{I}_t \Big]dt,
$$
where $Z^I_t=\sum_{k\in I} \sum_{j \ne k} K(X^{k,N}_t-X^{j,N}_t)$.
Using the It\^o formula, we thus find
\begin{align*}
d |X^{i,N}_t-\bar X^I_t|^2 =&2\sqrt 2 (X^{i,N}_t-\bar X^I_t)\cdot (dB^i_t-d\bar B^I_t) 
+4\frac{|I|-1}{|I|}dt \\
&+\frac {2\chi} N (X^{i,N}_t-\bar X^I_t)\cdot
\Big[\sum_{j\ne i} K(X^{i,N}_t-X^{j,N}_t) - |I|^{-1} Z^I_t\Big]dt
\end{align*}
and thus
\begin{align*}
dR^I_t=&\sqrt 2 \sum_{i\in I} (X^{i,N}_t-\bar X^I_t)\cdot (dB^i_t-d\bar B^I_t) 
+2(|I|-1)dt\\
&+\frac {\chi} N \sum_{i\in I} (X^{i,N}_t-\bar X^I_t)\cdot
\Big[\sum_{j\ne i} K(X^{i,N}_t-X^{j,N}_t) - |I|^{-1} Z^I_t\Big]dt
\end{align*}
We now observe that $\sum_{i\in I} (X^{i,N}_t-\bar X^I_t)\cdot (dB^i_t-d\bar B^I_t)=
\sum_{i\in I} (X^{i,N}_t-\bar X^I_t)\cdot d B^i_t=\sqrt{2R^I_t}d\beta^I_t$ and that
$\sum_{i\in I} (X^{i,N}_t-\bar X^I_t)Z^I_t=0$, so that
$$
dR^I_t=2 \sqrt{R^I_t}d\beta^I_t +2(|I|-1)dt 
+\frac {\chi} N \sum_{i\in I}\sum_{j\ne i} (X^{i,N}_t-\bar X^I_t)\cdot K(X^{i,N}_t-X^{j,N}_t) dt
$$
To conclude the proof of \eqref{dynrI}, it suffices to note that
$\sum_{i,j\in I,j\ne i} \bar X^I_t\cdot K(X^{i,N}_t-X^{j,N}_t)=0$ and that 
$$
\sum_{i,j\in I,j\ne i} X^{i,N}_t\cdot K(X^{i,N}_t-X^{j,N}_t)= \frac 1 2 \sum_{i,j\in I,j\ne i} 
(X^{i,N}_t-X^{j,N}_t)\cdot K(X^{i,N}_t-X^{j,N}_t)= -\frac{|I|(|I|-1)}{4\pi}.
$$

{\it Step 2: A key observation.} We see in \eqref{dynrI} that, up to the third-term in the right-hand side, 
the process $R^I$ evolves like the square of a Bessel process of dimension 
$(|I|-1)(2-\chi|I|/(4\pi N))$. 
As we will show in a few lines, the condition $\chi \in(0,8\pi(N-2)/(N-1)]$
implies that 
\begin{equation}\label{ko}
\min_{n=3,\dots,N} (n-1)(2-\chi n/(4\pi N)) \geq 2.
\end{equation}
Since by \cite[page 442]{reyor} a squared Bessel process of dimension $\delta\geq 2$ a.s. 
never reaches zero, we expect that indeed, for any $|I|\geq 3$, $R^I$ a.s. never reaches zero.

\vip

To check \eqref{ko}, observe that $\phi(x)=(x-1)(2- \chi x/(4\pi N))$ is concave,
so that we only have to verify that $\phi(3)\geq 2$ and $\phi(N)\geq 2$.
First, $\phi(N)\geq 2$ is equivalent to our condition that $\chi \leq 8\pi(N-2)/(N-1)$.
Next, $\phi(3)\geq 2$ is equivalent to $\chi \leq 4\pi N /3$.
Finally, it is not hard to verify that, $N\geq 3$ being an integer,
we always have $8\pi(N-2)/(N-1)\leq 4\pi N /3$.

\vip

{\it Step 3.} We now prove by backward induction that for all $n=3,\dots,N$,
\begin{equation}\label{pn}
\forall\; I \subset \{1,\dots,N\} \hbox{ with } |I|=n,\;\;\Pro\Big(\tau_N<\infty,\inf_{t\in[0,\tau_N)}R^I_t=0\Big)=0.
\end{equation}

We first observe that \eqref{pn} is clear when $n=N$. Indeed, $|I|=N$ implies that
$I=\{1,\dots,N\}$, so that the third 
term in the right-hand side of \eqref{dynrI} vanishes and $(R^I_t)_{t\in [0,\tau_N)}$
is a (true) squared Bessel process with dimension $(N-1)(2-\chi/(4\pi)) \geq 2$ restricted to
$[0,\tau_N)$. Hence $\inf_{[0,\tau_N)} R^I_t>0$ a.s. on the event $\{\tau_N<\infty\}$.

\vip

We now assume that \eqref{pn} holds for some $n\in\{4,\hdots,N\}$ and check that it also holds for $n-1$. 
We thus consider some fixed $I\subset\{1,\hdots,N\}$ with cardinality $n-1$.
We have to prove that a.s. on $\{\tau_N<\infty\}$, $\inf_{[0,\tau_N)}R^I_t>0$.
For each $j\in\{1,\hdots,N\}\setminus I$, we introduce $I_j=I\cup\{j\}$.

\vip

{\it Step 3.1.} We claim that for each $j\in\{1,\hdots,N\}\setminus I$, 
each $(x_1,\hdots,x_N)\in(\rd)^N$, setting $\bar x^I=(n-1)^{-1}\sum_{i\in I} x_i$ and 
$\bar x^{I_j}=n^{-1}\sum_{i\in I^j} x_i$, 
$$
(2n-3)\min_{k\in I}|x_k-x_j|^2 \geq n\sum_{i\in I_j}|x_i-\bar{x}^{I_j}|^2-3(n-1)\sum_{i\in I}|x_i-\bar{x}^{I}|^2.
$$
We fix $k\in I$ and start from 
$|x_k-x_j|^2=\sum_{i\in I}|x_i-x_j|^2-\sum_{i\in I,i\neq k}|x_i-x_j|^2$ whence,
since $\sum_{i\in I,i\neq k}|x_i-x_j|^2 \leq 2 (n-2)|x_k-x_j|^2 +2\sum_{i\in I,i\neq k}|x_i-x_k|^2$, 
$$
(2n-3)|x_k-x_j|^2 \geq \sum_{i\in I}|x_i-x_j|^2-2\sum_{i\in I,i\neq k}|x_i-x_k|^2.
$$
But one easily checks that $2\max_{k\in I}\sum_{i\in I,i\neq k}|x_i-x_k|^2\leq \sum_{i,k \in I} |x_i-x_k|^2$, whence
$$
(2n-3)\min_{k\in I}|x_k-x_j|^2 \geq  \sum_{i\in I}|x_i-x_j|^2-\sum_{i,k \in I} |x_i-x_k|^2
=\frac 12 \sum_{i,k \in I_j}|x_i-x_k|^2-\frac{3}{2}\sum_{i,k\in I}|x_i-x_k|^2.
$$
The claim then follows from the facts that $\sum_{i,k \in I_j}|x_i-x_k|^2= 2n \sum_{i \in I^j}|x_i-\bar x^{I_j}|^2$
and that $\sum_{i,k\in I}|x_i-x_k|^2=2(n-1)\sum_{i \in I}|x_i-\bar x^{I}|^2 $.

\vip

{\it Step 3.2.} We now fix $a>0$ and $b=a/3$. Step 3.1 implies that when 
$\min_{j\notin I}R^{I_j}_t\geq a$ and $R^I_t\leq b$, we have
\begin{equation}
   \min_{k\in I, j \notin I} |X^{k,N}_t-X^{j,N}_t|^2 \geq \frac{2an}{2n-3}-\frac{6(n-1)b}{2n-3}=\frac{2a}{2n-3},\label{IIj}
\end{equation}
whence
$$
\max_{k\in I, j \notin I} |K(X^{k,N}_t-X^{j,N}_t)| \leq \frac{\sqrt{2n-3}}{2\pi\sqrt{2a}}.
$$
Hence one may bound the third term in the right-hand side of \eqref{dynrI} from below:
\begin{align}\label{minodrift}
\indiq_{\{\min_{j\notin I}R^{I_j}_t\geq a,R^I_t\leq b\}} &\frac \chi N\sum_{i\in I}\sum_{j\notin I}(X^{i,N}_t-\bar{X}^I_t)\cdot
K(X^{i,N}_t-X^{j,N}_t)\\
& \ge - \frac{\chi\sqrt{2n-3} }{2\pi N\sqrt{2a}} \sum_{i\in I}\sum_{j\notin I}|X^{i,N}_t-\bar{X}^I_t|
\geq - c \sqrt{R^I_t}, \nonumber
\end{align}
with $c:=(N+1-n)\chi\sqrt{(2n-3)(n-1)}/(2\pi N \sqrt{a})$.
Let us now define the stopping time 
$$
\sigma_a=\inf \big\{t\in[0,\tau_N):\min_{j\notin I}R^{I_j}_t<a\big\}
$$ 
with convention $\inf\emptyset=\tau_N$ and introduce the process $(R^{I,a}_t)_{t\in [0,\tau_N)}$ 
defined by $R^{I,a}_t=R^I_t$ for $t\in [0,\sigma_a)$ and, when $\sigma_a<\tau_N$,
by being the unique solution, for $t\in [\sigma_a,\tau_N)$, to 
\begin{align*}
R^{I,a}_t=&R^{I}_{\sigma_a}+2\int_{\sigma_a}^t\sqrt{R^{I,a}_s}d\beta^I_s+(|I|-1)
\Big(2-\frac{\chi|I|}{4\pi N}\Big)(t-\sigma_a).
\end{align*}
The existence of a pathwise unique solution to this equation follows from \cite[Theorem 3.5 p 390]{reyor}.
We deduce from \eqref{dynrI} that this process satisfies, for all $t\in [0,\tau_N)$,
\begin{align*}
R^{I,a}_t=&R^I_0+2\int_{0}^t\sqrt{R^{I,a}_s}d\beta^I_s+(|I|-1)
\Big(2-\frac{\chi|I|}{4\pi N}\Big)t\\
&+\frac{\chi}{N}\int_0^t\indiq_{\{s<\sigma_a\}}\sum_{i\in I}\sum_{j\notin I}(X^{i,N}_s-\bar{X}^I_s)
\cdot K(X^{i,N}_s-X^{j,N}_s)ds.
\end{align*}

{\it Step 3.3.} Recall that $a>0$ and $b=a/3$ are fixed and that $c>0$ has been defined in Step 3.2. 
The existence of a solution $(\underline{R}^{I,b}_t)_{t\geq 0}$ such that
$\Pro(\forall t\geq 0,\;\underline{R}^{I,b}_t \in(0,b])=1$ to the SDE 
reflected at the level $b$
\begin{equation}\label{besref}
\begin{cases}
   \underline{R}^{I,b}_t=R^I_0\wedge b+2\int_0^t\sqrt{\underline{R}^{I,b}_s}d\beta^I_s+(|I|-1)
\Big(2-\frac{\chi|I|}{4\pi N}\Big)t-c\int_0^t\sqrt{\underline{R}^{I,b}_s}ds-L_t\\
(L_s)_{s\geq 0}\mbox{ is an adapted increasing process such that }L_0=0\mbox{ and }
\int_0^t(b-\underline{R}^{I,b}_s)dL_s=0
   \end{cases}   
\end{equation}
will be checked in Step 4 using that $|I|\geq 3$.
We take this for granted and show that a.s., for all
$t\in[0,\tau_N)$, $R^{I,a}_t \geq \underline{R}^{I,b}_t$.

\vip

By \cite[Lemma 3.3 p 389]{reyor} with the choice $\rho(u)=|u|$, 
the local time at $0$ of the continuous semimartingale 
$S_t=\underline{R}^{I,b}_t-R^{I,a}_t$ vanishes. Indeed, it suffices that a.s.,
$\intot (\rho(S_s))^{-1}d\langle S,S \rangle_s <\infty$, which follows from the fact that
$d\langle S,S \rangle_s = 4 (\sqrt{\underline{R}^{I,b}_s}-\sqrt{R^{I,a}_s})^2 ds 
\leq 4|\underline{R}^{I,b}_s-R^{I,a}_s |ds=4\rho(S_s)ds$.

\vip

Hence, setting $x^+=\max(x,0)$, one has, by Tanaka's formula, for all $t\in [0,\tau_N)$,
\begin{align*}
(\underline{R}^{I,b}_t-R^{I,a}_t)^+=&
(\underline{R}^{I,b}_0-R^{I,a}_0)^++\int_0^t\indiq_{\{\underline{R}^{I,b}_s>R^{I,a}_s\}}d(\underline{R}^{I,b}_s-R^{I,a}_s).
\end{align*}
Since $\underline{R}^{I,b}_0-R^{I,a}_0=R^I_0\wedge b-R^I_0\leq 0$, we find
\begin{align*}
(\underline{R}^{I,b}_t-&R^{I,a}_t)^+\leq 2\int_0^t\indiq_{\{\underline{R}^{I,b}_s>R^{I,a}_s\}}
\Big(\sqrt{\underline{R}^{I,b}_s}-\sqrt{R^{I,a}_s}\Big)d\beta^I_s - \intot\indiq_{\{\underline{R}^{I,b}_s>R^{I,a}_s\}} dL_s\\
&+\intot \indiq_{\{\underline{R}^{I,b}_s>R^{I,a}_s\}} \Big(-c\sqrt{\underline{R}^{I,b}_s}-
\frac{\chi}{N}\indiq_{\{s<\sigma_a\}}\sum_{i\in I}
\sum_{j\notin I}(X^{i,N}_s-\bar{X}^I_s)\cdot K(X^{i,N}_s-X^{j,N}_s)\Big) ds.
\end{align*}
Since $L$ is an increasing process, the second term on the right-hand side is nonpositive.
The third term on the right-hand side is also nonpositive, because 
$s<\sigma_a$ implies that $R^{I,a}_s=R^I_s$, so that $\underline{R}^{I,b}_s>R^{I,a}_s$ implies that
$R^I_s\leq b$, whence, using \eqref{minodrift} and the definition of $\sigma_a$,
for all $s \in [0,\tau_N)$ such that $\underline{R}^{I,b}_s>R^{I,a}_s$,
$$
-\frac{\chi}{N}\indiq_{\{s<\sigma_a\}}\sum_{i\in I} \sum_{j\notin I}(X^{i,N}_s-\bar{X}^I_s)\cdot K(X^{i,N}_s-X^{j,N}_s)
\leq c \sqrt{R^I_s}=c \sqrt{R^{I,a}_s}<c \sqrt{\underline{R}^{I,b}_s}.
$$
We conclude that a.s., for all $t\in [0,\tau_N)$,
\begin{align}
(\underline{R}^{I,b}_t-R^{I,a}_t)^+\leq& 2\int_0^t \indiq_{\{\underline{R}^{I,b}_s>R^{I,a}_s\}}
\Big(\sqrt{\underline{R}^{I,b}_s}-\sqrt{R^{I,a}_s}\Big)d\beta^I_s.\label{contdif+}
\end{align}
We next introduce $M_t:=\int_0^t \indiq_{\{s<\tau_N,\underline{R}^{I,b}_s>R^{I,a}_s\}}
(\sqrt{\underline{R}^{I,b}_s}-\sqrt{R^{I,a}_s})d\beta^I_s$, which is a true martingale 
(because the integrand is clearly bounded by $\sqrt b$), which is a.s. nonnegative for all times
by \eqref{contdif+} and which starts from $0$: we classically conclude that a.s., 
$M_t$ vanishes for all $t\geq 0$. Coming back to \eqref{contdif+}, we deduce
that $(\underline{R}^{I,b}_t-R^{I,a}_t)^+ \leq 2 M_{t}=0$ a.s. for all $t\in[0,\tau_N)$, which ends the step.

\vip

{\it Step 3.4.} We now conclude the induction. For any $a>0$ and $b=a/3$,
using that $(R^I_t)_{t\in[0,\sigma_a)}=(R^{I,a}_t)_{t\in[0,\sigma_a)}$
and the definition of $\sigma_a$,
\begin{align*}
   \Pro\bigg(\tau_N<\infty,&\inf_{t\in[0,\tau_N)}R^I_t=0\bigg)\leq 
\Pro\Big(\tau_N<\infty,\sigma_a=\tau_N,\inf_{t\in[0,\tau_N)}R^I_t=0\Big)+\Pro\Big(\tau_N<\infty,\sigma_a<\tau_N\Big)\\
&=\Pro\Big(\tau_N<\infty,\sigma_a=\tau_N,\inf_{t\in[0,\tau_N)}R^{I,a}_t=0\Big)
+\Pro\Big(\tau_N<\infty,\min_{j\notin I}\inf_{t\in[0,\tau_N)}R^{I_j}_t \leq a\Big)\\
&\leq \Pro\Big(\tau_N<\infty,\inf_{t\in[0,\tau_N)}\underline{R}^{I,b}_t=0\Big)
+\Pro\Big(\tau_N<\infty,\min_{j\notin I}\inf_{t\in[0,\tau_N)}R^{I_j}_t \leq a\Big).
\end{align*}
Since the continuous process $(\underline{R}^{I,b}_t)_{t\geq 0}$ does not reach $0$, the first 
term in the right-hand side is $0$. We thus can let $a$ tend to $0$ to get
\begin{align*}
\Pro\bigg(\tau_N<\infty,&\inf_{t\in[0,\tau_N)}R^I_t=0\bigg)\leq 
\Pro\Big(\tau_N<\infty,\min_{j\notin I}\inf_{t\in[0,\tau_N)}R^{I_j}_t = 0\Big).
\end{align*}
This last quantity vanishes by our induction assumption.

\vip

{\it Step 4.} To conclude the proof, we still have to check the existence of a solution 
$(\underline{R}^{I,b}_t)_{t\geq 0}$ such $\Pro(\forall t\geq 0,\;\underline{R}^{I,b}_t
\in(0,b])=1$ to \eqref{besref}. For $\ell\geq 1/b$, according to Skorokhod \cite{sk}, existence 
and trajectorial uniqueness hold for the reflected (at $b$) stochastic 
differential equation with Lipschitz drift and diffusion coefficients 
\begin{equation*}
\begin{cases}
\underline{R}^{I,b,\ell}_t=R^I_0\wedge b+2\int_0^t\sqrt{\ell^{-1}\vee\underline{R}^{I,b,\ell}_s}
d\beta^I_s+(|I|-1)\Big(2-\frac{\chi|I|}{4\pi N}\Big)t
-c\int_0^t\sqrt{\ell^{-1}\vee\underline{R}^{I,b,\ell}_s}ds-L^\ell_t\\
\forall t\geq 0,\;\underline{R}^{I,b,\ell}_t\leq b\\
(L^\ell_s)_{s\geq 0}\mbox{ is an adapted increasing process such that }
L^\ell_0=0\mbox{ and }\int_0^t(b-\underline{R}^{I,b,\ell}_s)dL^\ell_s=0.
\end{cases}.   
\end{equation*}
Denoting by $\nu_{\ell}=\inf\{t\geq 0:\underline{R}^{I,b,\ell}_t\leq 1/\ell\}$, 
we deduce from pathwise uniqueness
that for $\ell'\geq \ell$, $(\underline{R}^{I,b,\ell'}_t,L^{\ell'}_t)_{t\in[0,\nu_\ell]}$ and 
$(\underline{R}^{I,b,\ell}_t,L^\ell_t)_{t\in[0,\nu_\ell]}$ coincide and thus that $\ell \mapsto \nu_\ell$ is a.s. increasing.
Setting $\nu_\infty=\sup_{\ell\to\infty}\nu_\ell$, 
we easily deduce the existence of a solution $(\underline{R}^{I,b}_{t},L_t)_{t\in[0,\nu_\infty)}$ to \eqref{besref}
satisfying
$\sup_{t\in[0,\nu_\infty)}\underline{R}^{I,b}_{t}\leq b$
and $\underline{R}^{I,b}_{t}>0$ for all $t\in [0,\nu_\infty)$.
More precisely, $\underline{R}^{I,b}_{t}=\underline{R}^{I,b,\ell}_{t}\geq 1/\ell$ 
for all $\ell$ and all $t\in [0,\nu_\ell)$. It thus only remains to prove that
$\nu_\infty=\infty$ a.s.

\vip

By the Girsanov theorem, 
under the probability measure ${\mathbb Q}$ defined by
$\frac{d\mathbb Q}{d\Pro}|_{\sigma(R^I_0,(\beta^I_s)_{s\in[0,t]})}=\exp(c \beta^I_t / 2- c^2t/8)$ (which is of course a 
true martingale), 
the process $W_t=\beta^I_t-ct/2$ is a one-dimensional Brownian motion. We introduce the equation, satisfied by 
$(\underline{R}^{I,b}_t,L^l_t)$ on the time-interval $[0,\nu_\infty)$,  for a squared Bessel process  
$(\rho_t,\Lambda_t)_{t\geq 0}$ of 
dimension $(|I|-1)(2-\chi|I|/(4\pi N))$ driven by $W$ and 
reflected at the level $b$,
\begin{equation}\label{besqref}
\begin{cases}
\rho_t=R^I_0\wedge b+2\int_0^t\sqrt{\rho_s}dW_s+(|I|-1)\Big(2-\frac{\chi|I|}{4\pi N}\Big)t-\Lambda_t\\
\forall t\geq 0,\;\rho_t\leq b\\
(\Lambda_s)_{s\geq 0}\mbox{ is an adapted increasing process s.t. }\Lambda_0=0\mbox{ and }
\int_0^t(b-\rho_s)d\Lambda_s=0.
\end{cases}.   
\end{equation}
To check global existence for this equation, we set $\underline{\eta}_0=0$ and define, inductively on 
$k\geq 0$, $\rho_t$ to be equal to
\begin{itemize}
\item the squared Bessel process 
$$
R_t=\indiq_{\{k=0\}}R^I_0\wedge b+\indiq_{\{k\ge 1\}}\frac{b}{3}+2\int_{\underline{\eta}_k}^t\sqrt{R_s}dW_s+(|I|-1)
\Big(2-\frac{\chi|I|}{4\pi N}\Big)(t-\underline{\eta}_k)
$$ 
on the time interval $[\underline{\eta}_k,\bar{\eta}_{k+1}]$ where 
$\bar{\eta}_{k+1}=\inf\{t\geq \underline{\eta}_k:R_t\geq 2b/3\}$,
\item the solution to the stochastic differential equation with Lipschitz coefficients 
$$
R^b_t=\frac{2b}{3}+2\int_{\bar{\eta}_{k+1}}^t\sqrt{\frac{b}{3}\vee R^b_s}\,dW_s+(|I|-1)
\Big(2-\frac{\chi|I|}{4\pi N}\Big)(t-\bar{\eta}_{k+1})-\Lambda^b_t,
$$
reflected at $b$  on the time interval $[\bar{\eta}_{k+1},\underline{\eta}_{k+1}]$ where 
$\underline{\eta}_{k+1}=\inf\{t\geq \bar{\eta}_{k+1}:R^b_t\leq b/3\}$.
\end{itemize}
Since, under ${\mathbb Q}$, the delays $(\bar{\eta}_{k+1}-\bar{\eta}_k)_{k\geq 1}$ are i.i.d. and positive, 
${\mathbb Q}$-a.s., $\bar{\eta}_k$ goes to $\infty$ with $k$ by the law of large numbers and $\rho_t$ is 
defined for $t\in[0,+\infty)$. It is easily checked that the process $(\Lambda)_{t\geq 0}$ defined by the 
first equality in \eqref{besqref} also satisfies the last one.

Reasoning like in the comparison between $R^{I,a}$ and $\underline{R}^{I,b}$ performed in Step 3.3, we check 
that the first component of any of two solutions to \eqref{besqref} is above the other one so that the first 
components coincide. 

We deduce that $\underline{R}^{I,b}_{t}$ and $\rho_t$ coincide for $t\in[0,\nu_\infty)$. With the definition of 
$\nu_\infty$ and the continuity of $\rho$, this implies that 
$\{\nu_\infty\leq t\}\subset\{\exists s\in[0,t]:\rho_s=0\}$. Since $(\rho_t)_{t\geq 0}$ always evolves as a 
squared Bessel process of dimension $(|I|-1)(2-\chi|I|/(4\pi N))\geq 2$ under the level 
$b/3$, by \cite[p 442]{reyor}, ${\mathbb Q}(\exists s\in[0,+\infty):\rho_s=0)$. For each 
$t\in[0,\infty)$, we deduce that $0=\Pro(\exists s\in[0,t]:\rho_s=0)\geq \Pro(\nu_\infty\leq t)$ by 
equivalence of $\Pro$ and ${\mathbb Q}$ on $\sigma(R^I_0,(\beta^I_s)_{s\in[0,t]})$. Letting $t\to\infty$, 
we conclude that $\Pro(\nu_\infty<\infty)=0$.
\end{proof}

\section{Positive probability of collisions}\label{posprobcol}

The goal of this section is to establish that in the $N$-particle system, pairs of particles
do collide. The main idea is that for e.g. $I=\{1,2\}$, up to the third term in the right-hand side of 
\eqref{dynrI}, the process $R^I_t$ resembles
a squared Bessel process with dimension $(2-\chi/(2\pi N)) < 2$, which a.s. reaches $0$  
by \cite[page 442]{reyor}.

\begin{proof}[Proof of Proposition \ref{pairco}.]
We thus consider any fixed $N\geq 2$, $\chi>0$, $f_0 \in \cP(\rr^2)$, $t_0>0$ and any 
solution (if it exists) $(X^{i,N}_t)_{i=1,\dots,N,t \in [0,t_0]}$ to \eqref{ps}.
We work by contradiction and assume that a.s., $X^{i,N}_s \neq X^{j,N}_s$ for all $s \in [0,t_0]$ and all $i\ne j$. 
Then the singularity of $K$ is not visited and the particle system \eqref{ps} is classically
strongly well-posed on $[0,t_0]$. Thus for $f_0^{\otimes N}$-a.e.
$(x^1,\dots,x^N) \in (\rr^2)^N$, there is a unique strong solution  $(X^{i,N}_t)_{i=1,\dots,N,t\in [0,t_0]}$
to \eqref{ps} such that a.s., $X^{i,N}_0=x^i$ for all $i$ and $X^{i,N}_s \neq X^{j,N}_s$
for all $s \in [0,t_0]$ and all $i\ne j$.
We fix for the rest of the proof an initial condition $(x^1,\dots,x^N) \in (\rr^2)^N$ enjoying these properties.
All the processes below are defined on the finite time interval $[0,t_0]$.
\vip

{\it Step 1.} By construction, $d=\min_{i\neq j}|x^i-x^j|>0$ and we may of course assume that 
$d=|x^1-x^2|$. We introduce  $\bar{x}:=(x^1+x^2)/2$ and note that
$\min_{3\leq j\leq N}|x^j-\bar{x}|\geq \sqrt{3}d/2$. Fix $1/2< a<b<\sqrt 3/2$ and consider
the stopping time $\tau=\min\{\tau_1,\tau_2,\tau_3\}$, where
\begin{gather*}
\tau_1=\inf\Big\{t\in [0,t_0]\; : \;|X^{1,N}_t-X^{2,N}_t|\geq \frac{2a+1}2d  \Big\}, \\
\tau_2=\inf\Big\{t\in [0,t_0] : \;|X^{1,N}_t+X^{2,N}_t - 2 \bar x| \geq \frac{2a-1}2d  \Big\},\\
\tau_3=\inf\Big\{t\in [0,t_0] : \;\min_{j=3,\dots,N}|X^{j,N}_t-\bar x| \leq bd\Big\},
\end{gather*}
with the convention that $\inf \; \emptyset = t_0$.
We will use that a.s., for all $t\in [0,\tau]$,
$$
\min_{i=1,2,\; j=3,\dots,N} |X^{i,N}_t-X^{j,N}_t| \geq (b-a)d.
$$
Indeed, consider e.g. the case $i=1$ and $j=3$, write 
$|X^{1,N}_t-X^{3,N}_t| \geq  |X^{3,N}_t-\bar x| - |X^{1,N}_t-\bar x|$ and use that
$|X^{3,N}_t-\bar x|\geq bd$ and that $|X^{1,N}_t-\bar x| \leq |X^{1,N}_t-X^{2,N}_t|/2+|X^{1,N}_t+X^{2,N}_t - 2 \bar x|/2
\leq (2a+1)d/4 +(2a-1)d/4=ad$.

\vip

{\it Step 2.} Consider the exponential martingale defined on $[0,t_0]$ by
\begin{align*}
M_t=&\exp\Big[\frac{\chi}{\sqrt{2} N}\sum_{i=1}^N\int_0^{t\wedge\tau}
\Big(\indiq_{\{i\leq 2\}}\sum_{j=3}^N K(X^{j,N}_s-X^{i,N}_s) + \indiq_{\{i\geq 3\}}\sum_{j=1}^2
K(X^{j,N}_s-X^{i,N}_s)\Big)\cdot dB^i_s\\
&\hskip0.8cm-\frac{\chi^2}{4 N^2}\sum_{i=1}^N\int_0^{t\wedge \tau} \Big|
\indiq_{\{i\leq 2\}} \sum_{j=3}^NK(X^{j,N}_s-X^{i,N}_s)+\indiq_{\{i\geq 3\}} \sum_{j=1}^2K(X^{j,N}_s-X^{i,N}_s)\Big|^2ds\Big].
\end{align*}
This is indeed a true martingale, because $K(X^{j,N}_s-X^{i,N}_s)$
is bounded by $(2\pi(b-a)d)^{-1}$ on $[0,\tau]$ for each $i=1,2$ and $j=3,\dots,N$, see Step 1.
Hence $\tilde \Pro := M_{t_0}\cdot \Pro$ is a probability measure equivalent to $\Pro$.
In particular, it also holds that $\tilde \Pro$-a.s., $X^{i,N}_s \neq X^{j,N}_s$
for all $s \in [0,t_0]$ and all $i\ne j$. The Girsanov theorem tells us that, under $\tilde\Pro$, the processes
$$
W^i_t:=B^i_t+\frac{\chi}{\sqrt{2}N}\int_0^{t\wedge\tau}\Big(\indiq_{\{i\leq 2\}}\sum_{j=3}^NK(X^{i,N}_s-X^{j,N}_s)
+\indiq_{\{i\geq 3\}}\sum_{j=1}^2K(X^{i,N}_s-X^{j,N}_s)\Big)ds
$$ 
are independent two-dimensional Brownian motions on $[0,t_0]$. 
We next introduce 
$$
\beta_t=\int_0^t \frac{(X^{1,N}_s-X^{2,N}_s)}{|X^{1,N}_s-X^{2,N}_s|}\cdot d \Big(\frac{W^1_s-W^2_s}{\sqrt 2}\Big)
\quad \hbox{and} \quad \gamma_t= \frac{W^1_t+W^2_t}{\sqrt 2}.
$$ 
It is easily seen, computing brackets and using Karatzas and Shreve 
\cite[Theorem 4.13 p 179]{kashre}, that still under $\tilde\Pro$,
$\beta$ is a one-dimensional Brownian motion on $[0,t_0]$, $\gamma,W^3,\dots,W^N$ are two-dimensional Brownian 
motions on $[0,t_0]$, and all these processes are independent.

\vip

{\it Step 3.} We have 
\begin{align*}
X^{1,N}_t-X^{2,N}_t=&x^1-x^2 + \sqrt 2 (B^1_t-B^2_t) + \frac{2\chi}N \intot K (X^{1,N}_s-X^{2,N}_s) ds\\
&+ \frac \chi N \sum_{j=3}^N \intot \Big( K (X^{1,N}_s-X^{j,N}_s) - K (X^{2,N}_s-X^{j,N}_s) \Big) ds\\
=&x^1-x^2 + \sqrt 2 (W^1_t-W^2_t) + \frac{2\chi}N \intot K (X^{1,N}_s-X^{2,N}_s) ds
\end{align*}
for all $t \in [0,\tau]$.
By the It\^o formula, $Y_t=|X^{1,N}_t-X^{2,N}_t|^2/4$ thus solves, still for $t \in [0,\tau]$,
$$
Y_t:= \frac{d^2}4 + 2\int_0^t\sqrt{Y_s}d\beta_s+\Big(2-\frac{\chi}{2\pi N}\Big)t.
$$
We also have, for all $t \in [0,\tau]$ 
\begin{align*}
X^{1,N}_t+X^{2,N}_t=&2 \bar x  + \sqrt 2 (B^1_t+B^2_t) 
+ \frac \chi N \sum_{j=3}^N \intot \Big( K (X^{1,N}_s-X^{j,N}_s) + K (X^{2,N}_s-X^{j,N}_s) \Big) ds\\
=&2 \bar x+ \sqrt 2 (W^1_t+W^2_t)\\
=& 2 \bar x + 2 \gamma_t,
\end{align*}
and, for all $t\in[0,\tau]$ and all $i=3,\dots,N$ (recall that $K(0)=0$),
\begin{align*}
X^{i,N}_t=&x^i + \sqrt 2 B^i_t + \frac \chi N \intot \sum_{j=1}^N K(X^{i,N}_s-X^{j,N}_s) ds\\
=&x^i + \sqrt 2 W^i_t + \frac \chi N \intot \sum_{j=3}^N K(X^{i,N}_s-X^{j,N}_s) ds.
\end{align*}
We introduce $(\tilde Y_t)_{t\in [0,t_0]}$ the unique strong solution, see \cite[Theorem 3.5 p 390]{reyor}, to 
$$
\tilde Y_t:= \frac{d^2}4  + 2\int_0^t\sqrt{|\tilde Y_s|}d\beta_s+\Big(2-\frac{\chi}{2\pi N}\Big)t.
$$
We clearly have $(Y_t)_{t\in [0,\tau]}=(\tilde Y_t)_{t\in [0,\tau]}$. We next
consider the system 
$$
\tilde{X}^{i,N}_t=x^i+\sqrt{2}W^i_t
+\frac{\chi}{N}\int_0^t\sum_{j=3}^N K(\tilde{X}^{i,N}_s-\tilde{X}^{j,N}_s)ds,\quad i=3,\dots,N,
$$
which classically has a unique strong solution $(\tilde X^{i,N}_t)_{i=3,\dots,N,t\in [0,\sigma)}$
up to $\sigma=\lim_{\ell\to\infty}\inf\{t \in [0,t_0] 
\; : \;\min_{3\leq i<j\leq N}|\tilde X^{i,N}_t-\tilde X^{j,N}_t|\leq 1/\ell\}$ (convention : $\inf\emptyset=t_0$),
which is a.s. positive because the initial conditions $x^3,\dots,x^N$ are pairwise different.
Clearly, $(X^{i,N}_t)_{i=3,\dots,N,t\in [0,\tau\land \sigma)}=(\tilde X^{i,N}_t)_{i=3,\dots,N,t\in [0,\tau\land \sigma)}$.
We conclude this step mentioning that
the processes $(\tilde Y_t)_{t\in [0,t_0]}$, $(\gamma_t)_{t\in [0,t_0]}$ and 
$(\tilde X^{i,N}_t)_{i=3,\dots,N,t\in [0,\sigma)}$ are independent under 
$\tilde \Pro$.

\vip

{\it Step 4.} For any $s_0\in(0,t_0)$, we claim that
$$
\Omega_1 \cap \Omega_2 \cap \Omega_3 \subset \Big\{ \min_{[0,s_0]} |X^{1,N}_s-X^{2,N}_s|=0 \Big\},
$$
where
\begin{gather*}
\Omega_1 =  \Big\{ \min_{[0,s_0]} \tilde Y_s =0,\;\max_{[0,s_0]} \tilde Y_s < \frac{(2a+1)^2d^2}{16} \Big\},
\quad \Omega_2 =  \Big\{ \max_{[0,s_0]} |\gamma_s| < \frac{(2a-1)d}{4} \Big\},\\
\quad \Omega_3 =  \Big\{ \sigma > s_0, \; \min_{s\in[0,s_0],j\geq 3}|\tilde{X}^{j,N}_s-\bar{x}|> b d \Big\}.
\end{gather*}
Indeed, on $\Omega_1$, we have $\max_{[0,s_0]} \tilde Y_s <(2a+1)^2d^2/16$, whence,
since $|X^{1,N}_t-X^{2,N}_t|^2=4\tilde Y_t$ on $[0,\tau]$,
$\max_{[0,s_0\land \tau]} |X^{1,N}_s - X^{2,N}_s| < (2a+1)d/2$ and thus $\tau_1 > s_0\land \tau$.
Since $X^{1,N}_t+X^{2,N}_t=2 \bar x +2 \gamma_t$ on $[0,\tau]$, we deduce that on $\Omega_2$,
$\max_{[0,s_0\land \tau]} |X^{1,N}_s + X^{2,N}_s - 2\bar x| \leq \sup_{[0,s_0\land \tau]}
2|\gamma_s| < (2a-1)d/2$, whence  $\tau_2 > s_0\land \tau$. 
On $\Omega_3$, since $\sigma>s_0$, we have  $(X^{i,N}_t)_{i=3,\dots,N,t\in [0,\tau\land s_0]}
=(\tilde X^{i,N}_t)_{i=3,\dots,N,t\in [0,\tau\land s_0]}$, and thus 
$\min_{s\in[0,s_0 \land \tau ],j\geq 3}|X^{j,N}_s-\bar{x}|> b d$,
so that  $\tau_3 > s_0\land \tau$. As a conclusion, $\tau > s_0 \land \tau$
and thus $\tau>s_0$ on $\Omega_1\cap\Omega_2\cap\Omega_3$. We deduce that 
$\Omega_1 \cap \Omega_2 \cap \Omega_3 \subset \Big\{\tau >s_0,\; \min_{[0,s_0]} \tilde Y_s=0 \Big\}
\subset \Big\{ \min_{[0,s_0]} |X^{1,N}_s-X^{2,N}_s|=0 \Big\}$, because
$\tilde Y_t=|X^{1,N}_t-X^{2,N}_t|^2/4$ for all $t\in [0,\tau]$.

\vip

{\it Step 5.} Here we show that we can find $s_0 \in (0,t_0)$ such that $\tilde
\Pro(\Omega_1 \cap \Omega_2 \cap \Omega_3)>0$. As seen at the end of Step 3, the events
$\Omega_1,\Omega_2$ and $\Omega_3$ are independent (under $\tilde\Pro$). 
It obviously holds true that $\tilde\Pro(\Omega_2)>0$ (for any $s_0>0$) and that 
$\tilde\Pro(\Omega_3)>0$ if $s_0>0$ is small enough because $\sigma>0$ a.s. and 
by continuity of the sample-paths (at time $0$, we have $\min_{j\geq 3}|\tilde{X}^{j,N}_0-\bar{x}|
=\min_{j\geq 3}|x^j-\bar{x}| \geq \sqrt 3 d /2 > b d$).
It thus only remains to verify that $\tilde\Pro(\Omega_1)$ for all $s_0\in (0,t_0)$.
Since, by the comparison principle stated in \cite[Theorem 3.7 p 394]{reyor}, $\tilde\Pro(\Omega_1)$ 
is non-decreasing with $\chi$, it is enough to check that
$\tilde\Pro(\Omega_1)>0$ for all $s_0\in (0,t_0)$ when $\chi<4\pi N$, which we now do.

\vip

It holds that $\tilde Y$ is a squared Bessel process of dimension $\delta:=2-\chi/(2\pi N)$ started at
$y=d^2/4$ and restricted to the time-interval $[0,t_0]$. 
We set $z=(2a+1)^2d^2/16$ and observe that $z>y$. For $x\geq 0$, we also introduce
$\tau_x = \inf\{t\in [0,t_0] \;:\; \tilde Y_t=x\}$. Then $\Omega_1=\{\tau_0 < s_0 \land \tau_z\}$.

\vip

For $x\geq 0$, we denote by $Q_x$ the law of the squared Bessel process of dimension $\delta$ starting from $x$
(on the whole time interval $[0,\infty)$),
and by $q_s(x,u)$ the density of its marginal at time $s>0$, which is a positive function of $u$ on 
$(0,+\infty)$ according to \cite[Corollary 4.1 p441]{reyor}. 
For all $u \ne v$, we define $\tau_u$ as the first passage time at $u$ and 
$\tau_{uv}$ as the first passage time at $v$ after $\tau_u$.
 It holds that $\tilde\Pro(\Omega_1)
=Q_y(\tau_0 < s_0 \land \tau_z)$ and what we have to check is that 
$Q_y(\tau_0 < s_0 \land \tau_z)>0$ for all $s_0\in (0,t_0)$.

\vip
We first show that $Q_x(\tau_0<t)>0$ for all $t>0$ and all $x>0$. Since $\delta<2$, we know from 
\cite[page 442]{reyor} that $Q_x(\tau_0<\infty)=1$ for all $x>0$. With the Markov property, we deduce that
$$
1=\sum_{n\geq 0}Q_x(\tau_0\in (nt/2,(n+1)t/2])\leq Q_x(\tau_0\leq t/2)+\int_0^{+\infty}Q_u(\tau_0\leq t/2)
\bigg(\sum_{n\geq 1}q_{nt/2}(x,u)\bigg)du.
$$
Since, $u\mapsto q_{t/2}(x,u)$ is positive on $(0,+\infty)$, this ensures the positivity of 
$$
Q_x(\tau_0\leq t/2)+1_{\{Q_x(\tau_0\leq t/2)=0\}}\int_0^{+\infty}Q_u(\tau_0\leq t/2)q_{t/2}(x,u)du\leq Q_x(\tau_0\leq t).
$$

\vip

Using the strong Markov property, that $0<y<z$ and the monotonicity of $t\mapsto Q_y(\tau_0\leq t)$,
$$
Q_y(\tau_{z}<\tau_0\leq t)=Q_y(\tau_{zy}<\tau_0\leq t)=\!\int\!\indiq_{\{\tau_{zy}<t\}}Q_y(\tau_0\leq t-s)|_{s=\tau_{zy}}dQ_y
\leq Q_y(\tau_{zy}<t)Q_y(\tau_0\leq t).
$$
By continuity of the sample-paths, $\lim_{s\to 0}Q_y(\tau_{zy}<s)=0$ and we can find $s_1 \in (0,t_0)$ so that 
for all $s_0\in(0,s_1]$, $Q_y(\tau_{zy}<s_0)<1$. We conclude that for all $s_0\in(0,s_1]$,
$$
Q_y(\tau_0\leq s_0\wedge\tau_z)=Q_y(\tau_0\leq s_0)-Q_y(\tau_{z}<\tau_0\leq s_0)
\geq(1-Q_y(\tau_{zy}<s_0))Q_y(\tau_0\leq s_0)>0.
$$
If now $s_0\in[s_1,t_0]$, we obviously have $Q_y(\tau_0\leq s_0\wedge\tau_z) \geq Q_y(\tau_0\leq s_1\wedge\tau_z)>0$.
This ends the step.

\vip

{\it Step 6.} We deduce from Steps 4 and 5 that
$\tilde \Pro (\min_{[0,t_0]} |X^{1,N}_s-X^{2,N}_s|=0)>0$. But $\Pro$ and $\tilde \Pro$ being equivalent,
this implies that $\Pro (\min_{[0,t_0]} |X^{1,N}_s-X^{2,N}_s|=0)>0$, whence a contradiction.
\end{proof}

\section{Two particles system}\label{222}

In this section we consider the particle system \eqref{ps} with $N=2$. 
Assuming that $(X^1_t,X^2_t)_{t\geq 0}$ solves \eqref{ps} with $N=2$, we easily find
that $S_t=X^1_t+X^2_t$ and $D_t=X^1_t-X^2_t$ solve two autonomous equations, namely
$S_t=S_0+2B_t$ and 
\begin{equation}\label{eqd}
D_t=D_0+2W_t+\chi\intot K(D_s)ds, 
\end{equation}
with the two independent $2$-dimensional Brownian motions $B_t=(B^1_t+B^2_t)/\sqrt 2$ and 
$W_t=(B^1_t-B^2_t)/\sqrt 2$. The equation satisfied by $(S_t)_{t\geq 0}$ being trivial, only the study of
\eqref{eqd} is interesting.
During the whole section, the initial condition $D_0$ is only assumed to be a $\rd$-random variable
indpendent of $(W_t)_{t\geq 0}$.

\begin{rk}\label{rkpds}
Theorem \ref{pse} ensures us existence for \eqref{eqd} when $\chi<4\pi$ and $D_0$ is the difference of 
two i.i.d. integrable random vectors.
When $\chi\geq 4\pi$, the equation \eqref{eqd} has no global (in time) solution in the usual sense.
More precisely, assume that it has a global solution $(D_t)_{t\geq 0}$. Then 
$\tau=\inf\{t\geq 0\;:\;D_t=0\}$ is a.s. finite and a.s., $\int_\tau^{\tau+h} |K(D_s)|ds=\infty$ for all $h>0$.
\end{rk}

\begin{proof}
Let thus $\chi\geq 4\pi$ and assume that there is a global solution $(D_t)_{t\geq 0}$ to \eqref{eqd}.
By a direct application of the It\^o formula, this implies that $R_t=|D_t|^2/4$ 
solves $R_t=R_0+2\intot \sqrt{|R_s|}d\beta_s + (2-\chi/(4\pi))t$, where 
$\beta_t=\intot \indiq_{\{D_s\ne 0\}} |D_s|^{-1}D_s\cdot dW_s+\intot \indiq_{\{D_s = 0\}}d\tilde{\beta}_s$ 
is a $1$-dimensional Brownian motion (here $\tilde{\beta}$ is any one-dimensional Brownian 
motion independent of $(D_0,W)$).
According to \cite[p 442]{reyor} combined, when $\chi>8\pi$, with the comparison theorem
\cite[Theorem 3.7 p 394]{reyor}, $\tau=\inf\{t\geq 0\;:\;R_t=0\}$ is a.s. finite.
By the strong Markov property (for the process $R$), the comparison theorem
\cite[Theorem 3.7 p 394]{reyor} and since $\chi\ge 4\pi$, $(R_{\tau+t})_{t\geq 0}$ can be bounded from above by 
a squared $1$-dimensional Bessel process starting from $0$, process with the same law as $(|\beta_t|^2)_{t\geq 0}$. 
For $h>0$, by the occupation times formula
\cite[Corollary 1.6 p 224]{reyor}, $\int_0^h |\beta_s|^{-1}ds=\int_{\rr}|a|^{-1} L^a_h da$.
But $L^0_h>0$ as soon as $h>0$ and we know from \cite[Corollary 1.8 p 226]{reyor}
that $a \mapsto L^a_h$ is a.s. continuous, so that $\int_0^h |\beta_s|^{-1}ds=\infty$ for all $h>0$ a.s.
Thus $4\pi \int_\tau^{\tau+h} |K(D_s)|ds= 
\int_\tau^{\tau+h} R_s^{-1/2}ds =\infty$ for all 
$h>0$ a.s.
\end{proof}

Hence \eqref{eqd} has no global solution for $\chi\geq 4\pi$, while we expect that in some 
sense, the dynamics it represents is meaningful at least for all $\chi \in (0,8\pi)$.
We thus would like to refomulate it,
in such a way that it is possible to build global solutions. More precisely, we would
like to identify, for any value of $\chi > 0$, the limit, as $\e>0$, of the smoothed equation
\begin{equation}\label{eqde}
D^\e_t=D_0+2W_t+\chi \int_0^tK_\e(D^\e_s)ds,
\end{equation}
where $K_\e$ was defined in \eqref{ke}. The regularized drift coefficient $K_\e$ being Lipschitz, existence and 
trajectorial uniqueness hold for this SDE.
We introduce the equation formally satisfied by $Z_t=|D_t|^2D_t$ for $(D_t)_{t\geq 0}$ solution to \eqref{eqd}:
\begin{equation}\label{eqz}
Z_t=Z_0+\intot \sigma(Z_s)dW_s + \intot b(Z_s)ds,
\end{equation}
where $\sigma(z)=2|z|^{-4/3}(|z|^2 I_2 + 2z z^*)$ and $b(z)=(16 -3\chi/(2\pi))|z|^{-2/3}z $.
Here and below, $I_2$ is the identity matrix and $z^*$ is the transpose of $z$.
Here is the main result of this section.

\begin{thm}\label{thN2}
Set $Z_0=|D_0|^2D_0$.

\vip

(i) If $\chi \in (0,8\pi)$,  \eqref{eqz} has a unique (in law) solution $(Z_t)_{t\geq 0}$
such a.s., $\int_0^\infty\indiq_{\{Z_t=0\}}dt=0$. Moreover, if $\chi \in(0,4\pi)$, 
\eqref{eqd} has a unique (in law) solution.
\vip

(ii) If $\chi \geq 8\pi$, \eqref{eqz} has a pathwise unique solution frozen when it reaches $0$
(and it a.s. reaches $0$).

\vip

(iii) In any case, the solution $(D^\e_t)_{t\geq 0}$ to \eqref{eqde} goes in law, as $\e\to 0$,
to $(D_t)_{t\geq 0}$ defined by $D_t=|Z_t|^{-2/3}Z_t\indiq_{\{Z_t\ne 0\}}$ and, when $\chi\in (0,4\pi)$, this 
process $(D_t)_{t\geq 0}$ solves \eqref{eqd} .
\end{thm}

In point (i), uniqueness in law cannot hold true without restriction for \eqref{eqz}:
the time passed at $0$ by the solution that we consider is Lebesgue-nul, while it is easy
to build a solution by freezing the process when it reaches $0$.

\vip

The rest of the section is devoted to the proof of this theorem.
The following lemma is more or less standard.

\begin{lem}\label{tropfacile}
Let $\chi>0$ be fixed.
For each $\e\in (0,1)$, we consider the unique solution $(D^\e_t)_{t\geq 0}$ to \eqref{eqde}
and we put $Z^\e_t=|D^\e_t|^2D^\e_t$. 

\vip

(i) The family $\{(Z^\e_t)_{t\geq 0},\e \in (0,1)\}$ is tight in $C([0,\infty),\rd)$.

\vip

(ii) Any limit point $(Z_t)_{t\geq 0}$ is a weak solution to \eqref{eqz} and, setting $R_t=|Z_t|^{2/3}/4$,
it holds that

(a) if $\chi \in (0,8\pi)$, then $(R_t)_{t\geq 0}$ is a $(2-\chi/(4\pi))$-dimensional squared Bessel process;

(b) if $\chi \geq 8\pi$, then $(R_t)_{t\geq 0}$ is a $(2-\chi/(4\pi))$-dimensional squared Bessel process
frozen when it reaches $0$.
\end{lem}

\begin{proof} We divide the proof in several steps.

\vip

{\it Step 1.} 
Direct applications of the It\^o formula show that
$$
Z^\e_t =Z_0+\intot \sigma(Z^\e_s)dW_s + \intot b_\e(Z^\e_s)ds,
$$
where $b_\e(z)=16|z|^{-2/3}z - (3\chi/(2\pi))(|z|^{2/3}+\e^2)^{-1}z$ and that $R^\e_t:=|D^\e_t|^2/4$ solves 
$$
R^\e_t = R_0 + 2 \intot \sqrt{R^\e_s} d\beta^\e_s+ \intot \Big(2- \frac{\chi R^\e_s}{\pi(\e^2+4 R^\e_s)}\Big)ds,
$$
where $\beta^\e_t =\intot \indiq_{\{D^\e_s\neq 0\}}|D^\e_s|^{-1}D^\e_s\cdot dW_s$. 
Since $\sup_{r\ge 0} (\chi r)/[2\pi\sqrt r(\e^2+4r)]=\chi/(8\pi\e)$, the Girsanov theorem ensures us
that for all $T\in(0,+\infty)$, the law of $(R^\e_t)_{t\in[0,T]}$ is equivalent to the law of the restriction 
to the time interval $[0,T]$ of a $2$-dimensional squared Bessel process starting from $R_0$. 
By \cite[p 442]{reyor}, we deduce that a.s., for all $t>0$, $R^\e_t>0$. As a consequence 
$(\beta^\e_t)_{t\ge 0}$ is a one-dimensional Brownian motion.
\vip

{\it Step 2.} By trajectorial uniqueness for \eqref{eqde}, for $M>0$, on the event 
$\{|D_0|\le M\}$, the solution starting from $D_0$ coincides 
with the one starting from $D_0\indiq_{\{|D_0|\le M\}}$. Therefore, by both 
implications in the Prokhorov theorem, to check that the family $\{(Z^\e_t)_{t\geq 0},\e\in(0,1)\}$ is tight in 
$C([0,\infty),\rd)$ it is enough to do so when $D_0$ is bounded. The tightness property then easily follows from 
the Kolmogorov criterion, using that $\sup_{\e \in (0,1)} |b_\e(z)|$ and $|\sigma(z)|$ both have at most 
affine growth: one classically verifies successively that for all 
$\rho\geq 2$ and all $T>0$ there is $C_{T,\rho}$ such that for all $\e \in (0,1)$, 
$\sup_{[0,T]} \E[|Z_t^\e|^\rho]\leq C_{T,\rho}$ and $\E[|Z_t^\e-Z_s^\e|^\rho|]\leq C_{T,\rho} |t-s|^{\rho/2}$ for all
$0\leq s \leq t \leq T$.

\vip

{\it Step 3.} Using martingale problems, that $b$ and $\sigma$ are continuous and that $b_\e$ converges (uniformly)
to $b$, it is checked without difficulty that any limit point $(Z_t)_{t\geq 0}$ (as $\e\to 0$) of
the family $\{(Z^\e_t)_{t\geq 0},\e>0\}$ is indead a (weak) solution to \eqref{eqz}.

\vip

{\it Step 4.} Here we assume that $\chi \in (0,8\pi)$ and we prove that $(R_t^\e)_{t\geq 0}$ goes in law to 
the squared $(2-\chi/(4\pi))$-dimensional Bessel process. We consider the 
$(2-\chi/(4\pi))$-dimensional Bessel process $(R_t)_{t\geq 0}$ associated to $(\beta^\e_t)_{t\geq 0}$, that is
$R_t=R_0+2\intot \sqrt{R_s}d\beta^\e_s + (2-\chi/(4\pi))t$ (its law does of course not depend
on $\e$) and we prove that $\lim_{\e\to 0} \E[\sup_{[0,T]} |R^\e_t-R_t|]=0$ for all $T>0$, which clearly suffices.

\vip

Since $\int_{\e^{3/2}}^\e x^{-1} dx =\log(1/ \e) /2$, one may construct a family of 
$C^2$ nondecreasing 
convex functions $\varphi_\e:[0,\infty) \mapsto [0,\infty)$, indexed by $\e\in(0,1/2)$ such that 
$\varphi_\e(x)=0$ for $x\leq \e^{3/2}$, 
$\varphi'(x)=1$ for $x\geq \e$ and $\varphi''_\e(x)\leq C 1_{\{\e^{3/2}\leq x\leq \e\}} /[x\log(1/\e)]$ for some 
constant $C\in(1,+\infty)$ not depending on $\e$. Such functions are called Yamada functions in the 
literature. We then observe that $R^\e_t\geq R_t$ for all $t\geq 0$ by the comparison
theorem stated in \cite[Theorem 3.7 p 394]{reyor}. Computing $R^\e_t-R_t$ and applying the 
It\^o formula, we obtain that
\begin{align*}
\varphi_\e(R^\e_t-R_t)=&2\int_0^t\varphi'_\e(R^\e_s-R_s)(\sqrt{R^\e_s}-\sqrt{R_s})
d\beta_s+\frac{\chi}{4\pi}\int_0^t\varphi'_\e(R^\e_s-R_s)\frac{\e^2}{\e^2+4R^\e_s}ds\\
&+2\int_0^t \varphi''_\e(R^\e_s-R_s)(\sqrt{R^\e_s}-\sqrt{R_s})^2ds.
\end{align*}
We next remark that
$$
\varphi'_\e(R^\e_s-R_s)\frac{\e^2}{\e^2+4R^\e_s}
\leq \varphi'_\e(R^\e_s-R_s)\frac{\e^2}{\e^2+4(R^\e_s-R_s)}
\leq \indiq_{\{R^\e_s-R_s\geq \e^{3/2}\}} \frac{\e^2}{\e^2+4(R^\e_s-R_s)}\leq \frac{\sqrt\e}{4}
$$
and that
$$
\varphi''_\e(R^\e_s-R_s)(\sqrt{R^\e_s}-\sqrt{R_s})^2\leq \varphi''_\e(R^\e_s-R_s)(R^\e_s-R_s) 
\leq \frac{C}{\log(1/\e)},
$$
whence (the constant $C$ may now change from line to line)
\begin{align}\label{tip}
\varphi_\e(R^\e_t-R_t)\leq 2\int_0^t\varphi'_\e(R^\e_s-R_s)(\sqrt{R^\e_s}-\sqrt{R_s})d\beta_s
+ \frac{\chi \sqrt \e}{16\pi} t +\frac{C}{\log(1/\e)} t.
\end{align} 
Taking expectations, we conclude that $\E[\varphi_\e(R^\e_t-R_t) ] \leq Ct / \log(1/\e)$.
But since $\varphi_\e(x)\leq x \leq \varphi_\e(x)+\e$, 
we deduce that $\E[R^\e_t-R_t] \leq \e + C t/\log(1/\e)$.
Coming back to \eqref{tip}, using the Doob inequality and that $0 \leq \varphi'_\e \leq 1$ and 
$(\sqrt{R^\e_s}-\sqrt{R_s})^2 \leq R^\e_s-R_s$, we conclude that 
$\E[\sup_{[0,T]}\varphi_\e(R^\e_t-R_t) ] \leq C T/\log(1/\e) + C (\e T + C T^2/\log(1/\e))^{1/2}$
and, finally, that $\E[\sup_{[0,T]} (R^\e_t-R_t) ] \leq \e + C T/\log(1/\e) + C (\e T + C T^2/\log(1/\e))^{1/2}$,
from which the conclusion follows.

\vip

{\it Step 5.} Finally, we assume that $\chi\geq 8\pi$ and we prove that $(R_t^\e)_{t\geq 0}$ 
goes in law to the  $(2-\chi/(4\pi))$-dimensional squared Bessel process
frozen when it reaches $0$. We consider the frozen $(2-\chi/(4\pi))$-dimensional squared Bessel 
process associated to $(\beta^\e_t)_{t\geq 0}$, that is $R_t=R_0+2\intot \sqrt{R_s}d\beta^\e_s + (2-\chi/(4\pi))t$
for all $t \in [0,\tau]$, with $\tau=\inf\{t\geq 0 \; : \; R_t=0\}$ and $R_t=0$ for all $t\geq \tau$.
We will check that for all $\alpha>0$, all $T>0$, $\lim_{\e\to 0}\Pro(\sup_{[0,T]} |R^\e_t - R_t|>\alpha)=0$
and this will complete the proof.
We introduce $\tau_k=\inf\{t\geq 0 \; : \; R_t \leq 1/k\}$
and observe that $\tau=\sup_{k\geq 1} \tau_k$.

\vip

{\it Step 5.1.} For any $\alpha>0$, $t\geq 0$ and $k\geq 1$, 
$\lim_{\e\to 0}\Pro(\sup_{[0,t\land \tau_k)}|R^\e_s-R_s|\geq \alpha)=0$.
Indeed, using that $R^\e_t \geq R_t$ for all $t\geq 0$ by the comparison theorem \cite[Theorem 3.7 p 394]{reyor},
that $R^\e_t-R_t=2 \intot (\sqrt{R^\e_s}-\sqrt{R_s}) d\beta_s^\e+ (\chi/4\pi)\intot [\e^2/(\e^2+4R^\e_s)] ds$
for all $t\in [0,\tau_k]$, and that $|\sqrt x -\sqrt y|\leq k^{1/2}|x-y|/2$ for all $x,y\geq 1/k$, 
it is easily checked, by the Doob inequality, that 
$$
\E\Big[\sup_{[0,t\land \tau_k)}(R^\e_s-R_s)^2\Big]\leq C k \intot \E\Big[\sup_{[0,s\land \tau_k)}(R^\e_u-R_u)^2\Big] ds 
+ C \e^4 k^2 t^2, 
$$
whence $\E[\sup_{[0,t\land \tau_k)}(R^\e_s-R_s)^2 ]\leq  C \e^4k^2t^2\exp(Ckt)$
by the Gronwall lemma.

\vip
{\it Step 5.2.} We write, for $\alpha>0$ and $k\geq 1$ fixed,
\begin{align*}
\Pro\Big(\sup_{[0,T]} (R^\e_t-R_t) \geq \alpha \Big) \leq &
\Pro\Big(\sup_{[0,T\land \tau_k]} (R^\e_t-R_t) \geq \alpha \Big) + 
\Pro\Big(\tau_k<T, R^\e_{\tau_k} > 2/k\Big) \\
&+ \Pro\Big(\tau_k<T, R^\e_{\tau_k} \leq 2/k,  \sup_{[\tau_k,T]} R^\e_t \geq \alpha \Big).
\end{align*}
For the last term, we used that $\sup_{[\tau_k,T]}(R^\e_t-R_t) \geq \alpha$ implies that 
$\sup_{[\tau_k,T]} R^\e_t \geq \alpha$ because $0\leq R_t \leq R^\e_t$.
By Step 5.1, the two first terms tend to $0$ as $\e\to 0$ (recall that $R_{\tau_k}=1/k$), whence
\begin{align*}
\limsup_{\e\to 0} \Pro\Big(\sup_{[0,T]} (R^\e_t-R_t) \geq \alpha \Big) \leq\limsup_{\e\to 0}
\Pro\Big(\tau_k<T, R^\e_{\tau_k} \leq 2/k,  \sup_{[\tau_k,T]} R^\e_t \geq \alpha \Big).
\end{align*}
Using the strong Markov property for the process $R^\e$ as well as its monotony with respect to its initial
condition (by the comparison theorem), we deduce that 
\begin{align*}
\limsup_{\e\to 0} \Pro\Big(\sup_{[0,T]} (R^\e_t-R_t) \geq \alpha \Big) \leq\limsup_{\e\to 0}
\Pro\Big(\sup_{[0,T]} R^{2/k,\e}_t \geq \alpha \Big),
\end{align*}
where 
$$
R^{2/k,\e}_t=2/k+2 \intot \sqrt{R^{2/k,\e}_s} d\beta_s^\e
+ \intot \Big(2- \frac{\chi R^{2/k,\e}_s}{\pi(\e^2+4R^{2/k,\e}_s)}\Big)ds. 
$$
We introduce, for $r\in(0,1)$ and $\e\in(0,1/2)$, the solution $(S^{r,\e}_t)_{t\geq 0}$
to $S^{r,\e}_t=r+2 \intot \sqrt{|S^{r,\e}_s|}d\beta_s^\e + 2\e^2 \intot (\e^2+4 |S^{r,\e}_s|)^{-1}ds$.
Such a solution is pathwise unique by \cite[Theorem 3.5 p 390]{reyor} 
and nonnegative by the comparison theorem \cite[Theorem 3.7 p 394]{reyor}.
Again by the comparison theorem,
and since $\chi\geq 8 \pi$, we find that a.s., $R^{2/k,\e}_t\leq S^{2/k,\e}_t$ for all $t\geq 0$.
Hence 
\begin{align*}
\limsup_{\e\to 0} \Pro\Big(\sup_{[0,T]} (R^\e_t-R_t) \geq \alpha \Big) \leq\limsup_{\e\to 0}
\Pro\Big(\sup_{[0,T]} S^{2/k,\e}_t \geq \alpha \Big) \leq \limsup_{\e\to 0} \frac{\E[\sup_{[0,T]} S^{2/k,\e}_t]}{\alpha}.
\end{align*}
We will verify in the next step that (if $r \in (0,1]$)
\begin{equation}
\E\Big[\sup_{[0,T]} S^{r,\e}_t\Big] \leq C(1+T)(r+1/\log(1/\e))^{1/2}\label{majoes},
\end{equation}
so that $\limsup_{\e\to 0} \Pro(\sup_{[0,T]} (R^\e_t-R_t) \geq \alpha) \leq C(1+T) k^{-1/2}/\alpha$.
Letting $k$ tend to infinity, we 
conclude that, as desired,
$\limsup_{\e\to 0} \Pro(\sup_{[0,T]} (R^\e_t-R_t) \geq \alpha)=0$.

\vip

{\it Step 5.3.} 
To show \eqref{majoes}, we consider the Yamada function $\varphi_\e$ built in Step 4.
By the It\^o formula,
$$
\varphi_\e(S^{r,\e}_t)=\varphi_\e(r)+2\intot \varphi_\e'(S^{r,\e}_s)\sqrt{S^{r,\e}_s}d\beta^\e_s 
+  \intot \varphi_\e'(S^{r,\e}_s) \frac{2\e^2}{\e^2+4 S^{r,\e}_s}ds + 2 \intot\varphi_\e''(S^{r,\e}_s)S^{r,\e}_sds.
$$
Proceeding as in Step 4, we find that 
\begin{align}\label{tap}
\varphi_\e(S^{r,\e}_t)\leq& r +2\intot \varphi_\e'(S^{r,\e}_s)\sqrt{S^{r,\e}_s}d\beta^\e_s 
+ \frac{\sqrt \e}{2}t + \frac C {\log(1/\e)}t \\
\leq& r +\frac C {\log(1/\e)}t
+2\intot \varphi_\e'(S^{r,\e}_s)\sqrt{S^{r,\e}_s}d\beta_s.\notag
\end{align}
Taking expectations, we deduce that $\E[\varphi_\e(S^{r,\e}_t)]\leq r + C t / \log(1/\e)$, whence
$\E[S^{r,\e}_t]\leq r +\e + C t / \log(1/\e)$. Coming back to \eqref{tap} and using the Doob inequality and
that $0\leq \varphi_\e'\leq 1$, we conclude that
$$
\E\Big[\sup_{[0,T]} \varphi_\e(S^{r,\e}_t)\Big] \leq r + \frac{C T}{\log(1/\e)}+ 
C\Big(r T + \e T + \frac{T^2}{\log(1/\e)}   \Big)^{1/2} \leq C (1+T)\Big(r + \frac{1}{\log(1/\e)}\Big)^{1/2}
$$
because $r\in (0,1]$. Then \eqref{majoes} follows from the fact that $x\leq\e+\varphi_\e(x)$.

\vip

\end{proof}

This allows us to conclude when $\chi \geq 8\pi$. 

\begin{proof}[Proof of Theorem \ref{thN2} when $\chi\geq 8\pi$]
The existence of a (weak) solution to \eqref{eqz} follows from Lemma 
\ref{tropfacile}, and the solution built there is frozen when it reaches $0$. 
The pathwise  uniqueness of such a frozen solution follows from the Lipschitz continuity of coefficients 
$\sigma,b$ on $\rd\setminus\{0\}$ and can easily be verified using the stopping times $\tau_\ell=\inf\{t\geq 0\; : \;
|Z_t|\leq 1/\ell\}$ and that $\tau=\inf\{t\geq 0\; : \;|Z_t|=0\} = \sup_{\ell \geq 1} \tau_\ell$
(because $t\mapsto Z_t$ is a.s. continuous on $[0,\infty)$).
Using Lemma \ref{tropfacile}, we easily conclude that
$(Z^\e_t)_{t\geq 0}$ goes in law to this $(Z_t)_{t\geq 0}$.
Since $D^\e_t=|Z^\e_t|^{-2/3}Z^\e_t$ and since the map $z \mapsto |z|^{-2/3}z\indiq_{\{z \ne 0\}}$ is continuous,
we conclude that $(D_t^\e)_{t\geq 0}$ goes in law, as $\e\to 0$, to $(|Z_t|^{-2/3}Z_t\indiq_{\{Z_t\ne 0\}})_{t\geq 0}$.
\end{proof}

To conclude the proof when $\chi \in (0,8\pi)$, the only issue is to check the the uniqueness in law
of the solution.
We define $h_{2\pi}(\theta)=\theta - 2 \pi \lfloor \theta/ (2 \pi) \rfloor \in [0,2\pi)$.

\begin{lem}\label{superlem}
Consider $0\leq s_0 < t_0$, a continuous function $r : [0,\infty) \mapsto \rr_+$ satisfying
that $r_{s_0}=r_{t_0}=0$ and $r_t>0$ for all $t\in (s_0,t_0)$ and 
$\int_{s_0}^{t} (r_s)^{-1} ds=\infty$ for all $t\in (s_0,t_0)$.
There is a law $\Gamma(s_0,t_0,(r_s)_{s\in [s_0,t_0]})$ on $C((s_0,t_0),[0,2\pi))$ 
(with the torus topology on $[0,2\pi)$)  such that for any filtration $(\cH_t)_{t\geq 0}$ in which
we have a $1$-dimensional 
$(\cH_t)_{t\geq 0}$-Brownian motion $(\gamma_t)_{t\geq 0}$ and a $(\cH_t)_{t\geq 0}$-adapted process
$(T_t)_{t\in (s_0,t_0)}$ with $T_t=h_{2\pi}(T_u + \int_{u}^t\! (r_s)^{-1/2} d \gamma_s)$ for all $s_0<u<t<t_0$,
$(T_t)_{t\in (s_0,t_0)}$ is independent of $\cH_{s_0}$ and is
$\Gamma(s_0,t_0,(r_s)_{s\in [s_0,t_0]})$-distributed.
\end{lem}

\begin{proof}[Proof of Lemma \ref{superlem}]
{\it Existence.} Let $u_0\in(s_0,t_0)$ be chosen arbitrarily. We consider a 
Brownian motion $(\gamma_t)_{t\geq 0}$, independent of a random variable $\Theta$,
uniformly distributed on $[0,2\pi)$.
We put $T_t=h_{2\pi}(\Theta + \int_{u_0}^t (r_s)^{-1/2} d \gamma_s)$
for all $t\in (s_0,t_0)$ (with $\int_{u_0}^t (r_s)^{-1/2} d \gamma_s
=-\int_t^{u_0}(r_s)^{-1/2} d \gamma_s$ when $t<u_0$). Then 
$(T_t)_{t\in(s_0,t_0)}$ is clearly continuous for the torus topology and 
it holds that
$T_t=h_{2\pi}(T_u + \int_{u}^t (r_s)^{-1/2} d\gamma_s)$ for all $s_0<u<t<t_0$.
Furthermore, for each fixed $t\in(s_0,t_0)$, by independence between $\Theta$ and $\gamma$, 
the conditional law of $T_t$ knowing $(\gamma_s)_{s\geq 0}$ is the uniform distribution on $[0,2\pi)$,
which implies that $T_t$ is independent of $(\gamma_s)_{s\geq 0}$.
Finally, we have to verify that setting $\cH_t=\sigma((T_s)_{},(\gamma_s)_{s\in [0,t]})$, $(\gamma_s)_{s\geq 0}$
is a $(\cH_t)_{t\geq 0}$-Brownian motion. Let thus $t\in(s_0,t_0)$ be fixed.
We have to verify that $(\gamma_{s}-\gamma_t)_{s\geq t}$ is independent of $(T_s,\gamma_s)_{s\in(s_0,t]}$.
Since $T_s=h_{2\pi}(T_t - \int_s^t(r_u)^{-1/2} d \gamma_u)$ for all $s\in (s_0,t]$, it holds that
$\sigma((T_s,\gamma_s)_{s\in(s_0,t]})=\sigma(T_t,(\gamma_s)_{s\in(s_0,t]})$ and the conclusion easily 
follows from the independence  between $T_t$ and $(\gamma_s)_{s\in[s_0,t_0]}$.

\vip

{\it Uniqueness.} We thus consider a filtration $(\cH_t)_{t\geq 0}$ in which we have a
Brownian motion $(\gamma_t)_{t\geq 0}$ and an adapted process
$(T_t)_{t\in (s_0,t_0)}$ satisfying $T_t=h_{2\pi}(T_u + \int_{u}^t (r_s)^{-1/2} d \gamma_s)$ for all $s_0<u<t<t_0$.
We will show that for any fixed $u_0 \in (s_0,t_0)$, $T_{u_0}$ is uniformly distributed on $[0,2\pi)$
and independent of $\cH_{s_0}\vee\sigma((\gamma_t)_{t\geq 0})$. Since $(T_t)_{t\in[s_0,t_0]}$ is
$\sigma(T_{u_0},(\gamma_t-\gamma_{s_0})_{t\in (s_0,t_0)})$-measurable and since $(\gamma_t)_{t\geq 0}$
is a $(\cH_t)_{t\geq 0}$-Brownian motion, we conclude that  $(T_t)_{t\in(s_0,t_0)}$ is independent of $\cH_{s_0}$.
Furthermore, the process $(T_t)_{t\in(s_0,t_0)}$ clearly has the same law as the one built above.

\vip

For 
$0<\e<\eta<u_0-s_0$, we have $T_{u_0}=h_{2\pi}(T_{s_0+\epsilon} + \int_{s_0+\epsilon}^{s_0+\eta}(r_s)^{-1/2} d \gamma_s
+\int_{s_0+\eta}^{u_0}(r_s)^{-1/2} d \gamma_s)$. By assumption, the vector 
$(\int_{s_0+\epsilon}^{s_0+\eta}(r_s)^{-1/2} d \gamma_s,\int_{s_0+\eta}^{u_0}(r_s)^{-1/2} d \gamma_s)$
has independent components and is independent of $\cH_{s_0}\vee\sigma(T_{s_0+\epsilon})$. 
Setting $\sigma_{\e,\eta}=\int_{s_0+\epsilon}^{s_0+\eta}(r_s)^{-1}ds$, we thus have, for any
$\varphi:\rr \mapsto [0,\infty)$ continuous and $2\pi$-periodic, 
\begin{align}\label{jab}
&\E[\varphi(T_{u_0}) \;\vert\; \cH_{s_0}\vee\sigma(T_{s_0+\e},(\gamma_{s}-\gamma_{s_0+\eta})_{s\geq s_0+\eta})]\\
=& \int_\rr \varphi\Big(T_{s_0+\e} +\int_{s_0+\eta}^{u_0}(r_s)^{-1/2} d \gamma_s+ \sigma_{\e,\eta} x \Big) 
\frac{e^{-x^2/2}}{\sqrt{2\pi}} dx 
\to (2\pi)^{-1} \int_0^{2\pi} \varphi(x)dx\notag
\end{align}
a.s. as $\e\to 0$. This last convergence follows from the facts that $\lim_{\e\to 0} \sigma_{\e,\eta} =\infty$
and that, setting $\bar{\varphi}(x):=\varphi(x)-(2\pi)^{-1} \int_0^{2\pi} \varphi(y)dy$ 
and $\Phi(x):=\int_0^x\bar{\varphi}(y)dy$, for all $\theta \in [0,2\pi)$,
\begin{align*}
   \Big|\int_\rr \varphi\Big(\theta + \sigma x \Big) \frac{e^{-x^2/2}}{\sqrt{2\pi}}& dx
- \frac 1{2\pi} \int_0^{2\pi} \varphi(x)dx \Big|=\frac1{\sqrt{2\pi}}\Big|\sum_{k\in{\mathbb Z}}
\int_{0}^{2\pi/\sigma}\bar{\varphi}(\sigma y)e^{-(y-(\theta+2k\pi)/\sigma)^2/2}dy\Big |\\
&=\frac1{\sigma\sqrt{2\pi}}\Big|\sum_{k\in{\mathbb Z}}\int_{0}^{2\pi/\sigma}\Phi(\sigma y)
\times(y-(\theta+2k\pi)/\sigma)e^{-(y-(\theta+2k\pi)/\sigma)^2/2}dy\Big|\\
&\leq \frac{\sqrt{2\pi}}\sigma\sup_{x\in[0,2\pi)}|\varphi(x)|\int_\rr|z|e^{-z^2/2}dz\\
&=\frac{2\sqrt{2\pi}}\sigma\sup_{x\in[0,2\pi)}|\varphi(x)|.
\end{align*}
We used an integration by parts, that $\Phi(0)=\Phi(2\pi)=0$ and that 
$|\Phi(y)|\leq 2\pi \sup_{x\in[0,2\pi)}|\varphi(x)|$
for all $y \in [0,2\pi)$.

\vip

We deduce from \eqref{jab} that $T_{u_0}$ is uniformly distributed on $[0,2\pi)$ and is independent
of $\cH_{s_0}\vee\sigma((\gamma_{s}-\gamma_{s_0+\eta})_{s\geq s_0+\eta})$. Since $\eta>0$ can be chosen 
arbitrarily small, we conclude that
$T_{u_0}$ is independent of $\cH_{s_0}\vee\sigma((\gamma_s-\gamma_{s_0})_{s\geq s_0})
=\cH_{s_0}\vee\sigma((\gamma_s)_{s\geq 0})$ 
as desired.
\end{proof}

\begin{lem}\label{wu}
Assume that $\chi \in (0,8\pi)$.
There is uniqueness in law for \eqref{eqz} among solutions such that a.s., $\int_0^\infty\indiq_{\{Z_t=0\}}dt=0$.
\end{lem}

\begin{proof}
As in the proof of Theorem \ref{thN2} when $\chi\geq 8\pi$, \eqref{eqz} admits a 
pathwise unique solution until it reaches $0$.
All the difficulty is thus to prove the uniqueness in law of the solution started at $0$.
We thus consider, if it exists, a continuous solution
$(Z_t)_{t\geq 0}$ to \eqref{eqz} with $Z_0=0$, adapted to some filtration $(\cF_t)_{t\geq 0}$ in which
$(W_t)_{t\ge 0}$ is a $2$-dimensional Brownian motion, and such that $\int_0^\infty\indiq_{\{Z_t=0\}}dt$ vanishes a.s.

\vip

{\it Step 1.} We define $R_t=|Z_t|^{2/3}/4$ and 
$\beta_t=\intot \indiq_{\{Z_s\ne 0\}}|Z_s|^{-1}Z_s\cdot dW_s
$, which is clearly a $1$-dimensional $(\cF_t)_{t\geq 0}$-Brownian motion. 
Here we prove that 
\begin{equation}
   R_t=2\intot \sqrt{R_s} d\beta_s + (2-\chi/(4\pi))t.\label{evolr}
\end{equation}
Starting from \eqref{eqz} (with $Z_0=0$) and using the It\^o formula, we easily find that 
$$
|Z_t|^2= 12 \intot |Z_s|^{5/3} d\beta_s + (72-3\chi/\pi)\intot |Z_s|^{4/3}ds.
$$
For $\eta>0$, using again It\^o's formula, we find that
\begin{align*}
(|Z_t|^2+\eta)^{1/3}=&\eta^{1/3}+4\intot |Z_s|^{5/3}(|Z_s|^2+\eta)^{-2/3} d\beta_s+ 
(24-\chi/\pi)\intot|Z_s|^{4/3}(|Z_t|^2+\eta)^{-2/3}ds\\
&-16\intot|Z_s|^{10/3}(|Z_s|^2+\eta)^{-5/3}ds.
\end{align*}
Since $\int_0^t\indiq_{\{Z_s=0\}}ds=0$ a.s. by assumption, the Lebesgue theorem ensures us that 
the sum of the two last terms in 
the right-hand side converges a.s. to $(8-\chi/\pi)t$ as $\eta\to 0$. The It\^o isometry ensures that 
the second term in the right-hand side converges in $L^2$ to $4\intot |Z_s|^{1/3}d\beta_s$.
All in all, we find that $|Z_t|^{2/3}=4\intot|Z_s|^{1/3}d\beta_s +(8-\chi/\pi)t$. Dividing by $4$ completes
the proof of \eqref{evolr}.

\vip

{\it Step 2.} We consider, for each $\eta>0$, a nondecreasing $C^2$-function 
$\psi_\eta:[0,\infty)\mapsto[0,\infty)$ such that $\psi(u)=0$ for all $u\in [0,\eta/2]$ and
$\psi(u)=1$ for all $u\geq \eta$. Observing that $\psi_\eta(R_t)|Z_t|^{-1}Z_t=\Psi_\eta(Z_t)$ where 
$\Psi_\eta(z)=\psi_\eta(|z|^{2/3}/4)|z|^{-1}z$
is of class $C^2$ on $\rd$, we easily obtain, starting from \eqref{eqz} and applying the It\^o formula,
\begin{align}\label{psie}
\psi_\eta(R_t)\frac{Z_t}{|Z_t|}=& \intot \!
\psi_\eta(R_s)\Big(2Z^\perp_s\frac{Z_s^\perp.dW_s}{|Z_s|^{7/3}}-
\frac{2Z_s}{|Z_s|^{5/3}}ds\Big)\!+\!\intot\! \frac{Z_s}{|Z_s|}(\psi_\eta'(R_s)dR_s+2\psi_\eta''(R_s)R_sds)
\end{align}
where, for $z\in\rr^2$ with respective coordinates $z_1$ and $z_2$, $z^\perp$ denotes the element of 
$\rr^2$ with respective coordinates $-z_2$ and $z_1$.

\vip

Let $\gamma_t=\intot \indiq_{\{Z_s\ne 0\}}|Z_s|^{-1}Z_s^\perp\cdot dW_s$. Since 
$\langle\beta,\gamma\rangle_t=\intot \indiq_{\{Z_s\ne 0\}}|Z_s|^{-2}Z_s\cdot Z_s^\perp ds=0$, 
the process $(\gamma_t)_{t\geq 0}$ is a $1$-dimensional $(\cF_t)_{t\geq 0}$-Brownian independent of $(\beta_t)_{t\ge 0}$
and thus also of $(R_s)_{s\geq 0}$ (because $(R_s)_{s\geq 0}$ is 
$\sigma(R_0,(\beta_s)_{s\geq 0})$-measurable by pathwise uniqueness for the SDE it solves).

\vip

For any $0< u < t$, on the event $\{\inf_{[u,t]}R_s>0\}$, choosing $\eta \in (0,\inf_{[u,t]}R_s)$
in the difference between \eqref{psie} and the same equation with $t$ replaced by $u$, we obtain
\begin{align}\label{psie2}
\frac{Z_t}{|Z_t|}&= \frac{Z_u}{|Z_u|}
+ \int_u^t \Big(2\frac{Z^\perp_s}{|Z_s|}\frac{d\gamma_s}{|Z_s|^{1/3}}-
\frac{2Z_s}{|Z_s|^{5/3}}ds\Big)
=\frac{Z_u}{|Z_u|}+ \int_u^t \Big(\frac{Z^\perp_s}{|Z_s|}\frac{d\gamma_s}{\sqrt{R_s}}-
\frac{Z_s}{|Z_s|}\frac{ds}{2R_s}\Big).
\end{align}

{\it Step 3.} For $s>0$ such that $R_s>0$ we 
define $T_s\in[0,2\pi)$ through the equality $|Z_s|^{-1}Z_s=e^{iT_s}$.
For $s\geq 0$ with $R_s=0$, we simply put $T_s=0$.
We used the 
natural identification between $\rr^2$ and $\cc$ : for $\theta \in \rr$, we denote by $e^{i\theta}$ 
(resp. $ie^{i\theta}$) the $2$-dimensional vector with coordinates $\cos \theta$ and $\sin \theta$ 
(resp. $-\sin \theta$ and $\cos \theta$). We claim that for all $0< u < t$, on the event $\{\inf_{[u,t]}R_s>0\}$,
it holds that $T_t=h_{2\pi}(T_u+\int_u^tR_s^{-1/2}d\gamma_s)$.

\vip

To check this claim,  on the event $\{\inf_{[u,t]}R_s>0\}$,
we introduce
${\mathcal T}_v=T_u+\int_u^vR_s^{-1/2}d\gamma_s$, for all $v\in [u,t]$.
Since $(\gamma_v)_{v\geq 0}$ is independent of the event $\{\inf_{[u,t]}R_s>0\}$, we can apply the It\^o formula:
$$
\hbox{for all $v\in[u,t]$},\quad e^{i{\mathcal T}_v}=e^{iT_u}+\int_u^v \Big(ie^{i{\mathcal T}_s}\frac{d\gamma_s}{\sqrt{R_s}}-
e^{i{\mathcal T}_s}\frac{ds}{2R_s}\Big).
$$
Recalling \eqref{psie2} and using a uniqueness argument,
we deduce that on the event $\{\inf_{[u,t]}R_s>0\}$,
$|Z_v|^{-1}Z_v=e^{i{\mathcal T}_v}$ whence $T_v=h_{2\pi}({\mathcal T}_v)$ for all $v\in [u,t]$.

\vip

{\it Step 4.} Here we check that a.s., $\int_t^{t+h} R_s^{-1}ds = \infty$ for
all $t\geq 0$ such that $R_t=0$ and all $h>0$. This follows from the fact that for all $T>0$,
$\lim_{u\searrow 0}\sup_{t\in[0,T]} [u(1\vee \log(1/u))]^{-1/2}|\sqrt{R_{t+u}}-\sqrt{R_t}|=\sqrt{2}$ a.s., 
see Khoshnevisan \cite[(2.1a) p 1299]{kosh} and recall that $(R_s)_{s\geq 0}$ is a squared
$(2-\chi/(4\pi))$-dimensional Bessel process starting from $0$ by Step 1, with $2-\chi/(4\pi)>0$.

\vip

{\it Step 5.} 
Here we verify that conditionally on $(R_s)_{s\geq 0}$, for any $\sigma((R_s)_{s\geq 0})$-measurable finite family
$0<s_1<t_1<s_2<t_2<\dots<s_n<t_n$ such that for all $k=1,\dots,n$, 
$R_{s_k}=R_{t_k}=0$ and $R_s>0$ on $(s_k,t_k)$, the variables 
$\{(T_s)_{s\in(s_k,t_k)}, k=1,\dots,n\}$ are independent and for each $k=1,\dots,n$,
$(T_s)_{s\in(s_k,t_k)}$ is $\Gamma(s_k,t_k,(R_s)_{s\in(s_k,t_k)})$-distributed. The function $\Gamma$
was introduced in Lemma \ref{superlem}.

\vip
Let $({\bf Z}_t,{\bf g}_t)_{t\geq 0}$ denote the canonical process on $C([0,\infty),\rd\times\rr)$
endowed with the conditional law of $(Z_t,\gamma_t)_{t\geq 0}$ knowing $(R_t)_{t\geq 0}$.
We define ${\bf T}_t\in [0,2\pi)$ by $|{\bf Z}_t|^{-1}{\bf Z}_t=e^{i{\bf T}_t}$ if ${\bf Z}_t\ne 0$
and ${\bf T}_t=0$ else. We introduce the filtration $\cH_t=\sigma(({\bf T}_s,{\bf g}_s)_{s\in [0,t]})$.
We claim that a.s., $({\bf g}_t)_{t\geq 0}$ is a $(\cH_t)_{t\geq 0}$-Brownian motion, because 
$(\gamma_t)_{t\geq 0}$ is independent of $\sigma((R_s)_{s\geq 0})$ and is a Brownian motion in the filtration 
$(\cF_t)_{t\geq 0}$ to which $(T_t)_{t\geq 0}$ is adapted: for all 
$t>0$, all bounded measurable $\Phi,\Psi$,
\begin{align*}
&\E\Big[\Phi((\gamma_{t+s}-\gamma_t)_{s\geq 0})\Psi((\gamma_s,T_s)_{s\in[0,t]}) \Big\vert (R_s)_{s\geq 0} \Big]\\
=&\E\Big[\Psi((\gamma_s,T_s)_{s\in[0,t]}) \E\Big[\Phi((\gamma_{t+s}-\gamma_t)_{s\geq 0}) \Big\vert 
\cF_t\lor \sigma((R_s)_{s\geq 0}) \Big] \Big\vert (R_s)_{s\geq 0} \Big]
\\
=&\E\Big[\Phi((\gamma_{t+s}-\gamma_t)_{s\geq 0})\Big] 
\E\Big[\Psi((\gamma_s,T_s)_{s\in[0,t]}) \Big\vert (R_s)_{s\geq 0} \Big].
\end{align*}
Fix now $k\in \{1,\dots, n\}$.
It a.s. holds that ${\bf T}_t=h_{2\pi}({\bf T}_u+\int_u^t R_s^{-1/2}d {\bf g}_s)$ for all $s_k<u<t<t_k$ by Step 3
and that $\int_{s_k}^t R_s^{-1}ds=\infty$ for all $t\in (s_k,t_k)$ by Step 4.
Applying Lemma \ref{superlem}, we find that a.s.,
$({\bf T}_s)_{s\in(s_k,t_k)}$ is independent of $\cH_{s_k}$ and is $\Gamma(s_k,t_k,(R_s)_{s\in(s_k,t_k)})$-distributed.
Using that $({\bf T}_s)_{s\in(s_k,t_k)}$ is $\cH_{t_k}$-measurable for each $k=1,\dots,n$,
the independence easily follows.

\vip

{\it Step 6.} By Step 1, $(R_t)_{t\geq 0}$ is
a $(2-\chi/(4\pi))$-dimensional Bessel process starting from $0$.
By Step 5, the conditional law of $(T_t\indiq_{\{R_t\ne 0\}})_{t\geq 0}$ knowing $(R_t)_{t\geq 0}$ is also
determined: conditionally on $(R_s)_{s\geq 0}$, 
for any $\sigma((R_s)_{s\geq 0})$-measurable finite family $\{(s_k,t_k),k=1,\dots,n\}$ of excursions of 
$(R_s)_{s\geq 0}$, we know the law of 
$(T_s)_{s\in\cup_{k=1}^n (s_k,t_k)}$. Since by construction $Z_t=(4R_t)^{3/2}e^{iT_t}\indiq_{\{R_t \ne 0\}}$, 
the law of $(Z_t)_{t\geq 0}$ is thus entirely characterized.
\end{proof}

Finally, we can give the

\begin{proof}[Proof of Theorem \ref{thN2} when $\chi \in (0,8\pi)$]
First, the existence of a solution $(Z_t)_{t\geq 0}$ to \eqref{eqz} such that a.s. 
$\int_0^\infty\indiq_{\{Z_t=0\}}dt=0$ follows from Lemma 
\ref{tropfacile}: the solution $(Z_t)_{t\geq 0}$ built there satisfies that $|Z_t|^{2/3}/4$ is a
$(2-\chi/(4\pi))$-dimensional Bessel process, whence $\int_0^\infty\indiq_{\{Z_t=0\}}dt=0$ a.s. by \cite[p 442]{reyor}.
The uniqueness in law of this solution has been checked in Lemma \ref{wu}.
The convergence of $(Z^\e_t)_{t\geq 0}$ to $(Z_t)_{t\geq 0}$ clearly follows from Lemma \ref{tropfacile}
and from this uniqueness in law. This implies as in the case $\chi\geq 8\pi$ that 
$(D^\e_t)_{t\geq 0}$ goes in law to $(|Z_t|^{-2/3}Z_t\indiq_{\{Z_t\ne 0\}})_{t\geq 0}$.

\vip

It remains to verify that when $\chi\in (0,4\pi)$,
$D_t=|Z_t|^{-2/3}Z_t\indiq_{\{Z_t\ne 0\}}$ solves \eqref{eqd} and that uniqueness in law holds true for \eqref{eqd}.

\vip

For $(D_t)_{t\geq 0}$ a solution to \eqref{eqd},
one easily checks by It\^o's formula that $Z_t=|D_t|^2D_t$ solves \eqref{eqz} and that $|D_t|^2$ is a 
$(2-\chi/(4\pi))$-dimensional Bessel process, whence $\int_0^\infty\indiq_{\{Z_t=0\}}dt=0$ a.s. by \cite[p 442]{reyor}.
The uniqueness in law for \eqref{eqd} then follows from Lemma \ref{wu}.

\vip

For $(Z_t)_{t\geq 0}$ built above, by It\^o's formula, for $\eta>0$,
\begin{align*}
(|Z_t|^2+\eta)^{-1/3}&Z_t=(|Z_0|^2+\eta)^{-1/3}Z_0+2\intot |Z_s|^{2/3}(|Z_s|^2+\eta)^{-1/3}dW_s\\
&+4\intot\Big(|Z_s|^{-1/3}(|Z_s|^2+\eta)^{-1/3}-|Z_s|^{5/3}(|Z_s|^2+\eta)^{-4/3}\Big)Z_sd\beta_s\\
&+\intot\bigg((16-3\chi/(2\pi))|Z_s|^{-2/3}(|Z_t|^2+\eta)^{-1/3}+(\chi/\pi-48)|Z_s|^{4/3}(|Z_t|^2+\eta)^{-4/3}\\
&\hskip6cm+32|Z_s|^{10/3}(|Z_t|^2+\eta)^{-7/3}\bigg)Z_sds.
\end{align*}
By the It\^o isometry and the Lebesgue theorem and since a.s. $\int_0^t\indiq_{\{Z_s=0\}}ds=0$, the 
second term on the RHS tends to $2W_t$ in $L^2$ and the third term on the RHS tends to $0$ in $L^2$.
Since $|Z_t|^{2/3}/4$ is a $(2-\chi/(4\pi))$-dimensional squared Bessel process and $2-\chi/(4\pi)>1$, 
\cite[Exercise 1.26 p 451]{reyor} ensures that a.s. $\intot|Z_s|^{-1/3}ds<\infty$. Hence the Lebesgue 
theorem ensures us that the last term on the RHS converges a.s. to 
$-(\chi/(2\pi))\intot |Z_s|^{-4/3}Z_sds$.
We conclude that
$D_t=|Z_t|^{-2/3}Z_t\indiq_{\{Z_t\ne 0\}}$ solves $D_t=D_0+2W_t-(\chi/(2\pi))\intot |D_s|^{-2}D_sds$,
which completes the proof.
\end{proof}

\section{On the system with $N\ge 3$ particles}\label{Nps}

\subsection{Classification of reflecting and sticky collisions}\label{sss}
We have seen in the proof of Lemma \ref{lemnoncol3part}-Step 2 that very roughly,
the empirical variance of the positions of $k$ particles in the system with $N$ particles
resembles a squared Bessel process of dimension $\delta_{N,\chi}(k)=(k-1)(2-\chi k /(4\pi N))$.
Fix $\chi>0$ and $N\geq 3$ and consider the regularized particle system \eqref{psee}, which is always well-posed.
We now describe formally the expected behavior of its limit as $\e\to 0$.
According to \cite[Page 442]{reyor} and the comparison theorem 
\cite[Theorem 3.7 p 394]{reyor}, the following events should occur:

\vip

$\bullet$ if $\delta_{N,\chi}(k)\geq 2$, no collisions of subsystems of $k$ particles,

$\bullet$  if $\delta_{N,\chi}(k)\in (0,2)$, (instantaneously) reflecting collisions of subsystems of 
$k$ particles,

$\bullet$ if $\delta_{N,\chi}(k)\leq 0$, sticky collisions of subsystems of $k$ particles.

\vip

Let us now study the inequality $\delta_{N,\chi}(k) \geq 2$. 
We have already seen in the 
proof of Lemma \ref{lemnoncol3part}-Step 2 that when 
$\chi\in(0,8\pi(N-2)/(N-1)]$, $\delta_{N,\chi}(k)\ge 2$ for all $k\in\{3,\hdots,N\}$.
When $\chi\in(8\pi(N-2)/(N-1),4\pi N/3]$, $\delta_{N,\chi}(3)\ge 2$ 
whereas $\delta_{N,\chi}(2)<2$ and $\delta_{N,\chi}(N)<2$, hence the two roots 
$x^{\pm}_{N,\chi}=[1+(8\pi N)/\chi \pm \sqrt{(1+8\pi N/\chi)^2-64\pi N/\chi}]/2$
of the second order equation $\delta_{N,\chi}(x)=2$ are such that $x^-_{N,\chi}\in (2,3]$ and $x^+_{N,\chi}\in [3,N)$,
so that $\delta_{N,\chi}(2)<2$,  $\delta_{N,\chi}(k)\geq 2$ for $k\in \{3,\dots,\lfloor x^+_{N,\chi}\rfloor\}$
and   $\delta_{N,\chi}(k)<2$ for $k\in \{\lfloor x^+_{N,\chi}\rfloor+1,\dots,N\}$.
Finally, one easily checks that $x^-_{N,4\pi N/3}=3$ and $x^{+}_{N,4\pi N/3}=4$.
By strict monotonicity of the map $\chi\mapsto\delta_{N,\chi}(k)$, we conclude that if $\chi>4\pi N/3$,
then $\delta_{N,\chi}(k)<2$ for all $k\in\{2,\hdots,N\}$.

\vip

Let us next study the inequality $\delta_{N,\chi}(k)\le 0$, which, for $k\in\{2,\hdots,N\}$ is equivalent to 
$k\geq 8\pi N/\chi$. Hence for $\chi\in (0,8\pi)$, $\delta_{N,\chi}(k)>0$ for all $k\in\{2,\hdots,N\}$ 
whereas for $\chi\in [8\pi,4\pi N)$, $\delta_{N,\chi}(k)>0$ for all 
$k\in\{2,\hdots,\lceil 8\pi N/\chi\rceil -1\}$ and $\delta_{N,\chi}(k)\le 0$ for all 
$k\in\{\lceil 8\pi N/\chi\rceil,\hdots,N\}$ with the two sets non empty. When 
$\chi\ge 4\pi N$, $\delta_{N,\chi}(k)\le 0$ for all $k\in\{2,\hdots,N\}$.

%Since 
%$\delta'_{N,\chi}(\frac{8\pi N}{\chi})=\frac{\chi}{4\pi N}-2>-2$ and the function 
%$x\mapsto \chi_{N,\chi}(x)$ is concave, one has $\chi_{N,\chi}(\lceil 8\pi N/\chi \rceil-1)\in(0,2)$.

\vip

When $N\geq 6$, we end up with the following picture.

\vip

(a) If $\chi \in (0,8\pi(N-2)/(N-1)]$, the regularized particle system should tend
to the particle system \eqref{ps} and the latter should have a unique (in law) solution.
Indeed, it holds that $\delta_{N,\chi}(k) \geq 2$ for all $k\geq 3$ and that $\delta_{N,\chi}(2)\in (0,2)$,
so that only binary reflecting collisions occur.
We have already checked a tightness/consistency result in this spirit in Theorem \ref{pse2}.
Only the uniqueness in law remains open.

\vip

(b) If $\chi \in (8\pi(N-2)/(N-1),8\pi)$, the regularized particle system should tend
to the particle system \eqref{ps} and the latter should also have a unique (in law) solution.
One may check that $k_0:=\lfloor x^+_{N,\chi}\rfloor +1\in \{N-1,N\}$ (it suffices to verify that
$\delta_{N,\chi}(N-2)\geq \delta_{N,8\pi}(N-2)\geq 2$).
In this situation, there should be binary reflecting collisions and also 
reflecting collisions of subsystems of particles with cardinality in $\{k_0,N\}$.
To check the existence (and {\it a fortiori} uniqueness) of such a process, 
one has to control the drift term during the collisions with reflection.
In the present paper, we are more or less able to contol the drift during a (reflecting) binary 
collision, but we have not the least idea of what to do during a $k$-ary reflecting collision with $k\geq 3$.

\vip

(c) If $\chi \in [8\pi,4\pi N/3]$, the regularized particle system should tend to a particle
system with sticky collisions that we will describe more precisely in the next subsection. One can check
that, for $k_0:=\lfloor x^+_{N,\chi}\rfloor+1> 4$ and $k_1:=\lceil 8\pi N/\chi\rceil\leq N$, we have
$k_0 \in \{k_1-2,k_1-1\}$ (just verify that $\delta_{N,\chi}(k_1-3)\geq \delta_{N,\chi}(8\pi N/\chi-2)\geq 2$). 
Thus, binary reflecting collisions, as well as $k$-ary
reflecting collisions, for $k \in\{k_0,k_1-1\}$, should occur, as well as sticky collisions
of subsytem of $k$-particles, for $k\in\{k_1,\dots,N\}$.
Assume e.g. that $k_0=k_1-1$. What might happen is that, at some time,  
$k_0$ particles become close to each other, they may collide (with reflection) a few times,
then another particle is attracted in the zone, the $k_0+1=k_1$ particles meet and then remain stuck forever.
Such a cluster will move with a very small diffusion coefficient and should collide later with other particles
(or clusters) in a sticky way. Of course, such a result would be very interesting but
it seems very difficult to prove, because to check the existence of such a process,
one would have to control the drift term during the collisions with reflection, as mentioned previously.
The sticky collisions should be easier to describe.

\vip

(d) If $\chi \in (4\pi N/3,4\pi N)$, the same situation as previously should arise,
except that there should be $k$-ary reflecting collisions for all $k\in\{2,\dots,\lceil 8\pi N/\chi\rceil-1\}$
and sticky $k$-ary collisions for all $k\in\{\lceil 8\pi N/\chi\rceil,\dots,N\}$.

\vip

(e) If finally $\chi \geq 4\pi N$, then there should be  sticky $k$-ary collisions for all 
$k\in\{2,\dots,N\}$.

\vip

When $N=5$, we find the following dichotomy. If $\chi\in(0,6\pi]$, only binary reflecting collisions.
If $\chi\in(6\pi,20\pi/3]$, only reflecting collisions of subsystems of $k\in\{2,5\}$ particles.
If $\chi\in(20\pi/3,8\pi)$, only reflecting collisions of subsystems of $k\in\{2,3,4,5\}$ particles.
If $\chi\in[8\pi, 20\pi)$, reflecting collisions of subsystems of 
$k\in\{2,\hdots,\lceil40\pi/\chi\rceil-1\}$ particles and sticky collisions of subsystems of 
$k\in\{\lceil 40 \pi/\chi\rceil,\hdots,5\}$ particles.
If $\chi\geq 20\pi$, $k$-ary sticky collisions for all $k\in\{2,\dots,5\}$.

\vip

When finally $N\in\{3,4\}$, $8\pi(N-2)/(N-1)=4\pi N/3$ and 
the situation is as follows.  If $\chi \in (0,8\pi(N-2)/(N-1)]$, 
only binary reflecting collisions. If  $\chi\in (4\pi N/3,8\pi)$, reflecting collisions of subsystems of 
$k\in\{2,\hdots,N\}$ particles. If $\chi\in [8\pi,4\pi N)$, $k$-ary reflecting collisions for
$k\in\{2,\hdots,\lceil8\pi N/\chi\rceil-1\}$ and $k$-ary sticky collisions of subsystems of 
$k\in\{\lceil8\pi N/\chi\rceil,\hdots,N\}$ particles. If $\chi\geq 4\pi N$, $k$-ary sticky collisions for all 
$k\in\{2,\dots,N\}$.

\subsection{A particle system in the supercritical case}

When $\chi\ge 8\pi$, the following dynamics should describe the limit of the regularized particle system
as $\e\to 0$. Particles are characterized by their masses and their positions. Initially,
we start with $N$ particles with masses $\nu^1_0,\dots,\nu^N_0$ all equal to $1/N$ and with some given positions
$X^{1,N}_0,\dots,X^{N,N}_0$.
If now at some time
$t\geq 0$, we have $N_t$ particles ($N_t$ will be a.s. nonincreasing) with masses $\nu^1_t,\dots,\nu^{N_t}_t$
(such that $\sum_1^{N_t}\nu^i_t=1$),
we make the positions evolve according to
\begin{equation}\label{d}
dX^{i,N}_t=\sqrt{\frac{2}{N \nu^i_t}}dB^i_t+\chi \sum_{j=1}^{N_t}\nu^j_tK(X^{i,N}_t-X^{j,N}_t)dt,
\quad i=1,\dots,N_t
\end{equation}
until the next collision between at least two of these $N_t$ particles.
If the sum $S$ of the masses of the particles involved in the collision is smaller than
$8\pi/\chi$, they should automatically separate instantaneously and we carry on 
making evolve the system according to \eqref{d} (with the same values for the masses and for $N_t$)
until the next collision.
If now $S$ exceeds $8\pi/\chi$, the particles involved in the collision are replaced by a single
particle with mass $S$, the number of particles is decreased accordingly, the particles are relabeled,
and we make evolve the system according to \eqref{d} with these new values for $N_t$ and for the masses
until the next collision.

\vip

By construction, the masses take values in $\{1/N,2/N,\dots,N/N\}$
and actually in $\{k/N \; : \; k=1$ or $8\pi N/\chi \leq k \leq N\}$.
A particle
of mass $k/N$ with $k\geq 2$ has to be seen as a {\it cluster} of $k$ elementary particles.
The drift term is thus easily understood: a single elementary particle interacts with the other ones
proportionaly to $1/N$, so that a cluster consisting of $k$ elementary particles 
interacts with the other ones proportionally to its mass $k/N$.
The diffusion coefficients are also quite natural: a single particle being subjected to a Brownian
excitation with coefficient $\sqrt 2$, a cluster with mass $k/N$ is excited by the mean
of $k$ Brownian motions with coefficient $\sqrt 2$, that is, by a Brownian motion with coefficient 
$\sqrt {2 /k}$.

\vip

If $N_t\geq 2$, setting for
$I\subset\{1,\hdots,N_t\}$ with cardinality $|I|\geq 2$,
$\bar{X}^I_t=\sum_{i\in I}\nu^i_tX^{i,N}_t/\sum_{i\in I}\nu^i_t$ and 
$R^I_t=(N/2)\sum_{i\in I} \nu^i_t|X^{i,N}_t-\bar{X}^I_t|^2$,
a simple computation shows that, 
when neglecting the interaction with particles with label outside $I$, $R^I_t$ behaves like a squared Bessel
process of dimension $2(|I|-1)- (\chi N/4\pi )\sum_{i,j\in I,i\ne j}\nu^i_t\nu^j_t\leq (|I|-1)[2-S\chi/(4\pi)]$, which
is nonpositive as soon as $S=\sum_{i\in I} \nu^i_t \geq 8\pi/\chi$.

\vip

Let us mention that once a {\it cluster} is formed,
its mass necessarily exceeds $8\pi/\chi$, so that any collision involving a cluster
will be sticky. 

\vip

The existence of such a process is not clear. Sticky collisions should not be very hard to treat.
The main difficulty is to control reflecting collisions. As explained just above, reflecting collisions
only concern particles with masses $1/N$, so that the classification given in Subsection \ref{sss}
should still be relevant. Thus we believe that the main difficulty is to
build a (necessarily nontrivial) local (in time)
solution to \eqref{ps} when $\chi \geq 8\pi$ and starting from an initial condition
where $k$ particles have the same initial positions, for some $k\in\{2,\dots,\lceil 8\pi N/\chi\rceil-1\}$.

\subsection{Comments}
Observe that this process is different of the one introduced by Ha\v{s}kovec and Schmeiser in \cite{hs}
where they consider a system of particles with different masses to approximate the singular solution to 
the Keller-Segel equation. In fact, rather than considering like us the limit $\e\to 0$ of the 
regularized particle system \eqref{psee}, they first prove propagation of chaos as $N\to \infty$ 
for a fixed $\e>0$ in \cite{hs2}. More precisely, they
check that for fixed $k\geq 1$, the density of $(X^{1,N,\e}_t,\hdots,X^{k,N,\e}_t)$ solving  \eqref{psee} 
(with another regularized kernel $K_\e$) converges as $N\to\infty$ to 
$\prod_{i=1}^kf^\e_t(x_i)$ where $(f^\e_t)_{t\geq 0}$ solves the regularized Keller-Segel partial differential equation
\begin{equation*}
\partial_t f^\e_t(x) + \chi \diver_x ((K_\e\star f^\e_t)(x) f^\e_t(x))  = \Delta_x f^\e_t(x).
\end{equation*}
The limiting behaviour of $(f^\e_t)_{t\geq 0}$ as $\e\to 0$ was studied in \cite{ds} and involves a defect measure. 
Then Ha\v{s}kovec and Schmeiser 
introduce in \cite{hs} a particle system associated with this limit, in which there are
heavy particles that occupy a positive proportion of the mass, interact with the other particles,
but do not undergo any Brownian excitation.


\begin{thebibliography}{99}
\bibitem{bdp}{{\sc A. Blanchet, J. Dolbeault, B. Perthame}, Two-dimensional Keller-Segel model: 
optimal critical mass and qualitative properties of the solutions, {\it Electron. J. Differential Equations} 
{\bf 44} (2006), 32 pp.}

\bibitem{bt}{{\sc M. Bossy, D. Talay}, Convergence rate for the approximation of the limit law of weakly 
interacting particles: application to the Burgers equation, {\it Ann. Appl. Probab.} {\bf 6} (1996), 818--861.}

\bibitem{cl}{{\sc E. Cepa, D. Lepingle}, Brownian particles with electrostatic repulsion on the circle: 
Dyson’s model for unitary random matrices revisited, {\it ESAIM: Prob. and Stat.} {\bf  5} (2001), 203--224.}

\bibitem{ds}{{\sc J. Dolbeault, C. Schmeiser}, The two-dimensional Keller-Segel model after blow-up, 
{\it Discrete Contin. Dyn. Syst.} {\bf 25} (2009), 109--121.}

\bibitem{em}{{\sc G. Ega\~na, S. Mischler}, Uniqueness and long time asymptotic for the Keller-Segel 
equation: the parabolic-elliptic case, preprint, arXiv:1310.7771.}

\bibitem{f}{{\sc I. Fatkullin}, A study of blow-ups in the Keller-Segel model of chemotaxis,
{\it Nonlinearity} {\bf 26} (2013), 81--94.}

\bibitem{fh}{{\sc N. Fournier, M. Hauray}, Propagation of chaos for the Landau equation with moderately 
soft potentials, arXiv:1501.01802.}

\bibitem{fhm}{{\sc N. Fournier, M. Hauray, S. Mischler}, Propagation of chaos for the 2D viscous vortex model,
{\it J. Eur. Math. Soc.} {\bf 16} (2014), 1423--1466.}

\bibitem{gq}{{\sc D. Godinho, C. Quininao}, Propagation of chaos for a sub-critical Keller-Segel model,
to appear in {\it Ann. Inst. Henri Poincaré Probab. Stat.}, arXiv:1306.3831.}

\bibitem{hs}{{\sc J. Ha\v{s}kovec, C. Schmeiser}, Stochastic particle approximation for measure valued 
solutions of the 2D Keller-Segel system, {\it J. Stat. Phys.} {\bf 135}, 133--151.}

\bibitem{hs2}{{\sc J. Ha\v{s}kovec, C. Schmeiser}, Convergence of a stochastic particle approximation 
for measure solutions of the 2D Keller-Segel system, {\it Comm. Partial Differential Equations} 
{\bf 36(6)} (2011), 940--960.}

\bibitem{hj}{{\sc M. Hauray, P.E. Jabin}, Particle approximation of Vlasov equations with singular forces: 
Propagation of chaos. To appear in {\it Ann. ENS.}, arXiv:1107.3821.}

\bibitem{hv}{{\sc M. A. Herrero, J.J.L. Velazquez}, Singularity patterns in a chemotaxis model, {\it Math.
Ann.} {\bf 306} (1996), 583--623.}

\bibitem{h1}{{\sc D. Horstmann}, From 1970 until present: the Keller-Segel model in chemotaxis and its 
consequences I, {\it Jahresber. Deutsch. Math. Verein.} {\bf 105} (2003), 103--165.}

\bibitem{h2}{{\sc D. Horstmann}, From 1970 until present: the Keller-Segel model in chemotaxis and its 
consequences II, {\it Jahresber. Deutsch. Math. Verein.} {\bf 106} (2004), 51--69.}

\bibitem{jl}{{\sc W. J\"ager, S. Luckhaus}, On explosions of solutions to a system of partial differential
equations modelling chemotaxis, {\it Trans. Amer. Math. Soc.} {\bf 329} (1992), 819--824.}

\bibitem{j}{{\sc B. Jourdain},  Diffusion processes associated with nonlinear evolution equations for
 signed measures, {\it Methodol. Comput. Appl. Probab.}  {\bf 2},  69--91.}

\bibitem{jr}{{\sc B. Jourdain, J. Reygner}, A multitype sticky particle construction of Wasserstein stable 
semigroups solving one-dimensional diagonal hyperbolic systems with large monotonic data, arXiv:1501.01498.}

\bibitem{k}{{\sc M. Kac},  Foundations of kinetic theory. In Proceedings of the Third Berkeley 
Symposium on Mathematical Statistics and Probability, 1954–1955, vol. III (Berkeley and Los Angeles, 1956), 
University of California Press, 171--197.}

\bibitem{kashre}{{\sc I. Karatzas, S. Shreve}, Brownian motion and Stochastic Calculus. 
Second Edition. Springer, 1991.}

\bibitem{ks}{{\sc }, Initiation of slime mold aggregation viewed as
an instability, {\it J. Theor. Biol.} {\bf 26} (1970), 399--415.}

\bibitem{kosh}{{\sc D. Khoshnevisan}, Exact rates of convergence to Brownian local times, 
{\it Ann. Probab.} {\bf 22} (1994), 1295--1330.}

\bibitem{mp}{{\sc C. Marchioro, M. Pulvirenti}, Hydrodynamics in two dimensions and vortex theory,
{\it  Comm. Math. Phys.} {\bf 84}, (1982), 483--503.}

\bibitem{mc}{{\sc H.P. McKean Jr.}, Propagation of chaos for a class of non-linear parabolic equations. 
In Stochastic Differential Equations (Lecture Series in Differential Equations, Session 7, Catholic Univ., 1967). 
Air Force Office Sci. Res., Arlington, Va., 1967, 41--57.}

\bibitem{m}{{\sc S. M\'el\'eard}, Asymptotic behaviour of some interacting particle systems; McKean-Vlasov and 
Boltzmann models. In Probabilistic models for nonlinear partial differential equations 
(Montecatini Terme, 1995), vol. 1627 of Lecture Notes in Math. Springer, Berlin, 1996, 42--95.}.

\bibitem{mm}{{\sc S. Mischler, C. Mouhot}, Kac’s Program in Kinetic Theory, {\it  Invent. Math.} {\bf 193}
(2013), 1--147.}

\bibitem{o0}{{\sc H. Osada}, A stochastic differential equation arising from the vortex problem,
{\it Proc. Japan Acad. Ser. A Math. Sci. 61} {\bf 10} (1986), 333--336.}

\bibitem{o1}{{\sc H. Osada}, Propagation of chaos for the two-dimensional {N}avier-{S}tokes equation,
{\it Proc. Japan Acad. Ser. A Math. Sci.} {\bf 62} (1986), 8--11.}

\bibitem{o}{{\sc H. Osada}, Propagation of chaos for the two-dimensional Navier-Stokes
 equation, {\it Probabilistic methods in mathematical physics (Katata/Kyoto)} (1987), Academic Press.}

\bibitem{p}{{\sc C.S. Patlak}, Random walk with persistence and external bias, {\it Bull.
Math. Biophys.} {\bf 15} (1953), 311--338.}

\bibitem{pe}{{\sc B. Perthame}, PDE models for chemotactic movements: parabolic, hyperbolic and kinetic, 
{\it Appl. Math.} {\bf 49} (2004), 539--564.}

\bibitem{reyor}{{\sc D. Revuz, M. Yor}, Continuous martingales and Brownian motion.
Third Edition. Springer, 2005.}

\bibitem{sk}{{\sc A.V. Skorokhod}, Stochastic equations for diffusion processes in a bounded region, 
{\it Theory of Probability and its Applications} {\bf 6} (1961), 264--274.}

\bibitem{st}{{\sc A. Stevens}, The derivation of chemotaxis equations as limit dynamics of moderately 
interacting stochastic many-particle systems, {\it SIAM J. Appl. Math.} {\bf 61} (2000), 183--212.}

\bibitem{s}{{\sc A.S. Sznitman}, Topics in propagation of chaos. In {\it \'Ecole d'\'Et\'e de Probabilit\'es 
de Saint-Flour XIX-1989}, vol. 1464 of Lecture Notes in Math. Springer, 1991, 165--251.}

\bibitem{v1}{{\sc J.J.L. Velazquez}, Point dynamics in a singular limit of the Keller-Segel model. I. Motion of 
the concentration regions, {\it SIAM J. Appl. Math.}, {\bf 64} (2004), 1198--1223.}

\bibitem{v2}{{\sc J.J.L. Velazquez}, Point dynamics in a singular limit of the Keller-Segel model. II. Formation
of the concentration regions, {\it SIAM J. Appl. Math.}, {\bf 64} (2004), 1224--1248.}


\end{thebibliography}
\end{document}